\documentclass[12pt,reqno]{amsart}
\usepackage{latexsym}
\usepackage{amssymb}
\usepackage{amscd}

\numberwithin{equation}{section}

\newtheorem{theorem}{Theorem}
\newtheorem{proposition}[theorem]{Proposition}
\newtheorem{lemma}[theorem]{Lemma}
\newtheorem{corollary}[theorem]{Corollary}
\newtheorem{conjecture}[theorem]{Conjecture}

\theoremstyle{remark}

\newtheorem{remarknu}[theorem]{Remark}
\newtheorem*{note}{Note}

\def\fl#1{\left\lfloor#1\right\rfloor}

\def\coef#1{\left\langle#1\right\rangle}
\def\Cat{\operatorname{Cat}}
\def\Hom{\operatorname{Hom}}

\def\tS{{\tilde S}}
\def\FF{F}
\def\GG{G}
\def\al{\alpha}
\def\be{\beta}
\def\de{\delta}
\def\si{\sigma}
\def\ta{\tau}

\def\ga{\gamma}
\def\om{\omega}
\def\Ga{\Gamma}
\def\la{\lambda}

\newcommand{\Z}{\mathbb{Z}}

\begin{document}
\title[Mod-$2^k$ behaviour of recursive sequences]{A 
method for determining the mod-$2^k$ behaviour of
recursive sequences, with applications to subgroup counting}
\author[M. Kauers, C. Krattenthaler, and 
T.\,W. M\"uller]{M. Kauers$^\dagger$, C. Krattenthaler$^{\ddagger}$, and
T. W. M\"uller$^*$} 

\address{$^\dagger$Research Institute for Symbolic Computation, 
Johannes Kepler Universit\"at, Altenbergerstrasze 69,
A-4040 Linz, Austria.
WWW: {\tt http://www.kauers.de}}

\address{$^{\ddagger*}$Fakult\"at f\"ur Mathematik, Universit\"at Wien,
Nordbergstrasze~15, A-1090 Vienna, Austria.
WWW: {\tt http://www.mat.univie.ac.at/\lower0.5ex\hbox{\~{}}kratt}.}

\address{$^*$School of Mathematical Sciences, Queen Mary
\& Westfield College, University of London,
Mile End Road, London E1 4NS, United Kingdom.%\newline
%WWW: \tt http://www.maths.qmw.ac.uk/\~{}twm/.
}

\thanks{$^\dagger$Research supported by the Austrian
Science Foundation FWF, grant Y464-N18\newline\indent
$^\ddagger$Research partially supported by the Austrian
Science Foundation FWF, grants Z130-N13 and S9607-N13,
the latter in the framework of the National Research Network
``Analytic Combinatorics and Probabilistic Number
Theory"\newline\indent
$^*$Research supported by Lise Meitner Grant M1201-N13 of the Austrian
Science Foundation FWF}

\dedicatory{Dedicated to Doron Zeilberger}



\begin{abstract}
We present a method to obtain congruences modulo powers of $2$ for
sequences given by recurrences of finite 
depth with polynomial coefficients. 
We apply this method to 
Catalan numbers, Fu\ss--Catalan numbers, and to subgroup counting
functions associated with Hecke groups and their lifts.
This leads to numerous new results, including many extensions
of known results to higher powers of $2$. 
\end{abstract}
\maketitle

\section{Introduction}

Ever since the work of Sylow \cite{SyloAA}, 
Frobenius \cite{Frob1,Frob2}, and P.~Hall 
\cite{PHall2}, the study of congruences for subgroup numbers and 
related numerical quantities of groups has played an important role 
in group theory.

Divisibility properties of subgroup numbers of 
(finitely generated) infinite groups may to some extent be
viewed as some kind of
analogue to these classical results for finite groups. 
To the best of our knowledge, the first significant result in this
direction was obtained by Stothers \cite{Stothers}: {\it the number
of index-$n$-subgroups in the inhomogeneous modular group
$PSL_2(\mathbb Z)$ is odd if, and only if, $n$ is of the form
$2^k-3$ or $2^{k+1}-6,$ for some positive integer $k\ge2$.}
A different proof of this result was given by Godsil, Imrich, and
Razen \cite{GIR}. 

The systematic study of divisibility properties of
subgroup counting functions for infinite groups begins with
\cite{MuHecke}. There, the parity of subgroup numbers and the number
of free subgroups of given finite index are determined for arbitrary
Hecke groups $\mathfrak H(q)=C_2*C_q$ with $q\ge3$. 
Subsequently, the results of \cite{MuHecke} were generalised to
larger classes of groups and arbitrary prime modulus in 
\cite{CaMuAA,KrMuAC,ModSub2,ModSub,MuPu2}. 
A first attempt at obtaining congruences modulo higher prime powers
was made in \cite{MuPu}, where the behaviour of subgroup numbers in
$PSL_2(\mathbb Z)\cong \mathfrak H(3)$ 
is investigated modulo $8$ and a congruence modulo
$16$ is derived for the number of free subgroups of given index in
$PSL_2(\mathbb Z)$.

A common feature of all the above listed sequences of subgroup
numbers is that they obey recurrences of finite depth with polynomial
coefficients. The purpose of this paper is to present a new
method for determining congruences modulo arbitrarily large powers of
$2$ for sequences described by such recurrences. Our method is
inspired by the observation that many of the aforementioned results 
say in essence that the generating function for the subgroup numbers
under consideration, when reduced modulo a $2$-power, 
can be expressed as a polynomial in the basic series
\begin{equation} \label{eq:Phidef}
\Phi(z)=\sum _{n\ge0} ^{}z^{2^n}
\end{equation}
with coefficients that are Laurent polynomials in $z$. 
What our method affords is an algorithmic procedure to find such
polynomial expressions, provided they exist. 
By applying our method to Catalan numbers, to (certain) Fu\ss--Catalan
numbers, and to various subgroup counting problems in Hecke groups
and their lifts, we obtain far-reaching generalisations and extensions
of the previously mentioned results.
In order to give a concrete illustration, the recent result 
\cite[Theorems~6.1--6.6]{LiYeAA} of Liu and Yeh determining the
behaviour of Catalan numbers $\Cat_n$ modulo $64$ 
can be compactly written in the form
\begin{multline} \label{eq:Cat64}
\sum _{n=0} ^{\infty}\Cat_nz^n=
32 z^5+16 z^4+6
   z^2+13 z+1
+\left(32 z^4+32
   z^3+20 z^2+44 z+40\right)
   \Phi(z)\\
+\left(16 z^3+56
   z^2+30
   z+52+\frac{12}{z}\right)
   \Phi^2(z)
+\left(32 z^3+60
   z+60+\frac{28}{z}\right)
   \Phi^3(z)\\
+\left(32 z^3+16
   z^2+48
   z+18+\frac{35}{z}\right)
   \Phi^4(z)
+\left(32
   z^2+44\right)
   \Phi^5(z)\\
+\left(48
   z+8+\frac{50}{z}\right)
   \Phi^6(z)+\left(32
   z+32+\frac{4}{z}\right)
   \Phi^7(z)
\quad \quad \text { modulo }64,
\end{multline}
as may be seen by a straightforward (but
rather tedious) computation. Our method can not only {\it
find\/} this result, but it produces as well corresponding formulae
modulo {\it any} given power of $2$ in a completely automatic
fashion, see Theorems~\ref{thm:Cat} and \ref{thm:Cat4096} in
Section~\ref{sec:Cat}. 

In a sense, which is made precise in Section~\ref{sec:method}, 
our method is very
much in the spirit of Doron Zeilberger's philosophy that
{\it mathematicians should train computers to automatically produce
theorems}. Indeed, Theorems~\ref{thm:Cat}, \ref{thm:freie-m},
\ref{thm:Unterg}, \ref{thm:2mf}, \ref{thm:H5}
imply that our algorithm is able to
produce a theorem on the behaviour modulo {\it any} given $2$-power of
the subgroup counting functions featuring in these theorems,
and, if fed with a concrete $2$-power, 
our implementation will diligently output the corresponding result 
(provided the input does not cause the available computer resources
to be exceeded \dots). 
Moreover, when discussing subgroup numbers of 
lifts of $PSL_2(\mathbb Z)$, 
(such as the homogeneous modular group $SL_2(\mathbb Z)$),
a crucial role is also played by an application
of the holonomic functions approach to finding recurrences for
multi-variate hypergeometric sums, pioneered by Wilf and Zeilberger 
\cite{WiZeAC,ZeilAN}, and further developed in \cite{ChyzAA,ChSaAA,KoutAA}.

The rest of this introduction is devoted to a more detailed description of the contents of
this paper.
In Section~\ref{sec:Phi} we discuss our main character, the formal
power series $\Phi(z)$ defined in \eqref{eq:Phidef}. 
While $\Phi(z)$ is transcendental over $\mathbb Q[z]$ (or,
equivalently, over $\Z[z]$), it is easy to see
that it is algebraic modulo powers of $2$. The focus in that section
is on polynomial identities for $\Phi(z)$ modulo a given $2$-power
which are of {\it minimal degree}.
Then, in
Section~\ref{sec:method}, we describe our method of expressing the
generating function of a recursive sequence, when reduced modulo a
given $2$-power, as a polynomial in $\Phi(z)$ with coefficients that
are Laurent polynomials in $z$. 
The method relies in an essential way on the polynomial identities from
Section~\ref{sec:Phi}. The problem how to extract the explicit value of 
a concrete coefficient in a polynomial expression in $\Phi(z)$ 
(such as \eqref{eq:Cat64}) modulo a given $2$-power 
is discussed in Section~\ref{sec:extr}, where we present an efficient
algorithm performing this task. This algorithm is of theoretical
value (minimal length relations between powers of $\Phi(z)$ such as 
the ones in Proposition~\ref{prop:minpol} are established by applying
this algorithm to the powers of $\Phi(z)$; see also
Appendix~\ref{appA}) as well as of practical significance, as is
demonstrated by the derivations of Theorems~\ref{thm:1} and
\ref{thm:Ga3}.

As a first illustration of our method, we apply it to Catalan numbers, 
thereby
significantly improving numerous earlier results in the literature;
see Section~\ref{sec:Cat}. This is contrasted in
Section~\ref{sec:nonex} with an example (concerning particular
Fu\ss--Catalan numbers) where our method is bound to
fail. The reason is spelled out in Theorem~\ref{thm:C2C7}, which, at
the same time, also gives a new description for the parity pattern of the
numbers of free subgroups of given index in the Hecke group $\mathfrak
H(7)$. 

The subsequent sections contain several applications
of our method to the problem of determining congruences modulo a
given $2$-power for numbers of subgroups of Hecke groups $\mathfrak{H}(q)$ and their lifts
\begin{equation} \label{eq:Heckelift}
\Gamma_m(q) = C_{2m} \underset{C_m}{\ast} C_{qm} = \big\langle
x,y\,\big\vert\, x^{2m} = y^{qm} = 1,\, x^2 = y^q\big\rangle,
\quad m\ge1. 
\end{equation}
Ubiquitous in these applications is --- explicitly or implicitly ---
the intimate relation between subgroup numbers of a group $\Ga$ and 
numbers of permutation representations of $\Ga$, in the form of
identities between the corresponding generating functions. This is
directly visible in the folklore result \eqref{eq:Dey} (which is not 
only used in Sections~\ref{sec:SL2Z}, \ref{sec:Ga3}, and \ref{sec:H5}, but
also lies 
behind the crucial differential equation \eqref{eq:S1diff} in
Section~\ref{sec:PSL2Z}; cf.\ its derivation in \cite{GIR}), and also
indirectly in Lemma~\ref{Lem:Gamma_m(q)FreeRec} via the
$A$-invariants $A_\mu(\mathfrak H(q))$, see \cite[Sec.~2.2]{MuHecke}.
In the cases relevant here, the numbers of permutation
representations of $\Ga$ satisfy linear recurrences with polynomial
coefficients --- to make these explicit may require the algorithmic
machinery around the ``holonomic paradigm"
(cf.\ \cite{ChyzAA,ChSaAA,KoutAA,MallAA,SaZiAA,WiZeAC,ZeilAN}),
see Sections~\ref{sec:SL2Z} and \ref{sec:Ga3} for corresponding
examples.
Via the aforementioned generating function relation, such a 
recurrence can be translated into a Riccati-type differential
equation for the generating function of the subgroup numbers that we
are interested in. It is here, where our method comes in: it is
tailor-made for being applied to formal power series $F(z)$ satisfying this
type of differential equation, and it affords an algorithmic
procedure to find a polynomial in $\Phi(z)$ which agrees, after
reduction of the coefficients of $F(z)$ modulo a given power
of $2$, with the power series $F(z)$.

We start in Section~\ref{sec:free} with free subgroup numbers
of lifts $\Ga_m(q)$, for primes $q\ge3$, where we prepare the ground for 
application of our method. More specifically, in
Proposition~\ref{Prop:flambdaq2Part}
we present a lower bound for the
$2$-adic valuation of the number of free subgroups of index $n$
in $\Ga_m(q)$, where $q$ is a
Fermat prime. In particular,
this result implies that the sequence of free subgroup numbers under
consideration is essentially zero modulo a given $2$-power in the
case when $m$ is even. 
In Section~\ref{sec:free2}, we show that our method provides an
algorithm for determining these numbers of free subgroups of $\Ga_m(3)$
modulo any given $2$-power in the case when $m$ is odd. 
The corresponding results (see Theorems~\ref{thm:freie-m} and
\ref{thm:freie-m64}) go far beyond the previous result
\cite[Theorem~1]{MuPu} on the behaviour of the number of free subgroups
of $PSL_2(\mathbb Z)$ modulo~$16$.
Our method provides as well an
algorithm for determining the number of {\it all\/} subgroups
of index $n$ in $PSL_2(\mathbb Z)$ modulo powers of $2$, as we demonstrate in
Section~\ref{sec:PSL2Z}. Not only are we able to provide a new proof
of Stothers' result \cite{Stothers} (which was stated in the second
paragraph above), but our method leads as well to refinements modulo
arbitrary powers of $2$ of Stothers' result and of the mod-$8$ results 
in \cite[Theorem~2]{MuPu} mentioned earlier; 
see Theorems~\ref{thm:Unterg} and \ref{thm:Untergr64}.
For the homogeneous modular group
$SL_2(\mathbb Z)$ (being isomorphic to the lift $\Ga_2(3)$)
and for the lift $\Ga_3(3)$, however, our method
from Section~\ref{sec:method} 
fails already for the modulus $8=2^3$. 
We overcome this obstacle by instead tuning our computations
with the target of obtaining results modulo~$16=2^4$. Indeed,
this leads to the determination of the number of subgroups of index~$n$
in $SL_2(\mathbb Z)$ and in $\Ga_3(3)$ modulo~8
(see Theorems~\ref{thm:Sl2Z8} and \ref{thm:Ga38}), but direct application
of our method does not produce corresponding results modulo~16.
Only by an {\it enhancement\/} of the method,
which we outline in Appendix~\ref{appD}, we are able 
to produce descriptions of the subgroup numbers of $SL_2(\mathbb Z)$
modulo~16, see Theorem~\ref{thm:Sl2Z16}. 
For the subgroup numbers of $\Ga_3(3)$ even this
enhancement fails, and this shows that the generating function for
these subgroup numbers, when coefficients are reduced modulo~16, {\it
cannot\/} be represented as a polynomial in $\Phi(z)$ with
coefficients that are Laurent polynomials in $z$. Still,
the results in Theorems~\ref{thm:Sl2Z8},
\ref{thm:Sl2Z16}, and \ref{thm:Ga38}
go significantly beyond the earlier parity results 
\cite[Eq.~(6.3) with $\vert H\vert=1$]{KrMuAC} for these groups.
This is explained in
Sections~\ref{sec:SL2Z} and \ref{sec:Ga3}, 
with Section~\ref{sec:lift} preparing the
ground by providing formulae for the number of permutation
representations of $SL_2(\mathbb Z)$ as well as other lifts of
$PSL_2(\mathbb Z)$. 
A further example where our method works for
any $2$-power is the subject of Section~\ref{sec:Hecke}: there we
apply the method to a functional equation 
(see \eqref{eq:2meq}) 
extending the functional equation for Catalan numbers (producing {\it
Fu\ss--Catalan numbers}), and show that it works for any given
$2$-power; see Theorem~\ref{thm:2mf}. 
In fact, we apply a variation of the method here,
in that the basic series $\Phi(z)$ gets replaced by a slightly different
series, which we denote by $\Phi_h(z)$ (see \eqref{eq:Phim}).
If Theorem~\ref{thm:2mf} is combined with results from \cite{MuHecke}, then
it turns out that our method provides as well an algorithm for
determining the number of free subgroups of index $n$ in a Hecke
group $\mathfrak H(q)$ and its lifts, 
where $q$ is a Fermat prime, modulo
any given $2$-power; see Corollary~\ref{thm:freeHecke}. 
The same assertion holds as well for the problem
of determining the number of subgroups of index $n$ in the Hecke
group $\mathfrak H(5)$, again modulo any given $2$-power
(see Theorems~\ref{thm:H5} and \ref{cor:H5} in Section~\ref{sec:H5}). 
We conjecture that the same is true
for Hecke groups $\mathfrak H(q)$, with $q$ a Fermat prime
(see Conjecture~\ref{conj:Hq}). The results of Sections~\ref{sec:Hecke} 
and \ref{sec:H5} discussed above largely generalise the
parity results \cite[Cor.~A', respectively Theorem~B]{MuHecke} 
for subgroup numbers of Hecke groups, although our results are not
independent, in the sense that we base our analyses on prior
results from \cite{MuHecke}.

\medskip
Concluding the introduction, we remark that there is no principal
obstacle to generalising our method to other basic series and
moduli.
For example, one may think of analysing the behaviour of recursive
sequences modulo powers of any prime $p$ in terms of
the obvious generalisation of $\Phi(z)$, i.e., the series
$\sum _{n\ge0} ^{}z^{p^n}$. It is in fact not difficult to see that
our results from Sections~\ref{sec:Cat} and \ref{sec:Hecke} for
Fu\ss--Catalan numbers characterised by the functional equation
\eqref{eq:2meq} for their generating function have rather
straightforward analogues for Fu\ss--Catalan numbers whose
generating function satisfies the functional equation
\begin{equation} \label{eq:pmeq}
zf^{p^h}(z)-f(z)+1=0.
\end{equation}
However we are not aware of any applications of this (or of variants) 
to congruence properties of subgroup numbers modulo powers of
primes $p$ different from $2$. In fact, the known results (cf.\
\cite{MuPu2} or \cite[Theorem~3]{MuPu}) 
strongly point to the fact that the phenomena that appear modulo
primes different from $2$ cannot be captured by series of the type
$\sum _{n\ge0} ^{}z^{p^n}$.
So, currently, we do not know of interesting
applications in this direction, but we hope to be able to return to this 
circle of ideas in future publications. 

\begin{note}
This paper is accompanied by several {\sl Mathematica} files 
and a {\sl Mathematica} notebook so that an interested reader is
able to redo (most of) the computations that are presented in
this article. Files and notebook are available at the article's
website
{\tt http://www.mat.univie.ac.at/\lower0.5ex\hbox{\~{}}kratt/artikel/modlifts.html}.
\end{note}

\section{The $2$-power series $\Phi(z)$}
\label{sec:Phi}

Here we consider the formal power series $\Phi(z)$ defined in
\eqref{eq:Phidef}.
This series is the principal character in the method for determining
congruences of recursive sequences modulo $2$-powers, 
which we describe in Section~\ref{sec:method}. 
It is well known that this series is transcendental over
$\mathbb Z[z]$ (this follows for instance from the density argument 
used in the proof of Lemma~\ref{lem:minpol} below). However, if the
coefficients of $\Phi(z)$ are considered modulo a $2$-power $2^\ga$,
then $\Phi(z)$ obeys a polynomial relation with coefficients that are
polynomials in $z$. The focus of this section is on what may be said 
concerning such polynomial relations, and, in particular, about those of
minimal length. 

\medskip
Here and in the sequel,
given power series (or Laurent series) $f(z)$ and $g(z)$,
we write 
$$f(z)=g(z)~\text {modulo}~2^\ga$$ 
to mean that the coefficients
of $z^i$ in $f(z)$ and $g(z)$ agree modulo~$2^\ga$ for all $i$.

We say that a polynomial $A(z,t)$ in 
$z$ and $t$ is {\it minimal for the modulus
$2^\ga$}, if it is monic (as a polynomial in $t$), 
has integral coefficients, satisfies
$A(z,\Phi(z))=0$~modulo~$2^\ga$, and there is no monic polynomial $B(z,t)$
with integral coefficients
of $t$-degree less than that of $A(z,t)$ with $B(z,\Phi(z))=0$~modulo~$2^\ga$.
(Minimal polynomials are not unique; see Remark~\ref{rem:min}.)
Furthermore, we let $v_2(\al)$ denote the $2$-adic valuation of the
integer $\al$, that is, the maximal exponent $e$ such that $2^e$
divides $\al$.

The lemma below provides a lower bound for the degree of a polynomial
that is minimal for the modulus $2^\ga$.

\begin{lemma} \label{lem:minpol}
If $A(z,t)$ is minimal for the modulus $2^\ga$, 
then the degree $d$ of $A(z,t)$ in $t$ satisfies $v_2(d!)\ge\ga$.
In particular, the series $\Phi(z)$ is transcendental over $\mathbb
Z[z]$.
\end{lemma}

\begin{proof} 
We introduce the following {\it density function} with respect to a
given modulus $2^\ga$ for a power series $f(z)$ in $z$:
\begin{equation} \label{eq:density}
D(f,2^\ga;n):=\big\vert\{e:2^{n-1}\le e<2^n\text { and }
\coef{z^e}f(z)\not\equiv 0~\text {modulo}~2^\ga\}\big\vert,
\quad \quad n=1,2,\dots,
\end{equation}
where $\coef{z^{e}}f(z)$
denotes the coefficient of $z^{e}$ in $f(z)$.
Setting 
\begin{equation} \label{eq:Em}
E_m(z):=\sum _{n_1>\dots>n_m\ge0}
^{}z^{2^{n_1}+2^{n_2}+\dots+2^{n_m}},
\end{equation}
simple counting yields that
\begin{equation} \label{eq:DEm}
D(E_m,2^\ga;n)= \binom {n-1}{m-1}\sim \frac {1} {(m-1)!}n^{m-1},
\quad \quad \text {as }n\to\infty.
\end{equation}
(The modulus $2^\ga$ does not play a role here.)
Furthermore, by considering the binary representations of possible
exponents $e$ such that the coefficient of $z^e$ in $\Phi^m(z)$ does
not vanish, we have
\begin{equation} \label{eq:DPhi}
D(\Phi^m,2^\ga;n)=\mathcal O\left(n^{m-1}\right),
\quad \quad \text {as }n\to\infty.
\end{equation}
Indeed, by direct expansion, we see that
\begin{equation} \label{eq:PhiE}
\Phi^m(z)=m!\,E_m(z)+R_m(z),
\end{equation}
where all monomials $z^e$ which appear with non-vanishing coefficient
in $R_m(z)$ have a binary expansion with at most $m-1$ digits $1$.
Consequently, again by elementary counting, we have
$$D(R_m,2^\ga;n)\le \binom {n-1}{m-2}+\binom {n-1}{m-3}+\dots+\binom
{n-1}0,$$
and hence $D(R_m,2^\ga;n)=\mathcal O(n^{m-2})$ as $n\to\infty$. 

Let us, by way of contradiction, suppose that the degree $d$ of $A(z,t)$
satisfies $v_2(d!)<\ga$.
Considering \eqref{eq:PhiE} with $m=d$, we see that $E_d(z)$ appears
in $\Phi^d(z)$ with a non-zero coefficient modulo $2^\ga$. 
Furthermore, the other terms in $\Phi^d(z)$ (denoted by $R_d(z)$ in
\eqref{eq:PhiE}) have a density function modulo $2^\ga$ 
which is asymptotically
strictly smaller than the density function of $E_d(z)$.
Consequently, if we remember \eqref{eq:DEm}, we have
$$
D(\Phi^d,2^\ga;n)\sim \frac {d!} {(d-1)!}n^{d-1}=dn^{d-1},
\quad \quad \text {as }n\to\infty.
$$
Moreover, by \eqref{eq:DPhi}, all powers $\Phi^m(z)$ with $m<d$ have
a density function modulo $2^\ga$ which is 
asymptotically strictly smaller than $n^{d-1}$. Altogether, it is 
impossible that a linear combination of powers $\Phi^m(d)$,
$m=0,1,\dots,d$, with coefficients that are polynomials in $z$ sums
up to zero modulo $2^\ga$, a contradiction to our assumption that $d$ is
the degree of a minimal polynomial $A(z,t)$. The particular statement 
is an immediate consequence of the inequality just proven.
\end{proof}

\begin{proposition} \label{prop:minpol}
Minimal polynomials for the moduli
$2,4,8,16,32,64,128$ are
\begin{alignat*}2 %\label{eq:minpol}
&t^2+t+z&&\text {\em modulo 2},\\
&
(t^2+t+z)^2
&&\text {\em modulo 4},\\
&
t^4+6t^3+(2z+3)t^2+(2z+6)t+2z+5z^2
&&\text {\em modulo 8},\\
&
(t^2+t+z)(t^4+6t^3+(2z+3)t^2+(2z+6)t+2z+5z^2)\quad 
&&\text {\em modulo 16},\\
&
(t^4+6t^3+(2z+3)t^2+(2z+6)t+2z+5z^2)^2
&&\text {\em modulo 32},\\
&
(t^4+6t^3+(2z+3)t^2+(2z+6)t+2z+5z^2)^2
&&\text {\em modulo 64},\\
&
t^8+124
   t^7+t^6 (68
   z+18)+t^5 (124
   z+24)+t^4 \left(62
   z^2+64
   z+81\right)&&\\
&+t^3
   \left(20 z^2+76
   z+28\right)+t^2
   \left(116 z^3+114 z^2+12
   z+92\right)\\
&+t
   \left(116 z^3+28 z^2+8
   z+16\right)+9 z^4+124
   z^3+12 z^2+112 z
&&\text {\em modulo 128}
.
\end{alignat*}
\end{proposition}

\begin{proof}
In order to be consistent with Section~\ref{sec:extr}, let us change
notation and write
$$
H_{1,1,\dots,1}(z):=\sum _{n_1>\dots>n_m\ge0}
^{}z^{2^{n_1}+2^{n_2}+\dots+2^{n_m}}
$$
(with $m$ occurrences of $1$ in $H_{1,1,\dots,1}(z)$).
Note that the above series is identical with the series which we
earlier denoted by $E_m(z)$.
Straightforward calculations yield that
\begin{align} 
\label{eq:Phi2}
\Phi^2(z)&=\Phi(z)+2H_{1,1}(z)-z,\\
\label{eq:Phi3}
\Phi^3(z)&=-2\sum _{n\ge0} ^{}z^{3\cdot 2^n}+
3(1-z)\Phi(z)+6H_{1,1}(z)+6H_{1,1,1}(z)-3z,\\
\notag
\Phi^4(z)&=-12\sum _{n\ge0} ^{}z^{3\cdot 2^n}
-8\sum _{n_1>n_2\ge0} ^{}z^{3\cdot 2^{n_1}+2^{n_2}}
-8\sum _{n_1>n_2\ge0} ^{}z^{2^{n_1}+3\cdot 2^{n_2}}
+(13-18z)\Phi(z)\\
\label{eq:Phi4}
&\kern2cm
+(30-12z)H_{1,1}(z)+36H_{1,1,1}(z)+24H_{1,1,1,1}(z)+5z^2-13z.
\end{align}
In particular, relation~\eqref{eq:Phi2}, together with
Lemma~\ref{lem:minpol}, immediately implies the
claims about minimal polynomials for the moduli $2$ and $4$.
Moreover, a simple computation using
\eqref{eq:Phi2}--\eqref{eq:Phi4} shows that
\begin{equation} \label{eq:rel4}
\Phi^4(z)+6\Phi^3(z)+(2z+3)\Phi^2(z)+(2z+6)\Phi(z)+2z+5z^2=0
\quad \text {modulo }8.
\end{equation}
Together with Lemma~\ref{lem:minpol},
this establishes the claims about minimal polynomials for the moduli
$8$, $16$, $32$, and $64$. In order to prove the claim
for the modulus $128$, one uses the expressions for $\Phi^i(z)$,
$i=2,3,\dots,8$, given above and in Appendix~\ref{appA}.
\end{proof}

\begin{remarknu}\label{rem:min}
Minimal polynomials are highly non-unique: for example, 
the polynomial
$$
(t^2+t+z)^2+2(t^2+t+z)
$$
is obviously also a minimal polynomial for the modulus $4$.
\end{remarknu}

Based on the observations in Proposition~\ref{prop:minpol}
and Lemma~\ref{lem:minpol}, we propose the following conjecture.

\begin{conjecture} \label{conj:1}
The degree of a minimal polynomial for the modulus $2^\ga$, $\ga\ge1$,
is the least $d$ such that $v_2(d!)\ge\ga$.
\end{conjecture} 

\begin{remarknu} \label{rem:1}
(1) Given the binary expansion of $d$, say
$$d=d_0+d_1\cdot 2+d_2\cdot
4+\cdots+d_r\cdot 2^r, \quad 0\le \ga_i\le1,$$ 
by the well-known formula of
Legendre \cite[p.~10]{LegeAA}, we have
\begin{align}\notag
v_2(d!)&=\sum _{\ell=1} ^{\infty}\fl{\frac {d} {2^\ell}}=
\sum _{\ell=1} ^{\infty}\fl{\sum _{i=0} ^{r}d_i 2^{i-\ell}}=
\sum _{\ell=1} ^{\infty}\sum _{i=\ell} ^{r}d_i 2^{i-\ell}\\
\label{eq:Leg}
&=
\sum _{i=1} ^{r}\sum _{\ell=1} ^{i}d_i 2^{i-\ell}=
\sum _{i=1} ^{r}d_i\left(2^{i}-1\right)=
d-s(d),
\end{align}
where $s(d)$ denotes the sum of digits of $d$ in its binary
expansion.
Consequently, an equivalent way of phrasing Conjecture~\ref{conj:1}
is to say that the degree of a minimal polynomial for the 
modulus $2^\ga$ is the least $d$ with $d-s(d)\ge\ga$.

\medskip
(2)
We claim that, in order to establish Conjecture~\ref{conj:1}, 
it suffices to prove the conjecture for $\ga=2^{\de}-1$, 
$\de=1,2,\dots$. If we take into account
Lemma~\ref{lem:minpol} plus the above remark, this means that it is
sufficient to prove that, for each $\de\ge1$, there is
a polynomial $A_\de(z,t)$ of degree $2^{\de}$ such that
\begin{equation} \label{eq:min2ga}
A_\de(z,\Phi(z))=0\quad \text {modulo }2^{2^\de-1}.
\end{equation}
For, arguing by induction, let us suppose that 
we have already constructed $A_1(z,t),\break
A_2(z,t),
\dots,A_m(z,t)$ satisfying \eqref{eq:min2ga}.
Let
$$\al=\al_1\cdot 2+\al_2\cdot
4+\cdots+\al_m\cdot 2^m, \quad 0\le \al_i\le1,$$ 
be the binary expansion of the even positive integer
$\al$. In this situation, we have
\begin{equation} \label{eq:betaind}
\prod _{\de=1} ^{m}A_{\de}^{\al_\de}(z,\Phi(z))=0\quad \text {modulo }
\prod _{\de=1} ^{m}2^{\al_\de(2^\de-1)}=
\prod _{\de=0} ^{m}2^{\al_\de(2^\de-1)}=
2^{\al-s(\al)}.
\end{equation}
On the other hand, the degree of the left-hand side of
\eqref{eq:betaind} as a polynomial in $\Phi(z)$ is
$\sum _{\de=1} ^{m}\al_\de 2^\de=\al$.

Let us put these observations together.
In view of \eqref{eq:Leg},
Lemma~\ref{lem:minpol} says that the degree of a minimal
polynomial for the modulus $2^\ga$ cannot be smaller than the
least integer, $d^{(\ga)}$ say, for which $d^{(\ga)}-s(d^{(\ga)})\ge\ga$. 
(We remark that $d^{(\ga)}$
must be automatically even.) If we take into account that
the quantity $\al-s(\al)$, as a function in $\al$, is weakly monotone
increasing in $\al$, then \eqref{eq:betaind} tells us that,
as long as $d^{(\ga)}\le 2+4+\dots+2^m=2^{m+1}-2$, we have found a monic
polynomial of degree $d^{(\ga)}$, $B_\ga(z,t)$ say, for which
$B_\ga(z,\Phi(z))=0$ modulo~$2^\ga$, namely the left-hand side
of \eqref{eq:betaind} with $\al$ replaced by $d^{(\ga)}$, to wit
$$
B_\ga(z,t)=\prod _{\de=1} ^{m}A_{\de}^{d^{(\ga)}_\de}(z,t),
$$
where $d^{(\ga)}=d^{(\ga)}_1\cdot 2+d^{(\ga)}_2\cdot
4+\cdots+d^{(\ga)}_m\cdot 2^m$ is the binary expansion of 
$d^{(\ga)}$. Hence, it must necessarily be a minimal polynomial
for the modulus $2^\ga$.

Since $(2^{m+1}-2)-s(2^{m+1}-2)=2^{m+1}-2-m$, we have thus found minimal
polynomials for all moduli $2^\ga$ with $\ga\le 2^{m+1}-m-2$. Now we should note that the
quantity $\al-s(\al)$ makes a jump from $2^{m+1}-m-2$ to $2^{m+1}-1$
when we move from $\al=2^{m+1}-2$ to $\al=2^{m+1}$ (the reader should 
recall that it suffices to consider even $\al$). 
If we take $A_m^2(z,t)$, which has degree $2\cdot 2^m=2^{m+1}$, 
then, by \eqref{eq:min2ga}, 
we also have a minimal polynomial for the modulus 
$\left(2^{2^m-1}\right)^2=2^{2^{m+1}-2}$
and, in view of the preceding remark, as well for all moduli $2^\ga$
with $\ga$ between $2^{m+1}-m-1$ and $2^{m+1}-2$. 

So, indeed, the first modulus for which we do not have a minimal
polynomial is the modulus $2^{2^{m+1}-1}$. This is the role
which $A_{m+1}(z,t)$ (see \eqref{eq:min2ga} with $m+1$ in place of 
$\de$) would have to play.

\smallskip
The arguments above show at the same time that, supposing that we have
already constructed $A_1(z,t),A_2(z,t),\dots,A_m(z,t)$, the polynomial 
$A^2_m(z,t)$ is a very close ``approximation" to the polynomial
$A_{m+1}(z,t)$ that we are actually looking for next,
which is only ``off" by a factor of $2$.
In practice, one can recursively compute polynomials $A_\de(z,t)$
satisfying \eqref{eq:min2ga} by following the procedure outlined in
the next-to-last paragraph before Lemma~\ref{lem:aiodd} in the next
section. It is these computations (part of which are reported in
Proposition~\ref{prop:minpol}) which have led us to believe in the truth of 
Conjecture~\ref{conj:1}.
\end{remarknu}

\section{Coefficient extraction from powers of $\Phi(z)$}
\label{sec:extr}

In the next section we are going to describe a method for expressing
formal power series satisfying certain differential equations, after
the coefficients of the series have been reduced modulo $2^k$, 
as polynomials in the
$2$-power series $\Phi(z)$ (which has been discussed in the previous
section; for the definition see \eqref{eq:Phidef}), the
coefficients being Laurent polynomials in $z$. Such a method would be
without value if we could not, at the same time, provide a procedure
for extracting coefficients from powers of $\Phi(z)$. The description
of such a procedure is the topic of this section.

Clearly, a brute force expansion of a power $\Phi^K(z)$, where $K$ is
a given positive integer, yields
\begin{equation} \label{eq:Phipot}
\Phi^K(z)=\sum _{r=1} ^{K}
\underset{a_1+\dots+a_r=K}{\sum _{a_1,\dots,a_r\ge1} ^{}}\frac {K!}
{a_1!\,a_2!\cdots a_r!}H_{a_1,a_2,\dots,a_r}(z),
\end{equation}
where
$$
H_{a_1,a_2,\dots,a_r}(z):=\sum_{n_1>n_2>\dots>n_r\ge0}
z^{a_12^{n_1}+a_22^{n_2}+\dots+a_r2^{n_r}}.
$$
The expansion \eqref{eq:Phipot} is not (yet) suited for our purpose,
since, when $a_1,a_2,\dots,a_r$ vary over all possible choices such
that their sum is $K$, the series $H_{a_1,a_2,\dots,a_r}(z)$ are {\it not\/}
linearly independent over the ring $\Z[z,z^{-1}]$ 
of Laurent polynomials in $z$ over the integers\footnote{The
same is true for an arbitrary ring in place of the ring $\Z$ of
integers.},
and, second, coefficient extraction from a series
$H_{a_1,a_2,\dots,a_r}(z)$ can be a hairy task if some of the $a_i$'s 
are even.

However, we shall show (see Corollary~\ref{lem:Hind}) that, 
if we restrict to {\it odd\/}
$a_i$'s, then the corresponding series $H_{a_1,a_2,\dots,a_r}(z)$,
together with the (trivial) series $1$, 
{\it are} linearly independent over $\Z[z,z^{-1}]$, and there is an
efficient algorithm to express all other
series $H_{b_1,b_2,\dots,b_s}(z)$, where we do {\it not\/} make any
restriction on the $b_i$'s, as a linear combination over
$\Z[z,z^{-1}]$ of $1$ and the former series (see Lemma~\ref{lem:Hbi}). 
Since coefficient
extraction from a series $H_{a_1,a_2,\dots,a_r}(z)$ with all $a_i$'s
odd is straightforward (see Remark~\ref{rem:eff}), this solves the problem of
coefficient extraction from powers of $\Phi(z)$.

As a side result, the procedure which we described in the previous
paragraph, and which will be substantiated below, provides all the
means for determining minimal polynomials in the sense of
Section~\ref{sec:Phi}: as explained in Item~(2) of Remark~\ref{rem:1} 
at the end of that section, it suffices to find a minimal polynomial
for the modulus $2^{2^\de-1}$, $\de=1,2,\dots$. For doing this, we 
would take a minimal polynomial $A_{\de-1}(z,t)$ for the modulus
$2^{2^{\de-1}-1}$, expand the square $A^2_{\de-1}(z,t)$,
and replace each coefficient $c_{\al,\be}$ of a monomial
$z^\alpha t^\beta$ in $A^2_{\de-1}(z,t)$ by 
$c_{\al,\be}+2^{2^{\de}-2}x_{\al,\be}$, 
where $x_{\al,\be}$ is a variable, thereby obtaining a modified polynomial, 
$B_{\de-1}(z,t)$ say.
Now we would substitute $\Phi(z)$
for $t$, so that we obtain $B_{\de-1}(z,\Phi(z))$.
Here, we express powers of $\Phi(z)$
in terms of the series $H_{a_1,a_2,\dots,a_r}(z)$ with all $a_i$'s
being odd, and collect terms.
By reading the coefficients of $z^\ga H_{a_1,a_2,\dots,a_r}(z)$ in this
expansion of $B_{\de-1}(z,\Phi(z))$ and equating them to zero 
modulo $2^{2^\de-1}$,
we produce a system of linear equations
modulo $2^{2^\de-1}$ in the unknowns $x_{\al,\be}$. 
By the definition of $A_{\de-1}(z,t)$, after division by $2^{2^\de-2}$,
this system reduces to a system
modulo $2$, that is, to a linear system of equations over the field
with two elements. A priori, this system need not have a solution,
but experience seems to indicate that it always does; see
Conjecture~\ref{conj:1}.

\medskip
We start with an auxiliary result pertaining to the uniqueness of
representations of integers as sums of powers of $2$ with
multiplicities, tailor-made for application to the series
$H_{a_1,a_2,\dots,a_r}(z)$.

\begin{lemma} \label{lem:aiodd}
Let $d,r,s$ be positive integers with $r\ge s,$ $c$ an integer
with $\vert c\vert\le d,$ and let 
$a_1,a_2,\dots,a_r$ respectively $b_1,b_2,\dots,b_s$ be two sequences of 
odd integers, with $1\le a_i\le d$ for $1\le i\le r,$ and $1\le b_i\le
d$ for $1\le i\le s$. If
\begin{equation} \label{eq:a2b2}
a_12^{2rd}+a_22^{2(r-1)d}+\dots+a_r2^{2d}=
b_12^{n_1}+b_22^{n_2}+\dots+b_s2^{n_s}+c
\end{equation}
for integers $n_1,n_2,\dots,n_s$ with $n_1>n_2>\dots>n_s\ge0,$ then
$r=s,$ $c=0,$ $a_i=b_i,$ and $n_i=2d(r+1-i)$ for $i=1,2,\dots,r$.
\end{lemma}

\begin{proof}
We use induction on $r$.

First, let $r=1$. Then $s=1$ as well, and \eqref{eq:a2b2} becomes
\begin{equation} \label{eq:a2b2A}
a_12^{2d}=b_12^{n_1}+c.
\end{equation}
If $n_1>2d$, then the above equation, together with the assumption
that $a_1$ is odd, implies
$$
2^{2d}\equiv c\quad \text {modulo }2^{2d+1}.
$$
However, by assumption, we have $\vert c\vert\le d<2^{2d}$, which is
absurd.

If $d<n_1< 2d$, then it follows from \eqref{eq:a2b2A} that $c$ must
be divisible by $2^{n_1}$. Again by assumption, we have 
$\vert c\vert\le d<2^d<2^{n_1}$, so that $c=0$. But then
\eqref{eq:a2b2A} cannot be satisfied since $b_1$ is
assumed to be odd.

If $0\le n_1\le d$, then we estimate
$$
b_12^{n_1}+c\le d\left(2^{d}+1\right)\le (2^d-1)(2^d+1)<2^{2d},
$$
which is again a contradiction to \eqref{eq:a2b2A}. 

The only remaining possibility is $n_1=2d$. If this is substituted in
\eqref{eq:a2b2A} and the resulting equation is combined with 
$\vert c\vert\le d<2^{2d}$, then the conclusion is that the equation
can only be satisfied if $c=0$ and $a_1=b_1$, in accordance with
the assertion of the lemma.

\medskip
We now perform the induction step. We assume that the assertion of
the lemma is established for all $r<R$, and we want to show that this
implies its validity for $r=R$. Let $t$ be maximal such that $n_t\ge
2d$. Then reduction of \eqref{eq:a2b2} modulo $2^{2d}$ yields
\begin{equation} \label{eq:a2b2B}
b_{t+1}2^{n_{t+1}}+b_{t+2}2^{n_{t+2}}+\dots+b_s2^{n_s}+c\equiv 0\quad 
\text {modulo }2^{2d}.
\end{equation}
Let us write $b\cdot 2^{2d}$ for the left-hand side in
\eqref{eq:a2b2B}. Then, by dividing \eqref{eq:a2b2} (with $R$ instead
of $r$) by $2^{2d}$, we obtain
\begin{equation} \label{eq:a2b2C}
a_12^{2(R-1)d}+a_22^{2(R-2)d}+\dots+a_{R-1}2^{2d}=
b_12^{n_1-2d}+b_22^{n_2-2d}+\dots+b_t2^{n_t-2d}+b-a_R.
\end{equation}
We have
\begin{align*}
0\le b&\le
2^{-2d}d\left(2^{2d-1}+2^{2d-2}+\dots+2^{2d-s+t}+1\right)\\[2mm]
&\le 2^{-2d}d\left(2^{2d}-2^{2d-s+t}+1\right)\le d.
\end{align*}
Consequently, we also have $\vert b-a_R\vert\le d$.
This means that we are in a position to apply the induction
hypothesis to \eqref{eq:a2b2C}. The conclusion is that
$t=R-1$, $b-a_R=0$, $a_i=b_i$, and $n_i=2d(R+1-i)$ for
$i=1,2,\dots,R-1$. If this is used in \eqref{eq:a2b2} with $r=R$,
then we obtain
$$
a_R2^{2d}=c
$$
or
$$
a_R2^{2d}=b_R2^{n_R}+c,
$$
depending on whether $s=R-1$ or $s=R$. The first case is absurd since
$c\le d<2^{2d}\le a_R2^{2d}$. On the other hand, the second case
has already been
considered in \eqref{eq:a2b2A}, and we have seen there that it
follows that $c=0$, $a_R=b_R$, and $n_R=2d$.

This completes the proof of the lemma.
\end{proof}

The announced independence of the series $H_{a_1,a_2,\dots,a_r}(z)$
with all $a_i$'s odd is now an easy consequence.

\begin{corollary} \label{lem:Hind}
The series $H_{a_1,a_2,\dots,a_r}(z)$, with all $a_i$'s odd, together
with the series $1$ are
linearly independent over $(\Z/2\Z)[z,z^{-1}]$, and consequently as
well over $(\Z/2^\ga\Z)[z,z^{-1}]$ for an arbitrary positive integer
$\ga$, and over $\Z[z,z^{-1}]$.
\end{corollary}

\begin{proof}
Let us suppose that
\begin{equation} \label{eq:Hlincomb}
p_0(z)+\sum _{i=1}
^{N}p_i(z)H_{a_1^{(i)},a_2^{(i)},\dots,a_{r_i}^{(i)}}(z)=0,
\end{equation} 
where the $p_i(z)$'s are non-zero Laurent polynomials in $z$
over $\Z/2\Z$ (respectively over $\Z/2^\ga\Z$ or over $\Z$), the
$r_i$'s are positive integers, and 
$a_j^{(i)}$, $j=1,2,\dots,r_i$, $i=1,2,\dots,N$, are odd integers. 
We may also assume that the tuples
$(a_1^{(i)},a_2^{(i)},\dots,a_{r_i}^{(i)})$, $i=1,2,\dots,N$, 
are pairwise distinct. Choose $i_0$ such
that $r_{i_0}$ is maximal among the $r_i$'s. Without loss of
generality, we may assume that the coefficient of $z^0$ in $p_{i_0}(z)$
is non-zero (otherwise we could multiply both sides of 
\eqref{eq:Hlincomb} by an appropriate power of $z$). Let $d$ be the
maximum of all $a_j^{(i)}$'s and the absolute values of exponents of
$z$ appearing in monomials with non-zero coefficient in
the Laurent polynomials $p_i(z)$, $i=0,1,\dots,N$. Then, according to
Lemma~\ref{lem:aiodd} with $r=r_{i_0}$, $a_j=a_j^{(i_0)}$,
$j=1,2,\dots,r_{i_0}$, the coefficient of
$$
z^{a_1^{(i_0)}2^{2rd}+a_2^{(i_0)}2^{2(r-1)d}+\dots+a_r^{(i_0)}2^{2d}}
$$
is $1$ in $H_{a_1^{(i_0)},a_2^{(i_0)},\dots,a_{r_{i_0}}^{(i_0)}}(z)$,
while it is zero in series 
$z^eH_{a_1^{(i_0)},a_2^{(i_0)},\dots,a_{r_{i_0}}^{(i_0)}}(z)$, where
$e$ is a non-zero integer with $\vert e\vert\le d$, and in all other series
$z^eH_{a_1^{(i)},a_2^{(i)},\dots,a_{r_i}^{(i)}}(z)$,
$i=1,\dots,i_0-1,i_0+1,\dots,N$, where $e$ is a (not necessarily
non-zero) integer with $\vert e\vert\le d$.
This contradiction to \eqref{eq:Hlincomb} establishes
the assertion of the corollary.
\end{proof}

\begin{remarknu} \label{rem:eff}
Coefficient extraction from a series $H_{a_1,a_2,\dots,a_{r}}(z)$
with all $a_i$'s odd is straightforward: if we want to know whether
$z^M$ appears in $H_{a_1,a_2,\dots,a_{r}}(z)$, that is, whether we can
represent $M$ as
$$
M=a_12^{n_1}+a_22^{n_2}+\dots+a_r2^{n_r}
$$
for some $n_1,n_2,\dots,n_r$ with $n_1>n_2>\dots>n_r\ge0$, then
necessarily $n_r=v_2(M)$, $n_{r-1}=v_2(M-a_r2^{n_r})$, etc.
The term $z^M$ appears in $H_{a_1,a_2,\dots,a_{r}}(z)$ if, and only
if, the above process terminates after {\it exactly} $r$ steps.
This means, that, with $n_r,n_{r-1},\dots,n_1$ constructed as above, 
we have
$$
M-\left(a_s2^{n_s}+\dots+a_{r-1}2^{n_{r-1}}+a_r2^{n_r}\right)>0
$$
for $s>1$, and
$$
M-\left(a_12^{n_1}+\dots+a_{r-1}2^{n_{r-1}}+a_r2^{n_r}\right)=0.
$$
It should be noted that, given $a_1,a_2,\dots,a_r$, this procedure of
coefficient extraction needs at most $O(\log M)$ operations, that is,
its computational complexity is linear.
\end{remarknu}

Our next goal is to show that a series $H_{b_1,b_2,\dots,b_s}(z)$
can be expressed as a linear combination over $\Z[z,z^{-1}]$ of
the series $1$ and 
the series $H_{a_1,a_2,\dots,a_r}(z)$, where all $a_i$'s are odd.
In doing this, we are forced to consider the more general series
$$
H_{b_1,b_2,\dots,b_s}^{\be_1,\be_2,\dots,\be_s}(z):=
\sum _{n_1+\be_1>n_2+\be_2>\dots>n_s+\be_s\ge0} ^{}
z^{b_12^{n_1}+b_22^{n_2}+\dots+b_s2^{n_s}},
$$
where, as before, $b_1,b_2,\dots,b_s$ are positive integers, 
and $\be_1,\be_2,\dots,\be_s$ are integers.

\begin{lemma} \label{lem:Hbi}
For positive integers $b_1,b_2,\dots,b_s$ 
and integers $\be_1,\be_2,\dots,\be_s,$ the series
$H_{b_1,b_2,\dots,b_s}^{\be_1,\be_2,\dots,\be_s}(z)$
can be expressed as a linear combination
over $\Z[z^{1/2^e}]$ {\em(}for a suitable integer $e${\em)}
of the series $1$ and series of the form
$H_{a_1,a_2,\dots,a_r}(z),$ where all $a_i$'s are odd.
Moreover, in the above expansion of the series 
$H_{b_1,b_2,\dots,b_s}(z)=
H_{b_1,b_2,\dots,b_s}^{0,0,\dots,0}(z)$
we have $e=0$; that is, in that case all coefficients are
in $\Z[z]$.
\end{lemma}

\begin{proof}
We describe an algorithmic procedure for expressing 
$H_{b_1,b_2,\dots,b_s}^{\be_1,\be_2,\dots,\be_s}(z)$ 
in terms of series 
$H_{a_1,a_2,\dots,a_r}^{\ga_1,\ga_2,\dots,\ga_r}(z)$,
where either $r<s$, or $r=s$ and 
$$
\max\{i:a_i\text { is even or }\ga_i\ne0\}
<\max\{i:b_i\text { is even or }\be_i\ne0\}.
$$
In words, in the second case
the length of the string of consecutive $0$'s at the tail of the upper parameters 
respectively the length of the string of consecutive 
odd numbers at the tail of the lower parameters has been increased.

Our algorithmic procedure consists of four recurrence relations,
\eqref{eq:Rek1}--\eqref{eq:Rek4} below. 
For the first two of these, let $b_s=b'_s2^{e_s}$, where
$e_s=v_2(b_s)$. By definition, the number $b'_s$ is odd.
Then we have
\begin{multline*}
H_{b_1,b_2,\dots,b_s}^{\be_1,\be_2,\dots,\be_s}(z)\\[2mm]
= \sum _{n_1+\be_1-\be_s+e_s>\dots>
n_{s-1}+\be_{s-1}-\be_s+e_s>
n_s+e_s\ge -\be_s+e_s} ^{}
z^{b_12^{n_1}+b_22^{n_2}+\dots+b_{s-1}2^{n_{s-1}}+b'_s2^{n_s+e_s}}.
\end{multline*}
In the above sum on the right-hand side,
let $n'_s=n_s+e_s$ be a new summation index. Then, for
$e_s\le\be_s$, one sees that
\begin{multline} \label{eq:Rek1}
H_{b_1,b_2,\dots,b_s}^{\be_1,\be_2,\dots,\be_s}(z)=
H_{b_1,b_2,\dots,b_{s-1},b'_s}
^{\be_1-\be_s+e_s,\be_2-\be_s+e_s,\dots,
\be_{s-1}-\be_s+e_s,0}(z)\\[2mm]
+\sum _{k=1} ^{\be_s-e_s}z^{b'_s2^{-k}}
H_{b_1,b_2,\dots,b_{s-1}}
^{\be_1-\be_s+e_s+k-1,\be_2-\be_s+e_s+k-1,\dots,
\be_{s-1}-\be_s+e_s+k-1}(z).
\end{multline}
On the other hand, for $e_s\ge\be_s$, one has
\begin{multline} \label{eq:Rek2}
H_{b_1,b_2,\dots,b_s}^{\be_1,\be_2,\dots,\be_s}(z)=
H_{b_1,b_2,\dots,b_{s-1},b'_s}
^{\be_1-\be_s+e_s,\be_2-\be_s+e_s,\dots,
\be_{s-1}-\be_s+e_s,0}(z)\\[2mm]
-\sum _{k=0} ^{e_s-\be_s-1}z^{b'_s2^{k}}
H_{b_1,b_2,\dots,b_{s-1}}
^{\be_1-\be_s+e_s-k-1,\be_2-\be_s+e_s-k-1,\dots,
\be_{s-1}-\be_s+e_s-k-1}(z).
\end{multline}

Now consider
$$
H_{b_1,\dots,b_h,b_{h+1},\dots,b_s}
^{\be_1,\dots,\be_h,0,\dots,0}(z),
$$
where $1\le h<s$ and 
all of $b_{h+1},\dots,b_s$ are odd. Similar to the proceedings 
above, let $b_h=b'_h2^{e_h}$, where
$e_h=v_2(b_h)$. Again, by definition, the number $b'_h$ is odd.
Then we have
\begin{multline*}
H_{b_1,\dots,b_h,b_{h+1},\dots,b_s}
^{\be_1,\dots,\be_h,0,\dots,0}(z)\\[2mm]
= \underset{n_h+\be_h>n_{h+1}>\dots>n_s\ge0}
{\sum _{n_1+\be_1-\be_h+e_h>\dots>
n_{h-1}+\be_{h-1}-\be_h+e_h>
n_h+e_h} ^{}}
z^{b_12^{n_1}+\dots+b_{h-1}2^{n_{h-1}}+b'_h2^{n_h+e_h}+
b_{h+1}2^{n_{h+1}}+\dots+b_s2^{n_s}}.
\end{multline*}
In the above sum on the right-hand side,
let $n'_h=n_h+e_h$ be a new summation index. Then, for
$e_h\le\be_h$, one sees that
\begin{multline} \label{eq:Rek3}
H_{b_1,\dots,b_h,b_{h+1},\dots,b_s}
^{\be_1,\dots,\be_h,0,\dots,0}(z)=
H_{b_1,\dots,b_{h-1},b'_h,b_{h+1},\dots,b_s}
^{\be_1-\be_h+e_h,\dots,
\be_{h-1}-\be_h+e_h,0,0,\dots,0}(z)\\[2mm]
+\sum _{k=0} ^{\be_h-e_h-1}
H_{b_1,\dots,b_{h-1},b'_h+b_{h+1}2^k,b_{h+2},\dots,b_s}
^{\be_1-\be_h+e_h+k,\dots,
\be_{h-1}-\be_h+e_h+k,k,0,\dots,0}(z).
\end{multline}
On the other hand, for $e_h\ge\be_h$, one has
\begin{multline} \label{eq:Rek4}
H_{b_1,\dots,b_h,b_{h+1},\dots,b_s}
^{\be_1,\dots,\be_h,0,\dots,0}(z)=
H_{b_1,\dots,b_{h-1},b'_h,b_{h+1},\dots,b_s}
^{\be_1-\be_h+e_h,\dots,
\be_{h-1}-\be_h+e_h,0,0,\dots,0}(z)\\[2mm]
-\sum _{k=1} ^{e_h-\be_h}
H_{b_1,\dots,b_{h-1},b'_h2^k+b_{h+1},b_{h+2},\dots,b_s}
^{\be_1-\be_h+e_h-k,\dots,
\be_{h-1}-\be_h+e_h-k,0,0,\dots,0}(z).
\end{multline}

It is clear that, if we recursively apply \eqref{eq:Rek1}--\eqref{eq:Rek4} to a
given series 
$H_{b_1,b_2,\dots,b_s}^{\be_1,\be_2,\dots,\be_s}(z)$, 
and use $H_\emptyset^\emptyset(z)=1$ as an initial condition,
we will eventually arrive at a linear combination of $1$ and series
$H_{a_1,a_2,\dots,a_r}^{0,0,\dots,0}(z)=
H_{a_1,a_2,\dots,a_r}(z)$
with all $a_i$'s being odd, where the coefficients are 
polynomials in $z^{1/2^e}$ for a suitable $e$. (Potential fractional 
exponents come from the relation \eqref{eq:Rek1}.)
This proves the first assertion of the lemma.

Now let us consider the case where all the $\be_i$'s are zero. Suppose 
that we have an expansion as described in the first part of the lemma
for $H_{b_1,b_2,\dots,b_s}(z)$,
\begin{equation} \label{eq:be=0}
H_{b_1,b_2,\dots,b_s}(z)=
{\sum _{\mathbf a} ^{}}
c(\mathbf a)
z^{e(\mathbf a)}
H_{\mathbf a}(z),
\end{equation}
where the sum is taken over all finite
tuples $\mathbf a=(a_1,a_2,\dots)$ with all $a_i$'s being odd, and
where only finitely many coefficients $c(\mathbf a)$ are non-zero.
We also allow the tuple $\mathbf a$ to be the empty tuple $()$ and
make the convention that $H_{()}(z)=1$, so that the series $1$ is as
well included in the linear combination on the right-hand side of
\eqref{eq:be=0}.

Let us now consider exponents $e(\mathbf a)$ that are not
integral. Let $\epsilon$ be a real number strictly between $0$ and
$1$, and concentrate on exponents $e(\mathbf a)$ 
with fractional part $\epsilon$; in symbols 
$\{e(\mathbf a)\}=\epsilon$. Then we isolate these exponents
$e(\mathbf a)$ in the relation \eqref{eq:be=0}, and since there are
no fractional exponents on the left-hand side, we obtain
$$
0= 
{{\sum _{\{e(\mathbf a)\}=\epsilon} ^{}}}
c(\mathbf a)
z^{e(\mathbf a)}
H_{\mathbf a}(z).
$$
After dividing both sides through by $z^\epsilon$, an application 
of Corollary~\ref{lem:Hind} shows that $c(\mathbf a)=0$ for all
$\mathbf a$ with $\{e(\mathbf a)\}=\epsilon$. Thus, all exponents 
$e(\mathbf a)$ actually occurring in \eqref{eq:be=0} 
with non-zero coefficients $c(\mathbf a)$ are in fact integral.
This completes the proof of the lemma.
\end{proof}

Computer computations suggest that, if we restrict our attention to
the series\break 
$H_{b_1,b_2,\dots,b_s}(z)$, which are the ones that we are
actually interested in, there is a strengthening of
Lemma~\ref{lem:Hbi} (see also Appendix~\ref{appA}).

\begin{conjecture} \label{conj:H}
For any positive integers $b_1,b_2,\dots,b_s,$ the series
$H_{b_1,b_2,\dots,b_s}(z)$
can be expressed as a linear combination
over $\Z[z,z^{-1}]$ of the series $1$ and series of the form
$H_{a_1,a_2,\dots,a_r}(z),$ where all $a_i$'s are odd, $r\le s,$ and
$a_1+a_2+\dots+a_r\le b_1+b_2+\dots+b_s$.
\end{conjecture}

To conclude this section, let us provide an illustration of the above
discussion. We set ourselves the task of determining the coefficient
of $z^{1099511640192}$ in $\Phi^5(z)$.
In order to accomplish this task, we first express $\Phi^5(z)$ in
terms of series $H_{a_1,\dots,a_r}(z)$ with all $a_i$'s being odd.
This is done by means of the expansion \eqref{eq:Phipot} and 
the algorithm described in the proof of Lemma~\ref{lem:Hbi}.
The resulting expansion is displayed in Appendix~\ref{appA}.

Now we have to answer the question, in which of the series
$H_{a_1,\dots,a_r}(z)$ that appear in this expansion of $\Phi^5(z)$
do we find the monomial $z^{1099511640192}$. Using the algorithm
described in Remark~\ref{rem:eff}, we see that 
{\allowdisplaybreaks
\begin{align*}
1099511640192&=5\cdot 2^7+1099511639552,\\
&=3\cdot 2^{13} +2^{12} +2^7+1099511611392,\\
&=2^{40} +3\cdot 2^{12} +2^7 ,\\
&=2^9 +2^8 +3\cdot 2^7+1099511639040,\\
&=3\cdot 2^{12} +2^7+1099511627776,\\
&=2^8 +3\cdot 2^7+1099511639552,\\
&=2^{40} +2^{13} +2^{12} +2^7 ,\\
&=1+3\cdot 2^0+1099511640188,\\
&=3\cdot 2^7+1099511639808,\\
&=1+2^2+2^1+2^0+1099511640184,\\
&=2^{13} +2^{12} +2^7+1099511627776,\\
&=1+2^1+2^0+1099511640188,\\
&=2^{12} +2^7+1099511635968,\\
&=2+2^1+1099511640188,\\
&=1+2^0+1099511640190,\\
&=2^7+1099511640064.
\end{align*}}%
Here, the third line shows that $z^{1099511640192}$ appears in
$H_{1,3,1}(z)$, and the seventh line shows that it appears in
$H_{1,1,1,1}(z)$ (thereby making it impossible to appear
in $H_{1,1,1,1,1}(z)$), while the remaining lines show that it does not
appear in any other term in the expansion of $\Phi^5(z)$ displayed in
Appendix~\ref{appA}. Hence, by taking into account the coefficients
with which the series $H_{1,3,1}(z)$ and $H_{1,1,1,1}(z)$ appear in
this expansion, the coefficient of $z^{1099511640192}$
in $\Phi^5(z)$ is seen to equal $-40+240=200$.

\section{The method}
\label{sec:method}

We consider a (formal) differential equation 
\begin{equation} \label{eq:diffeq}
\mathcal P(z;F(z),F'(z),F''(z),\dots,F^{(s)}(z))=0,
\end{equation}
where $\mathcal P$ is a polynomial with integer coefficients, which has a
power series solution $F(z)$ with integer coefficients. In this
situation, we propose the following algorithmic approach to
determining the series $F(z)$ modulo a $2$-power $2^{3\cdot 2^\al}$,
for some positive integer $\al$. We make the Ansatz
\begin{equation} \label{eq:Ansatz}
F(z)=\sum _{i=0} ^{2^{\al+2}-1}a_i(z)\Phi^i(z)\quad \text {modulo
}2^{3\cdot 2^\al},
\end{equation}
with $\Phi(z)$ as given in \eqref{eq:Phidef}, and where the $a_i(z)$'s
are (at this point) undetermined Laurent polynomials in $z$.
Now we substitute \eqref{eq:Ansatz} into \eqref{eq:diffeq}, and
we shall gradually determine approximations $a_{i,\be}(z)$ to $a_i(z)$ such that
\eqref{eq:diffeq} holds modulo $2^\be$, for $\be=1,2,\dots,3\cdot
2^\al$. To start the procedure, we consider the differential equation
\eqref{eq:diffeq} modulo $2$, with
\begin{equation} \label{eq:Ansatz1}
F(z)=\sum _{i=0} ^{2^{\al+2}-1}a_{i,1}(z)\Phi^i(z)\quad \text {modulo
}2.
\end{equation}
Using the elementary fact that $\Phi'(z)=1$ modulo $2$, we see that
the left-hand side of \eqref{eq:diffeq} is a polynomial in
$\Phi(z)$ with coefficients that are Laurent polynomials in $z$. 
We reduce powers $\Phi^k(z)$ with $k\ge2^{\al+2}$ using
the relation (which is implied by the minimal polynomial for the
modulus~$8$ given in Proposition~\ref{prop:minpol})\,\footnote{Actually, 
if we would like to obtain an optimal result, 
we should use the relation implied by
a minimal polynomial for the modulus $2^{3\cdot 2^\al}$ in the sense of 
Section~\ref{sec:Phi}. But since we have no general formula available
for such a minimal polynomial (cf.\ Item (2) of Remark~\ref{rem:1} in that section),
and since we wish to prove results for arbitrary moduli, choosing instead
powers of a minimal polynomial for the modulus $8$ is the best compromise.
In principle, it may happen that there exists a polynomial in $\Phi(z)$
with coefficients that are Laurent polynomials in $z$, which is identical
with $F(z)$ after reduction of its coefficients modulo $2^{3\cdot 2^\al}$,
but the Ansatz \eqref{eq:Ansatz} combined with the
reduction \eqref{eq:PhiRel} fails because it is too restrictive. 
%However,
%it turns out that --- at least in all the cases that we treat --- this
%obstruction does not occur. 
We are not aware of a concrete example where this obstruction occurs.
The subgroup numbers of $SL_2(\Z)$ %??and of $\Ga_3(3)$ 
(which we treat
modulo~$8$ in Section~\ref{sec:SL2Z} %??and \ref{sec:Ga3} 
by 
the method described here, and modulo~$16$ by an enhancement of the method
outlined in Appendix~\ref{appD}) are a potential candidate when
considered modulo $2^{3\cdot 2^\al}$ for $\al\ge1$.
On the other hand, once we are successful using this
(potentially problematic) Ansatz, then the result can easily be converted into
an optimal one by further reducing the polynomial thus obtained, using
the relation implied by a minimal polynomial for the modulus
$2^{3\cdot 2^\al}$.}
\begin{equation} \label{eq:PhiRel}
\big(\Phi^4(z)+6\Phi^3(z)+(2z+3)\Phi^2(z)+(2z+6)\Phi(z)+2z+5z^2\big)
^{2^\al}=0\quad \text {modulo }2^{3\cdot 2^\al}.
\end{equation}
Since, at this point, we are only interested in finding a solution to
\eqref{eq:diffeq} modulo~$2$, the above relation simplifies to
\begin{equation} \label{eq:PhiRel2}
\Phi^{2^{\al+2}}(z)+\Phi^{2^{\al+1}}(z)+z^{2^{\al+1}}=0\quad 
\text{modulo }2.
\end{equation}
Now we compare coefficients of powers $\Phi^k(z)$,
$k=0,1,\dots,2^{\al+2}-1$ (see Remark~\ref{rem:comp}). This yields a
system of $2^{\al+2}$ (differential) equations (modulo~$2$)
for the unknown Laurent polynomials $a_{i,1}(z)$, $i=0,1,\dots,2^{\al+2}-1$,
which may or may not have a solution. 

Provided we have already found
Laurent polynomials $a_{i,\be}(z)$, $i=0,1,\dots,2^{\al+2}-1$, for some $\be$
with $1\le \be\le 3\cdot 2^{\al}-1$, such that
\begin{equation} \label{eq:Ansatz2}
\sum _{i=0} ^{2^{\al+2}-1}a_{i,\be}(z)\Phi^i(z)
\end{equation}
solves \eqref{eq:diffeq} modulo~$2^\be$, we put 
\begin{equation} \label{eq:Ansatz2a}
a_{i,\be+1}(z):=a_{i,\be}(z)+2^{\be}b_{i,\be+1}(z),\quad 
i=0,1,\dots,2^{\al+2}-1,
\end{equation}
where the $b_{i,\be+1}(z)$'s are (at this point) undetermined 
Laurent polynomials in $z$. Next we substitute
\begin{equation} \label{eq:Ansatz2b}
\sum _{i=0} ^{2^{\al+2}-1}a_{i,\be+1}(z)\Phi^i(z)
\end{equation}
instead of $F(z)$ in \eqref{eq:diffeq}. Using the fact that
$\Phi'(z)=\sum _{n=0} ^{\be}2^nz^{2^n-1}$ modulo~$2^{\be+1}$, 
we expand the left-hand side as a polynomial in $\Phi(z)$ (with
coefficients being Laurent polynomials in $z$), we apply again the reduction
using relation \eqref{eq:PhiRel}, we compare coefficients of powers
$\Phi^k(z)$, $k=0,1,\dots,2^{\al+2}-1$ (again, see
Remark~\ref{rem:comp}),
and, as a result, we obtain a
system of $2^{\al+2}$ (differential) equations (modulo~$2^{\be+1}$)
for the unknown Laurent polynomials $b_{i,\be+1}(z)$, $i=0,1,\dots,2^{\al+2}-1$,
which may or may not have a solution. If we manage to push this 
procedure through until $\be=3\cdot 2^\al-1$, then,
setting $a_i(z)=a_{i,3\cdot 2^\al}(z)$, $i=0,1,\dots,2^{\al+2}-1$,
the right-hand side of \eqref{eq:Ansatz} is a solution to
\eqref{eq:diffeq} modulo~$2^{3\cdot 2^\al}$, as required.

\begin{remarknu} \label{rem:comp}
As the reader will have noticed, each comparison of coefficients of
powers of $\Phi(z)$ is based on the ``hope" that, 
if a polynomial in $\Phi(z)$ is zero modulo a $2$-power $2^\be$
(as a formal Laurent series), then already all coefficients
of powers of $\Phi(z)$ in this polynomial vanish modulo $2^\be$.
However, this implication is false in general 
(see Lemma~\ref{lem:Null16} below for the case of modulus $2^4=16$).
It may thus happen that the method described in this section fails to find a
solution modulo $2^\be$
to a given differential equation in the form of a polynomial in
$\Phi(z)$ with coefficients that are Laurent polynomials in $z$ over
the integers, while such a solution does in fact
exist. As a matter of fact, this situation occurs in the 
analysis modulo~$16$ of the
subgroup numbers of $SL_2(\Z)$,
see 
%??Remarks~\ref{rem:2}, \ref{rem:Ga3}, and
Theorem~\ref{thm:Sl2Z16}.
In Appendix~\ref{appD}, we outline an enhancement of the method,
which (at least in principle; Appendix~\ref{appD} treats only the
case of the modulus~$16$ explicitly) allows us to decide whether
or not a solution modulo a given power in terms of a polynomial in
$\Phi(z)$ with coefficients that are Laurent polynomials in $z$ over
the integers exists, and, if so, to explicitly find such a solution.
\end{remarknu}

It is not difficult to see that performing the iterative step 
\eqref{eq:Ansatz2a} amounts to solving a system of linear
differential equations in the unknown functions $b_{i,\be+1}(z)$
modulo~$2$, where all of them are Laurent polynomials in $z$,
and where only first derivatives of the $b_{i,\be+1}(z)$'s occur.
Solving such a system is equivalent to solving an ordinary system of
linear equations, as is shown by the lemma below.

Given a Laurent polynomial $p(z)$ over the integers, we write
$p^{(o)}(z)$ for the odd part $\frac {1} {2}(p(z)-p(-z))$
and $p^{(e)}(z)$ for the even part $\frac {1} {2}(p(z)+p(-z))$ of $p(z)$,
respectively.

\begin{lemma} \label{lem:2x2diff}
Let $c_{i,j}(z)$ and $d_{i,j}(z),$ $1\le i,j\le N,$ and $r_i(z),$ $1\le
i\le N,$ be
given Laurent polynomials in $z$ with integer coefficients.
Then the system of differential equations
\begin{equation} 
\label{eq:2x2diff}
\sum _{j=1} ^{N}c_{i,j}(z)f_j(z)+
\sum _{j=1} ^{N}d_{i,j}(z)f'_j(z)=r_i(z)\quad 
\text {\em modulo }2, \quad \quad 1\le i\le N,
\end{equation}
has solutions $f_j(z),$ $1\le j\le N,$ that are Laurent polynomials
in $z$ over the integers
if, and only if, the system of\/ {\em linear} equations
\begin{align}
\notag
\sum _{j=1} ^{N}c^{(e)}_{i,j}(z)f^{(1)}_j(z)+
\sum _{j=1} ^{N}c^{(o)}_{i,j}(z)f^{(2)}_j(z)+
\sum _{j=1} ^{N}z^{-1}d^{(1)}_{i,j}(z)f^{(2)}_j(z)&=r^{(e)}_i(z)\quad 
\text {\em modulo }2, \\[2mm]
\notag
\sum _{j=1} ^{N}c^{(o)}_{i,j}(z)f^{(1)}_j(z)+
\sum _{j=1} ^{N}c^{(e)}_{i,j}(z)f^{(2)}_j(z)+
\sum _{j=1} ^{N}z^{-1}d^{(o)}_{i,j}(z)f^{(2)}_j(z)&=r^{(o)}_i(z)\quad 
\text {\em modulo }2,\\
&\kern3cm 
1\le i\le N,
\label{eq:4x4}
\end{align}
has a solution in Laurent polynomials 
$f_j^{(1)}(z),f_j^{(2)}(z)$ in $z$ over the integers for $1\le j\le N$.
\end{lemma}

\begin{proof}We write $f_j(z)=f_j^{(e)}(z)+f_j^{(o)}(z)$,
and observe that
$$
f_j'(z)=z^{-1}f_j^{(o)}(z)\quad 
\text {modulo }2.
$$
If this is used in
\eqref{eq:2x2diff}, and if we separate the even and odd parts on both
sides of the equations, then \eqref{eq:4x4} with $f_j^{(1)}(z)=f_j^{(e)}(z)$
and $f_j^{(2)}(z)=f_j^{(o)}(z)$, $j=1,2,\dots,N$, results after little
manipulation. 

Conversely, suppose that $g_j^{(1)}(z),g_j^{(2)}(z)$,
$j=1,2,\dots,N$, is a solution to the system \eqref{eq:4x4}, that
is,
\begin{align}
\notag
\sum _{j=1} ^{N}c^{(e)}_{i,j}(z)g^{(1)}_j(z)+
\sum _{j=1} ^{N}c^{(o)}_{i,j}(z)g^{(2)}_j(z)+
\sum _{j=1} ^{N}z^{-1}d^{(1)}_{i,j}(z)g^{(2)}_j(z)&=r^{(e)}_i(z)\quad 
\text { modulo }2, \\[2mm]
\notag
\sum _{j=1} ^{N}c^{(o)}_{i,j}(z)g^{(1)}_j(z)+
\sum _{j=1} ^{N}c^{(e)}_{i,j}(z)g^{(2)}_j(z)+
\sum _{j=1} ^{N}z^{-1}d^{(o)}_{i,j}(z)g^{(2)}_j(z)&=r^{(o)}_i(z)\quad 
\text { modulo }2,\\
&\kern3cm 
1\le i\le N.
%\label{eq:4x4g}
\end{align}
At this point, the $g_j^{(1)}(z)$'s need not be even Laurent polynomials, and
the $g_j^{(2)}(z)$'s need not be odd Laurent polynomials. We have to prove
that there exists a solution
$f_j^{(1)}(z),f_j^{(2)}(z)$, $j=1,2,\dots,N$, such that all
$f_j^{(1)}(z)$'s are even Laurent polynomials and all 
$f_j^{(2)}(z)$'s are odd Laurent polynomials.

By separating even and odd parts of the $g_k^{(1)}(z)$'s and the
$g_j^{(2)}(z)$'s, we obtain the equations
%{\allowdisplaybreaks
\begin{align}
\notag
\sum _{j=1} ^{N}c^{(e)}_{i,j}(z)(g^{(1)}_j)^{(e)}(z)+
\sum _{j=1} ^{N}c^{(o)}_{i,j}(z)(g^{(2)}_j)^{(o)}(z)+
\sum _{j=1}
^{N}z^{-1}d^{(1)}_{i,j}(z)(g^{(2)}_j)^{(o)}(z)&=r^{(e)}_i(z)\\
\label{eq:4x4g1}
&\text { modulo }2, \\[2mm]
\notag
\sum _{j=1} ^{N}c^{(e)}_{i,j}(z)(g^{(1)}_j)^{(o)}(z)+
\sum _{j=1} ^{N}c^{(o)}_{i,j}(z)(g^{(2)}_j)^{(e)}(z)+
\sum _{j=1} ^{N}z^{-1}d^{(1)}_{i,j}(z)(g^{(2)}_j)^{(e)}(z)&=0\\
\notag
&\text { modulo }2, \\[2mm]
\notag
\sum _{j=1} ^{N}c^{(o)}_{i,j}(z)(g^{(1)}_j)^{(e)}(z)+
\sum _{j=1} ^{N}c^{(e)}_{i,j}(z)(g^{(2)}_j)^{(o)}(z)+
\sum _{j=1}
^{N}z^{-1}d^{(o)}_{i,j}(z)(g^{(2)}_j)^{(o)}(z)&=r^{(o)}_i(z)\\
\label{eq:4x4g2}
&\text { modulo }2,\\
\notag
\sum _{j=1} ^{N}c^{(o)}_{i,j}(z)(g^{(1)}_j)^{(o)}(z)+
\sum _{j=1} ^{N}c^{(e)}_{i,j}(z)(g^{(2)}_j)^{(e)}(z)+
\sum _{j=1} ^{N}z^{-1}d^{(o)}_{i,j}(z)(g^{(2)}_j)^{(e)}(z)&=0\\
\notag
&\text { modulo }2,\\
\notag
&\kern0cm 
1\le i\le N.
%\label{eq:4x4g}
\end{align}%}%
Combining \eqref{eq:4x4g1} and \eqref{eq:4x4g2}, we see that
$(g_j^{(1)})^{(e)}(z),(g_j^{(2)})^{(o)}(z)$, $j=1,2,\dots,N$, is a
solution to \eqref{eq:4x4}, and now the 
$(g_j^{(1)})^{(e)}(z)$'s are indeed even polynomials while the 
$(g_j^{(2)})^{(o)}(z)$'s are indeed odd polynomials.
Addition of both sides of \eqref{eq:4x4g1} and \eqref{eq:4x4g2} then
yields that 
$$
f_j(z):=(g_j^{(1)})^{(e)}(z)+(g_j^{(2)})^{(o)}(z),\quad 1\le j\le N,
$$
is a solution to \eqref{eq:2x2diff} in Laurent polynomials in $z$
over the integers.
\end{proof}

In general, it is difficult to characterise when the system
\eqref{eq:4x4} has a solution. What one has to do is to solve the
system over the {\it field\/} of rational functions in $z$ over
$\Z/2\Z$, and then to see whether possibly occurring denominators
cancel out or, in the case of a parametric solution, whether denominators
can be {\it made} to cancel by a suitable choice of the parameters.
One simple case, where a
characterisation is possible, is given in Lemma~\ref{lem:diff710},
which is crucial for the proof that the generating function for the
number of subgroups of index $n$ of $PSL_2(\mathbb Z)$, when these
are reduced modulo a given power of $2$, can always be expressed as a
polynomial in $\Phi(z)$ with coefficients that are Laurent
polynomials in $z$.

\medskip
We remark that the idea of the method that we have described in this
section has certainly further potential. For example, the fact that
the series $\Phi(z)$ remains invariant under the substitution $z\to
z^2$ (or, more generally, under the substitution $z\to z^{2^h}$, 
where $h$ is some
positive integer) --- up to a simple additive correction --- can be
exploited in order to extend the range of applicability of our method
to equations where we not only allow differentiation but also this
kind of substitution. This is actually already used in a very hidden way in
Section~\ref{sec:H5} (cf.\ \cite[Theorem~12]{MuHecke}, setting in
relation subgroup numbers of the Hecke group $\mathfrak H(q)$ with
subgroup numbers of $C_q*C_q$ modulo~$2$; in terms of generating
functions, the meaning of this theorem is that the generating function
for the former numbers can be expressed in terms of the generating
function for the latter numbers by a relation
which involves a substitution $z\to z^2$).
Furthermore, as we already mentioned in the introduction, 
there is no obstacle to modifying the method presented here to work 
for recursive sequences which are reduced modulo powers of $p$, in
connection with
the series $\sum _{n\ge0} ^{}z^{p^n}$, although at present we are not
able to offer any interesting applications in this direction.

\section{A sample application: Catalan numbers}
\label{sec:Cat}

The {\it Catalan numbers}, defined by $\Cat_n=\frac {1} {n+1}\binom
{2n}n$, $n=0,1,\dots$, are ubiquitous in enumerative combinatorics.
(Stanley provides a list of 66~sequences of sets enumerated by 
Catalan numbers in \cite[Ex.~6.19]{StanBI},
with many more in the addendum \cite{Stadd}.)
Recently, there have been several papers on the congruence properties
of Catalan numbers modulo powers of $2$, see
\cite{EuLYAA,LiYeAA,PoSaAA,XiXuAA}. In particular, in \cite{LiYeAA} the
Catalan numbers are determined modulo $64$. As we already mentioned
in the introduction, the corresponding result
(cf.\ \cite[Theorems~6.1--6.6]{LiYeAA}) can be compactly written 
in the form \eqref{eq:Cat64}.
Clearly, once we know the right-hand side of \eqref{eq:Cat64}, 
the validity of the congruence \eqref{eq:Cat64} can be routinely
verified by substituting the right-hand side into the well-known
functional equation (cf.\ \cite[(2.3.8)]{Wilf})
\begin{equation} \label{eq:CatEF}
zC^2(z)-C(z)+1=0,
\end{equation}
where $C(z)=\sum _{n=0} ^{\infty}\Cat_nz^n$ denotes the generating
function for the Catalan numbers,
and reducing powers
of $\Phi(z)$ whose exponent exceeds $7$ by means of the relation
\eqref{eq:PhiRel} with $\al=1$.
We shall now demonstrate that the method from
Section~\ref{sec:method} allows one not only to {\it find\/} the congruence 
\eqref{eq:Cat64} algorithmically, but also to find analogous
congruences modulo {\it arbitrary} powers of $2$.\footnote{\label{F:Lucas}%
In principle, one could use the
generalisations of Lucas' theorem due to Davis and Webb \cite{DaWeAA}, and to
Granville \cite{GranAA}, respectively, to analyse the classical
expression $\frac {1} {n+1}\binom {2n}n$ for the Catalan numbers 
modulo a given $2$-power,
or, more generally, the right-hand side of
\eqref{eq:FCat}. But this approach would be
rather cumbersome in comparison with our method, and it is doubtful
that one would be able to derive results which are of the same level of 
generality as Theorems~\ref{thm:Cat}, \ref{thm:Cat4096}, or \ref{thm:2mf}.}

\begin{theorem} \label{thm:Cat}
Let $\Phi(z)=\sum _{n\ge0} ^{}z^{2^n}$, and let $\al$ be some
positive integer.
Then the generating function $C(z)$ for Catalan numbers, 
reduced modulo $2^{3\cdot 2^\al}$, 
can be expressed as a polynomial in $\Phi(z)$ of degree at most
$2^{\al+2}-1$ with coefficients that are Laurent polynomials in $z$
over the integers.
\end{theorem}

\begin{proof} 
We apply the method from Section~\ref{sec:method}. We start by
substituting the Ansatz \eqref{eq:Ansatz1} in \eqref{eq:CatEF} and
reducing the result modulo $2$. 
In this way, we obtain
\begin{equation} \label{eq:quadr}
z
\sum _{i=0} ^{2^{\al+2}-1}a_{i,1}^2(z)\Phi^{2i}(z)
+\sum _{i=0} ^{2^{\al+2}-1}a_{i,1}(z)\Phi^i(z)+1=0\quad \text
{modulo }2.
\end{equation}
We may reduce $\Phi^{2i}(z)$ further using the relation
\eqref{eq:PhiRel2}. This leads to
\begin{multline} \label{eq:GlCat}
z\sum _{i=0} ^{2^{\al+1}-1}\left(a_{i,1}^2(z)
+z^{2^{\al+1}}a_{i+2^{\al+1},1}^2(z)\right)\Phi^{2i}(z)
+
z\sum _{i=0} ^{2^{\al}-1}
z^{2^{\al+1}}a_{i+3\cdot2^{\al},1}^2(z)\Phi^{2i}(z)\\[2mm]
+
z\sum _{i=0} ^{2^{\al}-1}\left(a_{i+2^{\al+1},1}^2(z)
+a_{i+3\cdot2^{\al},1}^2(z)\right)\Phi^{2i+2^{\al+1}}(z)
+\sum _{i=0} ^{2^{\al+2}-1}a_{i,1}(z)\Phi^i(z)
+1=0\quad \text
{modulo }2.
\end{multline}
Now we compare coefficients of $\Phi^i(z)$, for
$i=0,1,\dots,2^{\al+2}-1$. For $i$ odd, we see immediately that this
implies that $a_{i,1}(z)=0$ modulo~$2$. Proceeding inductively, we
now suppose that $a_{2^\be u,1}(z)=0$ modulo~$2$ for odd $u$ and
some positive integer $\be$, $\be<\al$. Reading off coefficients of 
$\Phi^{2^{\be+1}i}$, where $i$ is odd, we then obtain 
\begin{multline*}
za_{2^{\be}i,1}^2(z)
+z^{2^{\al+1}+1}a_{2^{\be}i+2^{\al+1},1}^2(z)
+z^{2^{\al+1}+1}a_{2^{\be}i+3\cdot2^\al,1}^2(z)
\\[2mm]
+za_{2^{\be}i+2^\al,1}^2(z)
+za_{2^{\be}i+2^{\al+1},1}^2(z)
+a_{2^{\be+1}i,1}(z)
=0\quad \text {modulo }2.
\end{multline*}
However, due to our inductive assumption, all squared terms on the
left-hand side of this congruence vanish, and we conclude that
$a_{2^{\be+1}i,1}(z)=0$ modulo~$2$. 

So far, we have found that all coefficient Laurent polynomials $a_{i,1}(z)$
vanish modulo~$2$ except possibly $a_{0,1}(z)$ and
$a_{2^{\al+1},1}(z)$. 
The corresponding congruences that we obtain
from extracting coefficients of $\Phi^0(z)$ and
$\Phi^{2^{\al+1}}(z)$, respectively, in \eqref{eq:GlCat}, are
\begin{align} 
\label{eq:GlCat-a0}
za_{0,1}^2(z)
+z^{2^{\al+1}+1}a_{2^{\al+1},1}^2(z)+a_{0,1}(z)+1=0\quad 
&\text {modulo }2,
\\[2mm]
\label{eq:GlCat-aviel}
za_{2^{\al+1},1}^2(z)+a_{2^{\al+1},1}(z)=0\quad 
&\text {modulo }2.
\end{align}
The only solutions to \eqref{eq:GlCat-aviel} are
$
a_{2^{\al+1},1}(z)=0
$ modulo~$2$,
respectively
$
a_{2^{\al+1},1}(z)=z^{-1}
$ modulo~$2$. The first option is impossible, since it would imply
that, modulo~$2$, the series $C(z)$ reduces to a polynomial; a
contradiction to the well-known fact (easily derivable from
Legendre's formula \cite[p.~10]{LegeAA} for the $p$-adic valuation
of factorials; cf.\ \eqref{eq:Leg}) 
that the Catalan number $\Cat_n$ is odd if, and only if,
$n=2^k-1$ for some $k$. Thus,
\begin{equation*} %\label{eq:Cataviel}
a_{2^{\al+1},1}(z)=z^{-1}\quad \text {modulo }2.
\end{equation*}
Use of this result in \eqref{eq:GlCat-a0} yields the congruence
\begin{equation} \label{eq:Cata0}
za_{0,1}^2(z)+a_{0,1}(z)
+z^{2^{\al+1}-1}+1=0\quad 
\text {modulo }2
\end{equation}
for $a_{0,1}(z)$. We let
$$
a_{0,1}(z)=\widetilde a_{0,1}(z)+\sum _{k=0} ^{\al}z^{2^k-1}
$$
and substitute this in \eqref{eq:Cata0}. Thereby, we get
$$
z\widetilde a_{0,1}^2(z)+\sum _{k=0} ^{\al}z^{2^{k+1}-1}
+\widetilde a_{0,1}(z)+\sum _{k=0} ^{\al}z^{2^k-1}
+z^{2^{\al+1}-1}+1=0\quad 
\text {modulo }2,
$$
or, after simplification,
$$
z\widetilde a_{0,1}^2(z)+\widetilde a_{0,1}(z)=0\quad 
\text {modulo }2.
$$
Again, either $\widetilde a_{0,1}(z)=0$ modulo~$2$, or
$\widetilde a_{0,1}(z)=z^{-1}$ modulo~$2$. Here, the second option is
impossible, since it would imply that $C(z)$ contains a negative
$z$-power, which is absurd. 

In summary, we have found that
\begin{align*}
a_{0,1}(z)&=\sum _{k=0} ^{\al}z^{2^k-1}\quad \text {modulo }2,\\[2mm]
a_{2^{\al+1},1}(z)&=z^{-1}\quad \text {modulo }2,
\end{align*}
with all other $a_{i,1}(z)$ vanishing, forms the unique
solution modulo $2$ in Laurent polynomials $a_{i,1}(z)$
to the system of congruences resulting from \eqref{eq:GlCat}.

\medskip
After we have completed the ``base step," we now proceed with the
iterative steps described in Section~\ref{sec:method}. We consider
the Ansatz \eqref{eq:Ansatz2}--\eqref{eq:Ansatz2b}, where the
coefficients $a_{i,\be}(z)$ are supposed to provide a solution
$C_{\be}(z)=\sum _{i=0} ^{2^{\al+2}-1}a_{i,\be}(z)\Phi^i(z)$ to
\eqref{eq:CatEF} modulo~$2^\be$. This Ansatz, substituted in
\eqref{eq:CatEF}, produces the congruence
\begin{equation*}
zC_{\be}^2(z)-C_{\be}(z)
+2^\be\sum _{i=0} ^{2^{\al+2}-1}b_{i,\be+1}(z)\Phi^i(z)
+1=0
\quad 
\text {modulo }2^{\be+1}.
\end{equation*}
By our assumption on $C_{\be}(z)$, we may divide by $2^\be$.
Comparison of powers of $\Phi(z)$ then yields a system of congruences
of the form
$$
b_{i,\be+1}(z)+\text {Pol}_i(z)=0\quad 
\text {modulo }2,\quad \quad 
i=0,1,\dots,2^{\al+2}-1,
$$
where $\text {Pol}_i(z)$, $i=0,1,\dots,2^{\al+2}-1$, are certain
Laurent polynomials with integer coefficients. This system being trivially
uniquely solvable, we have proved that, for an arbitrary positive
integer $\al$, the algorithm of
Section~\ref{sec:method} will produce a solution $C_{{3\cdot
2^\al}}(z)$ to \eqref{eq:CatEF} modulo $2^{3\cdot 2^\al}$ which is a
polynomial in $\Phi(z)$ with coefficients that are Laurent polynomials in
$z$.
\end{proof}

For example, our computer program needs only about 30 seconds to 
come up with the corresponding congruence modulo $2^{3\cdot 2^2}=4096$.

{
\allowdisplaybreaks
\begin{theorem} \label{thm:Cat4096}
Let $\Phi(z)=\sum _{n\ge0} ^{}z^{2^n}$.
Then we have
\begin{multline} \label{eq:Cat4096}
\sum _{n=0} ^{\infty}\Cat_nz^n=
2048 {z^{14}} + 
3072 {z^{13}} + 
2048 {z^{12}} + 
3584 {z^{11}} + 
640 {z^{10}} + 
2240 {z^9} + 
32 {z^8} \\[2mm]
+ 832 {z^7} + 
2412 {z^6} + 
1042 {z^5} + 
2702 {z^4} + 
53 {z^3} + 
2 {z^2} + 
z + 
1 
\\[2mm]
\kern-3cm +  \left( 
2048 {z^{12}} +
3840 {z^{10}} + 
2112 {z^8} + 
2112 {z^7} + 
552 {z^6} \right. \\[2mm] 
\left.+ 3128 {z^5} + 
2512 {z^4} + 
4000 {z^3} + 
3904 {z^2} 
\right)  
   \Phi(z) \\[2mm]
+ \left( 
2048 {z^{13}} +
3072 {z^{11}} + 
1536 {z^{10}} + 
1152 {z^9} + 
1024 {z^8} + 
4000 {z^7} + 
3440 {z^6} \right. \\[2mm] 
\left.+ 3788 {z^5} + 
3096 {z^4} + 
3416 {z^3} + 
2368 {z^2} + 
288 z 
\right) 
{{\Phi^2(z)}}\\[2mm]
 +   \left( 
2048 {z^{11}} +
2048 {z^{10}} + 
2304 {z^9} + 
512 {z^8} + 
2752 {z^7} + 
3072 {z^6} + 
728 {z^5} \right. \\[2mm] 
\left.+ 3528 {z^4} + 
1032 {z^3} + 
3168 {z^2} + 
3456 z + 
3904 
\right)  {{\Phi^3(z)}}\\[2mm]
 + 
  \left( 
2048 {z^{12}} +
3072 {z^{11}} + 
1024 {z^{10}} + 
2048 {z^9} + 
1152 {z^8} + 
1728 {z^7} + 
2272 {z^6} + 
2464 {z^5} \right. \\[2mm] 
\left.+ 3452 {z^4} + 
3154 {z^3} + 
2136 {z^2} + 
3896 z + 
1600 + 
{\frac {48} z} 
\right)  
   {{\Phi^4(z)}}\\[2mm]
\kern-1.5cm + \left( 
2048 {z^{10}} +
2048 {z^9} + 
1792 {z^8} + 
1792 {z^7} + 
1088 {z^6} + 
1536 {z^5} \right. \\[2mm] 
\left.+ 1704 {z^4} + 
3648 {z^3} + 
3288 {z^2} + 
200 z + 
3728 +{\frac {2272} z} 
\right)  {{\Phi^5(z)}}\\[2mm]
 + 
  \left( 
2048 {z^{11}} 
1024 {z^9} + 
1536 {z^8} + 
3200 {z^7} + 
2816 {z^6} + 
1312 {z^5} + 
3824 {z^4} \right. \\[2mm] 
\left.+ 140 {z^3} + 
592 {z^2} + 
3692 z + 
488 + 
{\frac {2760} z} 
\right)  {{\Phi^6(z)}} \\[2mm]
+ 
  \left( 
2048 {z^9} +
2304 {z^7} + 
2304 {z^6} + 
3520 {z^5} + 
960 {z^4} + 
2456 {z^3} \right. \\[2mm] 
\left.+ 2128 {z^2} + 
2936 z + 
1784 + 
{\frac {4024} z} 
\right)  {{\Phi^7(z)}}\\[2mm]
 + \left( 
2048 {z^{10}} +
1024 {z^9} + 
2048 {z^8} + 
512 {z^7} + 
3968 {z^6} + 
1088 {z^5} + 
1888 {z^4} \right. \\[2mm] 
\left.+ 832 {z^3} + 
1444 {z^2} + 
2646 z + 
3258 + 
{\frac {339} z} 
\right)  
   {{\Phi^8(z)}}\\[2mm]
\kern-3cm + \left( 
2048 {z^8} +
3328 {z^6} + 
1536 {z^5} + 
3008 {z^4} \right. \\[2mm] 
\left. + 
320 {z^3} + 
2168 {z^2} + 
1144 z + 
3992 + 
{\frac {3152} z} 
\right)  {{\Phi^9(z)}}\\[2mm]
\kern-3cm + \left( 
2048 {z^9} +
3072 {z^7} + 
512 {z^6} + 
1408 {z^5} + 
2560 {z^4} \right. \\[2mm] 
\left. + 3424 {z^3} + 
3408 {z^2} + 
1316 z + 
3608 + 
{\frac {2380} z} 
\right)  {{\Phi^{10}(z)}}\\[2mm]
\kern-2cm + 
  \left( 
2048 {z^7} +
2048 {z^6} + 
2816 {z^5} + 
3072 {z^4} + 
1856 {z^3} \right. \\[2mm] 
\left. + 
2688 {z^2} + 
1288 z + 
3880 + 
{\frac {3904} z} 
\right)  
   {{\Phi^{11}(z)}}\\[2mm]
\kern-3cm + \left( 
2048 {z^8} +
1024 {z^7} + 
3072 {z^6} + 
2048 {z^5} + 
1408 {z^4} \right. \\[2mm] \left. +
2624 {z^3} + 
1440 {z^2} + 
224 z + 
948 + 
{\frac {358} z} 
\right)  {{\Phi^{12}(z)}}\\[2mm]
 + 
  \left( 
2048 {z^6} +
2048 {z^5} + 
3328 {z^4} + 
2816 {z^3} + 
1984 {z^2} + 
384 z + 
2488 + 
{\frac {2384} z} 
\right)  {{\Phi^{13}(z)}}\\[2mm]
 + 
  \left( 
2048 {z^7} +
1024 {z^5} + 
512 {z^4} + 
2432 {z^3} + 
1792 {z^2} + 
3040 z + 
336 + 
{\frac {260} z} 
\right)  {{\Phi^{14}(z)}}\\[2mm] 
+  \left( 
2048 {z^5} +
768 {z^3} + 
256 {z^2} + 
64 z + 
2752 + 
{\frac {2696} z} 
\right)  {{\Phi^{15}(z)}}\\[2mm]
\text{\em modulo }4096.
\end{multline}
\end{theorem}
}

The reader is reminded that coefficient extraction from an expression
such as the one on the right-hand side of \eqref{eq:Cat4096} is straightforward,
via the algorithm described in Section~\ref{sec:extr} (see \eqref{eq:Phipot}
and the proof of Lemma~\ref{lem:Hbi}).

\section{A non-example}
\label{sec:nonex}

Consider the equation
\begin{equation} \label{eq:EqH7}
z\FF^6(z)-\FF(z)+1=0,
\end{equation}
which has a unique formal power series solution $\FF(z)$.
We note that the coefficients in the series are special instances of
numbers that are now commonly known as {\it Fu\ss--Catalan numbers}, 
which have
numerous combinatorial interpretations; cf.\ \cite[pp.~59--60]{ArmDAA}.
It was shown in \cite[Eq.~(36)]{MuHecke} that the coefficient of $z^\la$ in 
the series $\FF(z)$ has the same parity as the number of free
subgroups of index $14\la$ in the Hecke group $\mathfrak H(7)=C_2*C_7$. 

If we try our method from Section~\ref{sec:method}, then 
already at the mod-$2$ level we fail:
let $\FF(z)=a_1(z)\Phi(z)+a_0(z)$~modulo~$2$, 
for some Laurent polynomials $a_0(z)$
and $a_1(z)$. Upon substitution in \eqref{eq:EqH7} and simplification
using (cf.\ Proposition~\ref{prop:minpol})
$$
\Phi^2(z)+\Phi(z)+1=0 \text{ modulo }2,
$$
we obtain
\begin{multline*}
z a_0^6(z) + z a_0^4(z)
   a_1^2(z) 
   +a_0(z)+z a_1^6(z)+1\\[2mm]
   +\Phi(z) \left(
   z a_0^4(z)
   a_1^2(z) +
   z a_0^2(z) a_1^4(z) 
   +a_1(z)\right)=0
\quad \text{modulo }2.
\end{multline*}
or, equivalently,
\begin{align} 
\label{eq:a0a1}
z a_0^6(z) + z a_0^4(z)
   a_1^2(z) 
+z a_1^6(z)
   +a_0(z)
+1&=0
\quad \text {modulo }2\\[2mm]
\label{eq:a0a12}
   z a_0^4(z)
   a_1^2(z) +
   z a_0^2(z) a_1^4(z) 
   +a_1(z)&=0
\quad \text {modulo }2.
\end{align}
However, this congruence has no solution in Laurent polynomials
$a_0(z)$ and $a_1(z)$. For, the Laurent polynomials
$a_0^6(z)$,
$a_0^4(z)a_1^2(z)$,
$a_1^6(z)$, all of them being squares, contain only even powers 
of $z$ when the coefficients are reduced modulo~$2$. Consequently, 
the term $a_0(z)+1$ on the left-hand side of \eqref{eq:a0a1} 
can only contain even powers (modulo~$2$). In particular, $a_0(z)$ must
contain the term $1$. If we now suppose that $a_0(z)$ and/or $a_1(z)$
contain negative powers of $z$, then we obtain a contradiction
regardless whether the orders (the minimal $e$ such that $z^e$
appears in a Laurent polynomial) of $a_0(z)$ and $a_1(z)$ are the
same or not. This implies that both $a_0(z)$ and $a_1(z)$
are actually polynomials in~$z$, with $a_0(z)$ being of the form 
$a_0(z)=1+\widetilde a_0(z)$, where $\widetilde a_0(z)$ is a
polynomial without constant term. 
%Hence, the left-hand side of
%\eqref{eq:a0a1} equals
%\begin{equation} \label{eq:a0a1A}
%z+z\widetilde a_0^6(z) + z 
%   a_1^2(z) + z\widetilde a_0^4(z)
%   a_1^2(z) 
%+z a_1^6(z)
%   +\widetilde a_0(z).
%\end{equation}
If we now multiply both sides of \eqref{eq:a0a1} by $a_1^2(z)$ and 
both sides of \eqref{eq:a0a12} by $a_0^2(z)$, and subsequently add the 
two congruences, then we obtain
$$
z a_1^8(z)
   +a_0(z)a_1^2(z)
   +a_0^2(z)a_1(z)
+a_1^2(z)=0
\quad \text {modulo }2.
$$
Dividing by $a_1(z)$
and replacing $a_0(z)$ by $1+\widetilde a_0(z)$, we obtain the
equivalent congruence
$$
z a_1^7(z)
   +\widetilde a_0(z)a_1(z)
   +\widetilde a_0^2(z)+1=0
\quad \text {modulo }2.
$$
This congruence has no solution since $\widetilde a_0(z)$ has no
constant term modulo $2$.

\medskip
In the next theorem, we reveal the deeper reason why our method must
fail for $\FF(z)$. Namely, it shows that exponents $e$ of terms $z^e$
which survive in $\FF(z)$ after reduction of its coefficients modulo
$2$ may have an arbitrary number of blocks of consecutive $1$'s. 
In contrast, a polynomial in
$\Phi(z)$ of degree $d$ with coefficients that are 
Laurent polynomials in $z$ can only have terms $z^e$, 
%after reduction of the coefficients of the
%series modulo $2$, 
where $e$ contains at most $d$ blocks of consecutive $1$'s,
apart from a right-most block of bounded length.\footnote{The length
of this right-most block is in fact bounded by the maximal
modulus of an exponent of $z$ occurring in a Laurent polynomial
coefficient.} 

\begin{theorem} \label{thm:C2C7}
Let $\FF(z)$ be the unique formal power series solution to the
functional equation \eqref{eq:EqH7}. Then the coefficient of $z^n$ in
$\FF(z)$ is odd if, and only if, the sequence of binary digits of $n$
is built by concatenating (in any order) blocks of $0011$ and\/ $01$.
In particular, the number of free subgroups of index $n$ in $\mathfrak
H(7)=C_2*C_7$ is odd if, and only if, the above condition holds.
\end{theorem}

\begin{proof}
By replacing $\FF(z)$ by $1+\GG(z)$ in \eqref{eq:EqH7}, we obtain
$$
z\left(1+\GG(z)\right)^6-\GG(z)=0,
$$
or, equivalently,
$$
z=\frac {\GG(z)} {(1+\GG(z))^6},
$$
so that $\GG(z)$ is the compositional inverse of the series
$z/(1+z)^6$. By the Lagrange inversion formula (cf.\
\cite[Theorem~5.4.2 with $k=1$]{StanBI}), we obtain for $n\ge1$ that
\begin{align*}
\coef{z^n}\FF(z)
&=\coef{z^n}\GG(z)=\frac {1} {n}\coef{z^{-1}}\frac {(1+z)^{6n}}
{z^n}\\[2mm]
&=\frac {1} {n}\coef{z^{n-1}}{(1+z)^{6n}} 
=\frac {1} {n}\binom {6n}{n-1}
=\frac {1} {6n+1}\binom {6n+1}{n}.
\end{align*}
By the well-known theorem of Legendre \cite[p.~10]{LegeAA}
(cf.\ \eqref{eq:Leg}), 
we see that the coefficient of
$z^n$ in $\FF(z)$ is odd if, and only if,
\begin{equation} \label{eq:s65}
s(6n+1)-s(5n+1)-s(n)=0,
\end{equation}
where, as in Section~\ref{sec:Phi}, $s(m)$ denotes the binary digit sum of $m$.
Another way to phrase \eqref{eq:s65} is to say that, whenever we find a
$1$ in the binary expansion of $n$, then there must also be a $1$ in
the binary expansion of $6n+1$ at the same digit place.

We are now ready to establish the claim of the theorem.
In view of the above considerations, it suffices to show that the
condition on $n$ in the statement of the theorem is 
equivalent to \eqref{eq:s65}.

\medskip
Let $n$ be a positive integer with the property that its binary
expansion is formed by concatenating blocks of the form $0011$ and
$01$. We prove that \eqref{eq:s65} holds in this case by induction
on $n$. It is routine to check that our assertion 
holds true for $n=1,2,\dots,15$.
Now, let $n=4n_1+1$, with some positive integer $n_1$. In other words,
the right-most digits in the binary expansion of $n$ are $01$ and
the binary expansion of $n_1$ is formed by 
concatenating blocks of the form $0011$ and $01$.
In that case, we have
\begin{align} \notag
s(6n+1)-s(5n+1)&-s(n)\\[2mm]
\notag
&=s\left(4(6n_1+1)+2+1\right)-s\left(4(5n_1+1)+2\right)-s(4n_1+1)\\[2mm]
&=s(6n_1+1)+2-s(5n_1+1)-1-s(n_1)-1=0,
\label{eq:4n+1}
\end{align}
by the induction hypothesis applied to $n_1$.
On the other hand, if $n=16n_1+3$ for some positive integer $n_1$,
that is, if the right-most digits in the binary expansion of $n$ are
$0011$ and the binary expansion of $n_1$ is formed by 
concatenating blocks of the form $0011$ and
$01$, then
\begin{align} \notag
s(6n+1)-s(5n+1)&-s(n)\\[2mm]
\notag
&=s\left(16(6n_1+1)+2+1\right)-s\left(16(5n_1+1)\right)-s(16n_1+2+1)\\[2mm]
&=s(6n_1+1)+2-s(5n_1+1)-s(n_1)-2=0,
\label{eq:6n+1}
\end{align}
establishing again the truth of \eqref{eq:s65}.

\medskip
In order to prove the converse, let us suppose that $n$ satisfies
\eqref{eq:s65}. We start by 
showing that the binary expansion of $n$ cannot contain
any of the substrings $111$, $000$, $1011$. We call an occurrence of any
of these substrings a ``violation." 

Assuming that the right-most violation is a substring of the form $111$,
we have
\begin{align*}
n&=\hphantom{11}\dots \mathbf1110 \dots,\\[1mm]
2n&=\hphantom{0}\dots 1110 \dots0,\\[1mm]
4n&={}\dots 1110 \dots00,\\[1mm]
6n+1&=\hphantom{1}\dots 1\mathbf01\hphantom{0} \dots1,
\end{align*}
since to the right of the substring $111$ in $n$ there are only blocks
of the form $0011$ and $01$ according to our assumption, which implies
that there cannot be any carries ``destroying" the substring $101$ in
$6n+1$. However, this means that at the place where we find the
bold-face $1$ in (the binary expansion of) $n$ we find a $0$
in (the binary expansion of) $6n+1$, a contradiction to
\eqref{eq:s65}.

Now we assume that the right-most violation is a substring of the form
$000$. In that case, we have
\begin{align*}
n&=\hphantom{11}\dots \mathbf1000 \dots,\\[1mm]
2n&=\hphantom{0}\dots 1000 \dots0,\\[1mm]
4n&={}\dots 1000 \dots00,\\[1mm]
6n+1&=\hphantom{1}\dots 1\mathbf0\hphantom{00} \dots1,
\end{align*}
a contradiction to \eqref{eq:s65} for the same reason.

Finally we assume that the right-most violation is a substring of the
form $1011$. Then we have
\begin{align*}
n&=\hphantom{11}\dots \mathbf10110 \dots,\\[1mm]
2n&=\hphantom{0}\dots 10110 \dots0,\\[1mm]
4n&={}\dots 10110 \dots00,\\[1mm]
6n+1&=\hphantom{1}\dots 0\mathbf001\hphantom{0} \dots1,
\end{align*}
since to the right of the substring $1011$ in $n$ there are only blocks
of the form $0011$ and $01$ according to our assumption, which implies
that there cannot be any carries ``destroying" the substring $0001$ in
$6n+1$. However, this means that at the place where we find the
bold-face $1$ in $n$ we find a $0$ in $6n+1$, again a contradiction to
\eqref{eq:s65}.

Now let us suppose that $n=2^\al n_1+n_0$, where $n_0$ is an $\al$-digit
(binary) number formed by concatenating blocks of the form $0011$ and
$01$. By applying the computations \eqref{eq:4n+1} and \eqref{eq:6n+1}
(possibly several times), 
one sees that $n$ satisfies \eqref{eq:s65} if, and only if, $n_1$ does. 

We claim that $n_1$ cannot be even. For, if it were, say $n_1=2^\be n_2$,
then $n_1$ has a $1$ at the digit place $\be$ while $6n_1+1$ has not,
a contradiction to \eqref{eq:s65}. But then the already established fact 
that $n$, and hence $n_1$, cannot contain any of the substrings 
$111$, $000$, $1011$ implies that the right-most digits
in the binary expansion of $n_1$
form either a block $01$ or a block $0011$. This provides an
inductive argument that the binary expansion of $n$ 
is formed by concatenating blocks of the form $0011$ and
$01$, and thus completes the proof of the theorem.
\end{proof}

\section{Free subgroups in lifts of Hecke groups}
\label{sec:free}

For integers $m,q$ with $m\geq1$, $q\geq3$, and $q$ prime, we consider
the group $\Gamma_m(q)$ as defined in \eqref{eq:Heckelift}.
Denote by $f_\lambda^{(q)}(m)$ the number of free subgroups of index
$2qm\lambda$ in $\Gamma_m(q)$. 
The purpose of this section is to estimate the $2$-adic valuation of 
$f_\lambda^{(q)}(m)$ in the case when $q$ is a Fermat prime. 
This estimate is
based on a recurrence relation for these numbers, which, in turn,
results from a specialisation of a differential equation in
\cite[Sec.~2]{MuHecke}. Moreover, this differential equation for the
generating function of free subgroup numbers in $\Ga_m(q)$
will become of crucial importance in
Sections~\ref{sec:free2} and \ref{sec:Hecke}. 

In order to present the aforementioned differential equation, we
first need to compute several important invariants of $\Ga_m(q)$.
Using notation and definitions from \cite[Sec.~2]{MuHecke}, 
we have $m_{\Gamma_m(q)} =
2qm$, $\chi(\Gamma_m(q)) = -\frac{q-2}{2qm}$, and thus, 
\[
\mu(\Gamma_m(q)) = 1 - m_{\Gamma_m(q)} \chi(\Gamma_m(q)) =q-1 =
\mu(\mathfrak{H}(q)). 
\]
Moreover, for the family of zeta-invariants
$\{\zeta_\kappa(\Gamma_m(q)):\kappa\mid 2qm\}$
of $\Gamma_m(q)$, we find that
\begin{equation}
\label{Eq:ZetaGamma_m(q)}
\zeta_\kappa(\Gamma_m(q)) = \begin{cases}
   1,& \kappa=m,\\[1mm]
   -1,& \kappa=2qm,\\[1mm]
   0,& \mbox{otherwise.}
   \end{cases}
\end{equation}
Our first result in this section
compares the $A$-invariants of $\Gamma_m(q)$, as
defined in \cite[Eq.~(14)]{MuHecke}, to those of the underlying Hecke
group $\mathfrak{H}(q)$, and provides an estimate for their
$2$-adic valuation. 
\begin{lemma}
\label{Lem:ACompare}
\begin{enumerate}
\item For an integer $m\ge1$ and a prime $q\ge3$, we have
\[
A_\mu(\Gamma_m(q)) = m^{q-1} A_\mu(\mathfrak{H}(q)),\quad 0\leq\mu\leq q-1.
\]
\item We have 
\[
v_2(A_\mu(\mathfrak{H}(q)) \geq \mu,\quad 0\leq \mu\leq q-1.
\]
\end{enumerate}
\end{lemma}
\begin{proof}
(i) In view of (\ref{Eq:ZetaGamma_m(q)}), we have, for $0\leq \mu\leq
q-1$, that 
{\allowdisplaybreaks
\begin{align*}
A_\mu(\Gamma_m(q)) &= \frac{1}{\mu!} \sum_{j=0}^\mu (-1)^{\mu-j}
\binom{\mu}{j} 2qm(j+1) \left(\vphantom{\sum_{g_{g_g}}}\right.\underset{(k,2qm)=m}{\prod_{1\leq k\leq
2qm}} (2qmj+k)\left.\vphantom{\sum_{g_{g_g}}}\right) 
\big(2qm(j+1)\big)^{-1}\\[1mm] 
&= \frac{1}{\mu!} \sum_{j=0}^\mu (-1)^{\mu-j} \binom{\mu}{j}
\underset{(k,2qm)=m}{\prod_{1\leq k\leq 2qm}} (2qmj+k)\\[1mm] 
&= \frac{1}{\mu!} \sum_{j=0}^\mu (-1)^{\mu-j} \binom{\mu}{j}
\underset{(k',2q)=1}{\prod_{1\leq k'\leq 2q}}
\big[m(2qj+k')\big]\\[1mm] 
&= m^{\varphi(2q)} \frac{1}{\mu!} \sum_{j=0}^\mu (-1)^{\mu-j}
\binom{\mu}{j} \underset{(k',2q)=1}{\prod_{1\leq k'\leq 2q}}
(2qj+k')\\[1mm] 
&= m^{q-1} A_\mu(\mathfrak{H}(q)),
\end{align*}}%
as claimed.

(ii) This follows from \cite[Lemma~1]{MuHecke}.
\end{proof}

Our next result is a recurrence
relation for the subgroup numbers $f_{\lambda}^{(q)}(m)$. 

\begin{lemma}
\label{Lem:Gamma_m(q)FreeRec}
For $m, q$ as in Lemma~{\em\ref{Lem:ACompare}} and  $\lambda\geq1,$ we have
\begin{multline}
\label{Eq:fqmRec}
f_{\lambda+1}^{(q)}(m) = \sum_{\mu=1}^{q-1} \sum_{\nu=1}^\mu
\underset{\mu_1+\cdots+\mu_\nu=\mu}
{\sum_{\mu_1,\ldots,\mu_\nu>0}}
\underset{\lambda_1+\cdots+\lambda_\nu=\lambda-\mu}
{\sum_{\lambda_1,\ldots,\lambda_\nu\geq0}}
\binom{\mu}{\mu_1,\ldots,\mu_\nu} \big(\nu!\, (2q)^\nu\big)^{-1}
m^{q-\nu-1} A_\mu(\mathfrak{H}(q))\\[2mm] \times
\prod_{j=1}^\nu\Big[(\mu_j-1)! \binom{\lambda_j+\mu_j-1}{\mu_j-1}\Big]
\prod_{j=1}^\nu f^{(q)}_{\lambda_j+\mu_j}(m) ,
\end{multline}
with initial value $f_1^{(q)}(m) = m^{q-1} A_0(\mathfrak{H}(q))$.
\end{lemma}
\begin{proof}
Setting $\mathfrak{G} =  \Gamma_m(q)$ in \cite[Eq.~(18)]{MuHecke},
and using Part~(i) of Lemma~\ref{Lem:ACompare} to compute
$A_\mu(\Gamma_m(q))$ in terms of $A_\mu(\mathfrak{H}(q))$, leads to the
differential equation 
\begin{multline}
\label{Eq:FmqDiffEq}
\kern-2pt
G_{m}(q;z) = m^{q-1} A_0(\mathfrak{H}(q))\\[2mm]
 + \sum_{\mu=1}^{q-1}
\sum_{\nu=1}^\mu
\underset{\mu_1+\cdots+\mu_\nu=\mu}{\sum_{\mu_1,\ldots,\mu_\nu>0}}
\binom{\mu}{\mu_1,\ldots,\mu_\nu}\big(\nu!\, (2q)^\nu\big)^{-1}
m^{q-\nu-1} A_\mu(\mathfrak{H}(q)) z^\mu \prod_{j=1}^\nu
\big(G_{m}(q;z)\big)^{(\mu_j-1)} 
\end{multline} 
for the generating function 
$G_{m}(q;z):= \sum_{\lambda\geq0}
f_{\lambda+1}^{(q)}(m) z^\lambda$.
Comparing the coefficient of
$z^\lambda$ in (\ref{Eq:FmqDiffEq}) for $\lambda\geq1$ yields
(\ref{Eq:fqmRec}); while, for $\lambda=0$, we obtain the required
initial value $f_1^{(q)}(m)$. 
\end{proof}
Given these preparations, we can now show the following estimate for
the $2$-adic valuation of $f_\lambda^{(q)}(m)$ in the case when $q$ is a Fermat prime. 
\begin{proposition}
\label{Prop:flambdaq2Part}
{\em (i)} Let $m\geq1$ be an integer, and let $q\geq3$ be a Fermat prime. Then we have
\begin{equation}
\label{Eq:flambdaq2Part}
v_2(f_\lambda^{(q)}(m)) \geq v_2(m)(\lambda+q-2),\quad \lambda\geq1.
\end{equation}
In particular, if $m$ is even, then $f_\lambda^{(q)}(m)$ is zero
modulo any given $2$-power for all sufficiently large values of
$\lambda$.  

{\em (ii)} For $q=3,$ equality occurs in Inequality~{\em
(\ref{Eq:flambdaq2Part})} if, and only if, $\lambda+1$ is a $2$-power.

\end{proposition}
\begin{proof}
(i) Since (\ref{Eq:flambdaq2Part}) is trivially true for $m$ odd, we
may suppose that $v_2(m)>0$. We use induction on $\lambda$. For
$\lambda=1$, we have  
\[
v_2(f_1^{(q)}(m)) = v_2(m^{q-1} A_0(\mathfrak{H}(q))) \geq (q-1)
v_2(m) = (\lambda+q-2)v_2(m),
\]
as desired. Now suppose that our claim (\ref{Eq:flambdaq2Part})
holds for all $f_\gamma^{(q)}(m)$ such that $\gamma < L$, with some
integer $L\geq2$, and consider an arbitrary summand  
\[
S = S(\mu, \nu, \mu_1,\ldots,\mu_\nu, \lambda_1,\ldots,\lambda_\nu)
\] 
in the recurrence relation (\ref{Eq:fqmRec}) with $\lambda=L-1$. We find that  
\begin{align*}
v_2(S) &\geq v_2\binom{\mu}{\mu_1,\ldots,\mu_\nu} - v_2(\nu!) - \nu +
(q-\nu-1)v_2(m) \\[1mm] &\hspace{7cm}+ v_2(A_\mu(\mathfrak{H}(q))) +
\sum_{j=1}^\nu v_2(f^{(q)}_{\lambda_j+\mu_j}(m))\\[2mm]  
&\geq v_2\binom{\mu}{\mu_1,\ldots,\mu_\nu} - v_2(\nu!) - \nu + (q-\nu-1) v_2(m) + \mu\\[1mm] 
&\hspace{7cm}+ \sum_{j=1}^\nu (\lambda_j+\mu_j+q-2) v_2(m)\\[2mm]  
&\ge v_2\binom{\mu}{\mu_1,\ldots,\mu_\nu} - v_2(\nu!) - \nu + (q-\nu-1)
v_2(m) + \mu\\[1mm] 
&\hspace{7cm}+ (L-1)v_2(m) + (q-2)\nu v_2(m)\\[2mm] 
&\geq (L+q-2) v_2(m) + (q-3)\nu v_2(m) +
v_2\binom{\mu}{\mu_1,\ldots,\mu_\nu} + \mu - \nu - v_2(\nu!), 
\end{align*}
where we have used Part~(ii) of Lemma~\ref{Lem:ACompare} plus the
induction hypothesis in the second 
step. Since $\nu\leq\mu$, the desired inequality for the $2$-adic
valuation of
$f_L^{(q)}(m)$ 
will follow, if we can show that 
\begin{equation}
\label{Eq:NecEst}
v_2(\nu!) \leq (q-3)\nu v_2(m) + v_2\binom{\mu}{\mu_1,\ldots,\mu_\nu}.
\end{equation} 
Since $q$ is a Fermat prime, we have $q-1=2^\alpha$ for some
$\alpha\geq1$. Thus, if $\nu<q-1$, then, by Legendre's formula for the
$p$-adic valuation of factorials (cf.\ \cite[p.~10]{LegeAA}), we get  
\[
v_2(\nu!) \leq \sum_{i\geq0}
\Big\lfloor\frac{2^\alpha-1}{2^i}\Big\rfloor < \sum_{i\geq1}
2^{\alpha-i} = q-2, 
\]
and (\ref{Eq:NecEst}) holds, the left-hand side already being
compensated by the term 
\[
(q-3)\nu\, v_2(m) \geq q-3.
\]
On the other hand, for $\nu=q-1$, we have $\mu=\nu=q-1$,
$\mu_1=\cdots=\mu_{q-1}=1$,  $v_2(\nu!)=q-2$, and  
\[
v_2\binom{\mu}{\mu_1,\ldots,\mu_\nu} = v_2((q-1)!) = q-2,
\] 
and the desired conclusion holds again. We have thus shown that every
summand $S$ on the right-hand side of (\ref{Eq:fqmRec}) satisfies 
$v_2(S) \geq (L+q-2) v_2(m)$, which implies that  
\[
v_2(f_L^{(q)}(m)) \geq (L+q-2) v_2(m), 
\]
completing the induction. 

(ii) For $q=3$, the recurrence relation~(\ref{Eq:fqmRec}), with
$\lambda$ replaced by $\lambda-1$,  takes the form 
\begin{equation}
\label{Eq:FreeRecq=3}
f_\lambda^{(3)}(m) = 6 m \lambda f_{\lambda-1}^{(3)}(m) +
\underset{\mu+\nu=\lambda-1}{\sum_{\mu,\nu\geq1}} f_\mu^{(3)}(m)
f_\nu^{(3)}(m),\quad \lambda\geq2, 
\end{equation}
with initial value $f_1^{(3)}(m) = 5m^2$. In order to establish our
second claim, we rewrite Equation 
(\ref{Eq:FreeRecq=3}) as 
\begin{multline}
\label{Eq:FreeRecVar}
f^{(3)}_\lambda(m) = 6 m \lambda f^{(3)}_{\lambda-1}(m)\, +\, \begin{cases}
  2\sum_{\mu=1}^{\frac{\lambda-2}{2}} f^{(3)}_\mu(m),
f^{(3)}_{\lambda-\mu-1}(m),&\lambda\equiv 0\,\, (2),\\[3mm]  
2\sum_{\mu=1}^{\frac{\lambda-3}{2}} f^{(3)}_\mu(m) 
f^{(3)}_{\lambda-\mu-1}(m)\,+\,
\big(f^{(3)}_{\frac{\lambda-1}{2}}(m)\big)^2,& \lambda \equiv 1\,\, (2) ,
\end{cases}\\[2mm]
\quad \text{for }\lambda\geq2,
\end{multline}
and argue again by induction on $\lambda$. For $\lambda=1$, Inequality
(\ref{Eq:flambdaq2Part}) is sharp, as required. Now suppose that, for
$\lambda<L$ with some $L\geq2$, 
Inequality~\eqref{Eq:flambdaq2Part} with $q=3$ is
sharp if, and only if, $\lambda+1$ is a $2$-power, and consider
$f_L^{(3)}(m)$ as given by (\ref{Eq:FreeRecVar}). Setting $m=2^am'$
with $m'$ odd, we have  
\[
v_2(6mLf^{(3)}_{L-1}(m)) = 1 + a + v_2(L) + v_2(f^{(3)}_{L-1}(m)) \geq a(L+1)+1.
\]
Consequently, if $L$ is even then, by what we have already shown,
$$v_2(f^{(3)}_L(m)) \geq a(L+1)+1.$$ 
For $L$ odd, all terms except possibly
$\big(f^{(3)}_{\frac{L-1}{2}}(m)\big)^2$ are divisible by $2^{a(L+1)+1}$ or a higher
$2$-power, while this exceptional term satisfies 
\[
v_2\big(f^{(3)}_{\frac{L-1}{2}}(m)\big)^2 \geq a(L+1),
\] 
with equality occurring (according to our induction hypothesis) if,
and only if, $\frac{L-1}{2} + 1 = 2^\gamma$ for some $\gamma\geq1$;
that is, if, and only if, $L+1=2^{\gamma+1}$ is a $2$-power. This
completes the induction, and the proof.
\end{proof}

\section{Free subgroup numbers
for lifts of the inhomogeneous modular group}
\label{sec:free2}

In this section, we investigate the behaviour of the numbers
$f_\la^{(3)}(m)$ of free subgroups in lifts of the inhomogeneous
modular group $PSL_2(\mathbb Z)\cong \mathfrak H(3)$ modulo powers
of~$2$. As mentioned in the introduction, the best previous result
available in the literature is \cite[Theorem~1]{MuPu}, which
determines the behaviour of $f_\la^{(3)}(1)$ modulo~$16$.
The results in this section solve the problem
of determining $f_\la^{(3)}(m)$ modulo powers of $2$ not only for $m=1$ and
the $2$-power $2^4=16$, but for {\it all\/} $m$ and modulo {\it any}
power of $2$.

Let $F_m(z):=1+\sum _{\la\ge1} ^{}f_\la^{(3)}(m)\,z^\la$ be the generating
function for these numbers. (In the notation of the previous section, 
$F_m(z)=1+zG_{m}(3;z)$.)
By specialising $q=3$ in
\eqref{Eq:FmqDiffEq}, one obtains the differential equation
%\begin{equation} \label{eq:fmdiff}
%(1-6mz)F_m(z)-6mz^2F_m'(z)-zF_m^2(z)-5m^2z=0.
%\end{equation}
\begin{equation} \label{eq:fmdiff}
(1-(6m-2)z)F_m(z)-6mz^2F_m'(z)-zF_m^2(z)
-1-(1-6m+5m^2)z=0.
\end{equation}

\begin{theorem} \label{thm:freie-m}
Let $\Phi(z)=\sum _{n\ge0} ^{}z^{2^n},$ and let $\al$ be some
positive integer.
Then, for every positive integer $m,$ 
the generating function $F_m(z),$ when reduced modulo $2^{3\cdot 2^\al},$ 
can be expressed as a polynomial in $\Phi(z)$ of degree at most
$2^{\al+2}-1,$ with coefficients that are Laurent polynomials in $z$
over the integers.
\end{theorem}

\begin{proof} In view of Proposition~\ref{Prop:flambdaq2Part}, the assertion
is trivially true for even $m$, the polynomial in $\Phi(z)$ being a
polynomial of degree zero in this case. 
We may thus assume from now on that $m$ is odd.

We apply the method from Section~\ref{sec:method}. We start by
substituting the Ansatz \eqref{eq:Ansatz1} in \eqref{eq:fmdiff} and
reducing the result modulo $2$. 
In this way, we obtain
\begin{equation*} %\label{}
\sum _{i=0} ^{2^{\al+2}-1}a_{i,1}(z)\Phi^i(z)+
z
\sum _{i=0} ^{2^{\al+2}-1}a_{i,1}^2(z)\Phi^{2i}(z)+1=0\quad \text
{modulo }2.
\end{equation*}
This congruence is identical with the congruence \eqref{eq:quadr}.
Hence, we can copy the resulting solution from there. Namely, the
unique solution to \eqref{eq:quadr} (and, hence, to the above
congruence) is given by
\begin{align*}
a_{0,1}(z)&=\sum _{k=0} ^{\al}z^{2^k-1}\quad \text {modulo }2,\\[2mm]
a_{2^{\al+1},1}(z)&=z^{-1}\quad \text {modulo }2,
\end{align*}
with all other $a_{i,1}(z)$ vanishing.

\medskip
After we have completed the ``base  step," we now proceed with the
iterative steps described in Section~\ref{sec:method}. We consider
the Ansatz \eqref{eq:Ansatz2}--\eqref{eq:Ansatz2b}, where the
coefficients $a_{i,\be}(z)$ are supposed to provide a solution
$F_{m,\be}(z)=\sum _{i=0} ^{2^{\al+2}-1}a_{i,\be}(z)\Phi^i(z)$ to
\eqref{eq:fmdiff} modulo~$2^\be$. This Ansatz, substituted in
\eqref{eq:fmdiff}, produces the congruence
\begin{multline*}
2^\be\sum _{i=0} ^{2^{\al+2}-1}b_{i,\be+1}(z)\Phi^i(z)
+(1-(6m-2)z)F_{m,\be}(z)\\[2mm]
-6mz^2F_{m,\be}'(z)-zF_{m,\be}^2(z)-1-(1-6m+5m^2)z=0\quad 
\quad 
\text {modulo }2^{\be+1}.
\end{multline*}
By our assumption on $F_{m,\be}(z)$, we may divide by $2^\be$.
Comparison of powers of $\Phi(z)$ then yields a system of congruences
of the form
$$
b_{i,\be+1}(z)+\text {Pol}_i(z)=0\quad 
\text {modulo }2,\quad \quad 
i=0,1,\dots,2^{\al+2}-1,
$$
where $\text {Pol}_i(z)$, $i=0,1,\dots,2^{\al+2}-1$, are certain
Laurent polynomials with integer coefficients. This system being trivially
(uniquely) solvable, we have proved that, for an arbitrary positive
integer $\al$, the algorithm of
Section~\ref{sec:method} will produce a solution $F_{m,2^{3\cdot
2^\al}}(z)$ to \eqref{eq:fmdiff} modulo $2^{3\cdot 2^\al}$ which is a
polynomial in $\Phi(z)$ with coefficients that are Laurent polynomials in
$z$.
\end{proof}

We have implemented this algorithm. As an illustration, the next
theorem contains the result for the modulus $64$.\footnote{To be
precise, our implementation finds an expression for each fixed $m$. 
These particular results can then be put together ``manually" 
into the uniform expression displayed in \eqref{eq:Loes1}.}

{
\allowdisplaybreaks
\begin{theorem} \label{thm:freie-m64}
Let $\Phi(z)=\sum _{n\ge0} ^{}z^{2^n}$.
Then, for all positive odd integers $m,$ we have
\begin{multline} 
\label{eq:Loes1}
1 + \sum _{\la\ge1} ^{}f^{(3)}_\lambda(m)\,z^\lambda
=32 z^9+48 z^7+32 z^6+(16 m+8) z^5
+(16 m+8) z^4\\[1mm]
+\left(2
   m^2+34\right) z^3 + \left(4
   m^2-4 m+24\right)
   z^2+\left(5 m^2+12\right)
   z + 1
\\[2mm]
+\left(48 z^4+24 z^3+12 z^2+60
   z+40\right) \Phi(z)\\[2mm]
+ \left(16 z^5+(16 m+32)
   z^4+(4 m^2-32
   m+68) z^3+36 z^2+22
   z+12+\frac{12}{z}\right) \Phi^2(z) \\[2mm]
+  \left(32
   z^5+32 z^4+(16 m-16) z^3+40 z^2+4
   z+52+\frac{28}{z}\right)\Phi^3(z)\\[2mm]
+\bigg(32 z^7+32
   z^5+32
   z^4+(16 m+24)
   z^3+(16 m+40) z^2
\kern4cm\\[2mm]
\kern4cm\left.
+(2 m^2+16
   m+38\right) z
+24+\frac{35}{z}\bigg)\Phi^4(z) \\[2mm]
+\left(32
   z^3+16 z^2+(16 m-8) z+44\right)\Phi^5(z) \\[2mm]
+ \left(16
   z^3+16 m z^2+\left(4 m^2-16
   m+20\right) z+44+\frac{50}{z}\right)\Phi^6(z)\\[2mm]
+
   \left(32
   z^3+32
   z^2+(16 m+16) z+40+\frac{4}{z}\right)\Phi^7(z) \quad \quad 
\text {\em modulo }64.
\end{multline}
\end{theorem}
}

\section{Subgroup numbers for the inhomogeneous modular group}
\label{sec:PSL2Z}

For a finitely generated group $\Ga$, let $s_n(\Ga)$ denote the number of 
subgroups of index $n$ in $\Ga$, and write $S_\Ga(z)$ for the (shifted)
generating function $\sum _{n\ge0} ^{}s_{n+1}(\Ga)\,z^n$. 

In this section, we focus on the
sequence $\big(s_n(PSL_2(\Z))\big)_{n\ge1}$ and its generating function
$S(z):=S_{PSL_2(\mathbb Z)}(z)$. We shall show that our method solves
the problem of determining these subgroup numbers modulo any given
power of $2$, thus refining the parity result of Stothers
\cite{Stothers} and the mod-$8$ result from
\cite[Theorem~2]{MuPu} mentioned in the introduction.

By the first displayed equation on top of p.~276 in \cite{GIR} (cf.\ also
\cite[Eq.~(5.29)]{KrMuAC} with  
$H=\{1\}$ and $a=b=h=1$), 
the series $S(z)$ obeys the differential equation
\begin{multline} \label{eq:S1diff}
(-1+4z^3+2z^4+4z^6-2z^7-4z^9)S(z)+(z^7-z^{10})(S'(z)+S^2(z))\\[2mm]
+1+z+4z^2+4z^3-z^4+4z^5-2z^6-2z^8=0.
\end{multline}
The differential equation \eqref{eq:S1diff} has a unique solution
since comparison of coefficients of $z^{N}$ fixes the initial values,
and yields a recurrence for
the sequence $\big(s_n(PSL_2(\mathbb Z))\big)_{n\ge1}$ 
which computes $s_{n+1}(PSL_2(\mathbb Z))$ from terms
involving only $s_i(PSL_2(\mathbb Z))$ with $i\le n$. 

\begin{theorem} \label{thm:Unterg}
Let $\Phi(z)=\sum _{n\ge0} ^{}z^{2^n},$ and let $\al$ be some
positive integer.
Then the generating function $S(z)=S_{PSL_2(\mathbb Z)}(z),$ 
when reduced modulo $2^{3\cdot 2^\al},$ 
can be expressed as a polynomial in $\Phi(z)$ of degree at most
$2^{\al+2}-1$ with coefficients that are Laurent polynomials in $z$
over the integers.
\end{theorem}

\begin{proof} 
We apply the method from Section~\ref{sec:method}. We start by
substituting the Ansatz \eqref{eq:Ansatz1} in \eqref{eq:S1diff} and
reducing the result modulo $2$. 
In this way, we obtain
\begin{multline*} %\label{}
\sum _{i=0} ^{2^{\al+2}-1}a_{i,1}(z)\Phi^i(z)+
(z^7+z^{10})\Bigg(
\sum _{i=0} ^{2^{\al+2}-1}ia_{i,1}(z)\Phi^{i-1}(z)\Phi'(z)+
\sum _{i=0} ^{2^{\al+2}-1}a'_{i,1}(z)\Phi^i(z)\\[2mm]
+
\sum _{i=0} ^{2^{\al+2}-1}a_{i,1}^2(z)\Phi^{2i}(z)\Bigg)
+1+z+z^4=0\quad \text
{modulo }2.
\end{multline*}
We may reduce $\Phi^{2i}(z)$ further using 
Relation~\eqref{eq:PhiRel2}. This leads to
\begin{multline} \label{eq:GlUnterg}
\sum _{i=0} ^{2^{\al+2}-1}a_{i,1}(z)\Phi^i(z)+
(z^7+z^{10})\Bigg(
\sum _{i=0} ^{2^{\al+2}-1}ia_{i,1}(z)\Phi^{i-1}(z)\Phi'(z)+
\sum _{i=0} ^{2^{\al+2}-1}a'_{i,1}(z)\Phi^i(z)\\[2mm]
+\sum _{i=0} ^{2^{\al+1}-1}\left(a_{i,1}^2(z)
+z^{2^{\al+1}}a_{i+2^{\al+1},1}^2(z)\right)\Phi^{2i}(z)
+
\sum _{i=0} ^{2^{\al}-1}
z^{2^{\al+1}}a_{i+3\cdot2^{\al},1}^2(z)\Phi^{2i}(z)\\[2mm]
+
\sum _{i=0} ^{2^{\al}-1}\left(a_{i+2^{\al+1},1}^2(z)
+a_{i+3\cdot2^{\al},1}^2(z)\right)\Phi^{2i+2^{\al+1}}(z)\Bigg)
+1+z+z^4=0\quad \text
{modulo }2.
\end{multline}
In the same way as in the proof of Theorem~\ref{thm:freie-m}, one
sees that all coefficients $a_{i,1}(z)$
vanish modulo~$2$, except possibly $a_{0,1}(z)$ and
$a_{2^{\al+1},1}(z)$. 
The corresponding congruences obtained by extracting coefficients of $\Phi^0(z)$ and
$\Phi^{2^{\al+1}}(z)$, respectively, in \eqref{eq:GlUnterg}, are
\begin{equation}
\label{eq:GlUnterg-a0}
a_{0,1}(z)
+(z^7+z^{10})\left(a'_{0,1}(z)+a_{0,1}^2(z)
+z^{2^{\al+1}}a_{2^{\al+1},1}^2(z)\right)+1+z+z^4=0\quad 
\text {modulo }2
\end{equation}
and
\begin{equation}
\label{eq:GlUnterg-aviel}
a_{2^{\al+1},1}(z)
+(z^7+z^{10})\left(
a'_{2^{\al+1},1}(z)
+a_{2^{\al+1},1}^2(z)
\right)=0\quad 
\text {modulo }2.
\end{equation}
The only solutions to \eqref{eq:GlUnterg-aviel} are
$
a_{2^{\al+1},1}(z)=0
$ modulo~$2$,
respectively
$
a_{2^{\al+1},1}(z)=z^{-7}+z^{-4}
$ modulo~$2$. The first option is impossible, since there is
no Laurent polynomial $a_{0,1}(z)$ solving the
equation resulting from \eqref{eq:GlUnterg-a0}. Thus, we have 
\begin{equation} \label{eq:Untergaviel}
a_{2^{\al+1},1}(z)=z^{-7}+z^{-4}\quad \text {modulo }2.
\end{equation}
Use of this result in \eqref{eq:GlUnterg-a0} yields the congruence
\begin{equation} \label{eq:Unterga0}
a_{0,1}(z)
+z^7(1+z^3)\left(
a'_{0,1}(z)+a_{0,1}^2(z)
\right)
+z^{2^{\al+1}-7}(1+z^3)^3
+1+z+z^4=0\quad 
\text {modulo }2.
\end{equation}
for $a_{0,1}(z)$. We let
$$
a_{0,1}(z)=\widetilde a_{0,1}(z)+z^{-7}+z^{-4}+z^{-3}
+(1+z^3)\sum _{k=2} ^{\al}z^{2^k-7}
$$
and substitute this in \eqref{eq:Unterga0}. Thereby, we get
\begin{multline*}
\widetilde a_{0,1}(z)
+z^7(1+z^3)\left(
\widetilde a'_{0,1}(z)+\widetilde a_{0,1}^2(z)
\right)
+z^{-7}+z^{-4}+z^{-3}
+(1+z^3)\sum _{k=2} ^{\al}z^{2^k-7}
\\[2mm]
+z^7(1+z^3)\left(
z^{-8}+z^{-4}
+\sum _{k=2} ^{\al}z^{2^k-8}
+z^{-14}+z^{-8}+z^{-6}
+(1+z^3)^2\sum _{k=2} ^{\al}z^{2^{k+1}-14}
\right)\\[2mm]
+z^{2^{\al+1}-7}(1+z^3)^3
+1+z+z^4=0\quad 
\text {modulo }2,
\end{multline*}
or, after simplification,
$$
\widetilde a_{0,1}(z)
+z^7(1+z^3)\left(
\widetilde a'_{0,1}(z)+\widetilde a_{0,1}^2(z)
\right)
=0\quad 
\text {modulo }2.
$$
Again, either $\widetilde a_{0,1}(z)=0$ modulo~$2$ or
$\widetilde a_{0,1}(z)=z^{-7}+z^{-4}$ modulo~$2$. Here, the second option is
impossible, since it would imply that $S(z)$ contains a negative
$z$-power, which is absurd. 

In summary, we have found that
\begin{align*}
a_{0,1}(z)&=z^{-7}+z^{-4}+z^{-3}
+(1+z^3)\sum _{k=2} ^{\al}z^{2^k-7}
\quad \text {modulo }2,\\[2mm]
a_{2^{\al+1},1}(z)&=z^{-7}+z^{-4}\quad \text {modulo }2,
\end{align*}
with all other $a_{i,1}(z)$ vanishing, forms the unique
solution modulo $2$ to the system of congruences resulting from 
\eqref{eq:GlUnterg} in Laurent polynomials $a_{i,1}(z)$.
It should be noted that all $a_{i,1}(z)$'s, $1\le i\le
2^{2^{\al+2}}-1$, are divisible by $1-z^3$ modulo $2$, as is $a_{0,1}(z)-1$.

\medskip
After we have completed the ``base  step," we now proceed with the
iterative steps described in Section~\ref{sec:method}. 
The arguments turn out to be slightly more delicate here than in the
proof of Theorem~\ref{thm:freie-m}. To be more precise, when
considering
the Ansatz \eqref{eq:Ansatz2}--\eqref{eq:Ansatz2b}, where,
inductively, the
coefficients $a_{i,\be}(z)$ are supposed to provide a solution
$S_{\be}(z)=\sum _{i=0} ^{2^{\al+2}-1}a_{i,\be}(z)\Phi^i(z)$ to
\eqref{eq:S1diff} modulo~$2^\be$, we must also assume 
that $a_{i,\be}(z)$, $1\le i\le
2^{2^{\al+2}}-1$, and $a_{0,\be}(z)-1-2z^2$, are all divisible by $1-z^3$.
The reader should note that the divisibility assumptions do indeed
hold for $\be=1$, the term $-2z^2$ being negligible, since for $\be=1$
we are computing modulo $2^{\be}=2$.

The above Ansatz, substituted in
\eqref{eq:S1diff}, produces the congruence
\begin{multline} \label{eq:2be}
2^\be\sum _{i=0} ^{2^{\al+2}-1}b_{i,\be+1}(z)\Phi^i(z)\\[2mm]
+2^\be(z^7-z^{10})\Bigg(
\sum _{i=0} ^{2^{\al+2}-1}ib_{i,\be+1}(z)\Phi^{i-1}(z)\Phi'(z)
+\sum _{i=0} ^{2^{\al+2}-1}b'_{i,\be+1}(z)\Phi^i(z)
\Bigg)\\[2mm]
+(-1+4z^3+2z^4+4z^6-2z^7-4z^9)
\sum _{i=0} ^{2^{\al+2}-1}a_{i,\be}(z)\Phi^i(z)\\[2mm]
+(z^7-z^{10})\Bigg(
\sum _{i=0} ^{2^{\al+2}-1}ia_{i,\be}(z)\Phi^{i-1}(z)\Phi'(z)
\kern5cm\\[2mm]
\kern4cm
+\sum _{i=0} ^{2^{\al+2}-1}a'_{i,\be}(z)\Phi^i(z)
+\bigg(\sum _{i=0} ^{2^{\al+2}-1}a_{i,\be}(z)\Phi^i(z)\bigg)^2
\Bigg)\\[2mm]
+1+z+4z^2+4z^3-z^4+4z^5-2z^6-2z^8=0
\quad 
\text {modulo }2^{\be+1}.
\end{multline}
By our inductive construction, we know that the terms contained in 
lines 3--7 of \eqref{eq:2be} are divisible by $2^\be$. Hence, if we
were to divide by $2^\be$ and compare coefficients of $\Phi^i(z)$, for 
$i=0,1,\dots,2^{\al+2}-1$, we would obtain the modular differential
equations
\begin{multline} \label{eq:bdiff}
b_{i,\be+1}(z)+(z^7-z^{10})\big(b'_{i,\be+1}(z)+
(i+1)b_{i+1,\be+1}\Phi'(z)\big)
+\text {Pol}_i(z)=0 \quad 
\text {modulo }2,\\[2mm]
i=0,1,\dots,2^{\al+2}-1,
\end{multline}
where $\text {Pol}_i(z)$, $i=0,1,\dots,2^{\al+2}-1$, are certain
Laurent polynomials with integer coefficients. If $i$ is odd, then
the term $(i+1)b_{i+1,\be+1}$ in \eqref{eq:bdiff} vanishes
modulo~$2$. Hence, in this case, the differential equation
\eqref{eq:bdiff} is of the form appearing in Lemma~\ref{lem:diff710}.
The lemma then says that such
a differential equation has a solution if, and only if,
the Laurent polynomial $\text {Pol}_i(z)$ satisfies the condition given
there. We must therefore verify this condition for our Laurent polynomials 
$\text
{Pol}_i(z)$, arising through division of lines 3--7 of \eqref{eq:2be} by
$2^\be$. We shall actually prove (see the following paragraph) 
that $\text {Pol}_i(z)$ is divisible
by $(1-z^3)^2$, for {\it all\/} $i$ with $0\le i\le 2^{\al+2}-1$.
For odd $i$,
Corollary~\ref{cor:diff710} thus implies not only
unique existence of solutions $b_{i,\be+1}(z)$ but also divisibility of
these solutions by $1-z^3$. If we now consider 
Equation~\eqref{eq:bdiff} for even $i$, that is,
\begin{multline*} 
b_{i,\be+1}(z)+(z^7-z^{10})b'_{i,\be+1}(z)+
z^7(1-z^3)b_{i+1,\be+1}\Phi'(z)
+\text {Pol}_i(z)=0 \quad 
\text {modulo }2,\\[2mm]
i=2,4,\dots,2^{\al+2}-2,
\end{multline*}
then we see that divisibility of $b_{i+1,\be+1}(z)$ by $1-z^3$
guarantees that we may again apply Corollary~\ref{cor:diff710} to
obtain that there is also a unique solution $b_{i,\be+1}(z)$ for even
$i$, and that this solution is divisible by $1-z^3$.
In summary, we would have proved that, for an arbitrary positive
integer $\al$, the algorithm of
Section~\ref{sec:method} produces a solution $S_{2^{3\cdot
2^\al}}(z)$ to \eqref{eq:S1diff} modulo $2^{3\cdot 2^\al}$, which is a
polynomial in $\Phi(z)$ with coefficients that are Laurent polynomials in
$z$. This would establish the claim of the theorem.

It remains to prove that lines 3--7 of \eqref{eq:2be} are indeed divisible
by $(1-z^3)^2$. In order to see this conveniently, 
we write
$$
a_{i,\be}(z)=
d_{i,\be}(z)(1-z^3)\quad 
\text {modulo }(1-z^3)^2,\quad \quad 
i=1,2,\dots,2^{\al+2}-1,
$$
and
$$
a_{0,\be}(z)=
-1-2z^2+
d_{0,\be}(z)(1-z^3)\quad 
\text {modulo }(1-z^3)^2,
$$
where the $d_{i,\be}(z)$'s are polynomials of the form $p_0+p_1z+p_2z^2$, for
some integers $p_0,p_1,p_2$. (It should be noted that it is at this
point where we use our inductive hypothesis on divisibility of the
coefficients $a_{i,\be}(z)$.)
Then, reduction of lines 3--7 of \eqref{eq:2be} modulo
$(1-z^3)^2$ leads to the remainder
\begin{multline} \label{eq:Poli}
(3+2z(1-z^3))\Bigg(
-1-2z^2+
(1-z^3)\sum _{i=0} ^{2^{\al+2}-1}d_{i,\be}(z)\Phi^i(z)\Bigg)\\[2mm]
+z^7(1-z^3)\Bigg(
-4z
-3z^2\sum _{i=0} ^{2^{\al+2}-1}d_{i,\be}(z)\Phi^i(z)
+1+4z^2+4z^4
\Bigg)\\[2mm]
+3+6z^2+z(1-z^3)\quad 
\text {modulo }(1-z^3)^2.
\end{multline}
Using the fact that
$$
z^7(1-z^3)=z(1-z^3)\quad 
\text {modulo }(1-z^3)^2
$$
and similar reductions modulo $(1-z^3)^2$,
one sees that the expression in \eqref{eq:Poli} is in fact
divisible by $(1-z^3)^2$, as claimed.
This completes the proof of the theorem.
\end{proof}

Recall that, given a Laurent polynomial $p(z)$ over the integers, we write
$p^{(o)}(z)$ for the odd part $\frac {1} {2}(p(z)-p(-z))$
and $p^{(e)}(z)$ for the even part $\frac {1} {2}(p(z)+p(-z))$ of $p(z)$,
respectively.

\begin{lemma} \label{lem:diff710}
The differential equation
\begin{equation} \label{eq:diff710}
a(z)+(z^7-z^{10})a'(z)+\text {\em Pol}(z)=0\quad 
\text {\em modulo }2,
\end{equation}
where $\text {\em Pol}(z)$ is a given Laurent polynomial in $z$ with integer
coefficients, has a solution that is a Laurent polynomial
if, and only if, $\text {\em Pol}^{(o)}(z)$ is
divisible by $1+z^6$ {\em(modulo~$2$)}. In the latter case, the unique
solution is given by
$$
a(z)=\text {\em Pol}^{(e)}(z)+\frac {1+z^9} {1+z^6}
\text {\em Pol}^{(o)}(z)=
\text {\em Pol}^{(e)}(z)+\frac {1+z^3+z^6} {1-z^3}\text {\em Pol}^{(o)}(z)\quad 
\text {\em modulo }2.
$$
\end{lemma}

\begin{proof} Let $a_0(z)=z^m$. Then it is obvious that
$$
a_0(z)+(z^7-z^{10})a'_0(z)=a_0(z)\quad 
\text {modulo }2
$$ 
if $m$ is even, and that 
$$
a_0(z)+(z^7-z^{10})a'_0(z)=(1+z^6)a_0(z)+z^9a_0(z)\quad 
\text {modulo }2
$$ 
if $m$ is odd. The assertion of the lemma follows now immediately.
\end{proof}

\begin{corollary} \label{cor:diff710}
If, in the differential equation \eqref{eq:diff710}{\em,} the 
Laurent polynomial
$\text {\em Pol}(z)$ is divisible by $(1-z^3)^2$ {\em(modulo~$2$),} 
then the uniquely determined solution $a(z)$ is divisible by $1-z^3$
{\em(modulo~$2$)}.
\end{corollary}

We have implemented the algorithm described in the proof of
Theorem~\ref{thm:Unterg}. As an illustration, we present
the result for the modulus $64$.

{
\allowdisplaybreaks
\begin{theorem} \label{thm:Untergr64}
Let $\Phi(z)=\sum _{n\ge0} ^{}z^{2^n}$.
Then we have
\begin{multline} \label{eq:Loes2}
\sum _{n\ge0} ^{}s_{n+1}(PSL_2(\mathbb Z))\,z^n\\[2mm]
=
32 z^{50}+48 z^{44}+48
   z^{41}+32 z^{36}+32
   z^{35}+32 z^{33}+48
   z^{32}+16 z^{28}+40
   z^{26}+16 z^{25}\\[1mm]
+32
   z^{24}+32 z^{23}+16
   z^{22}+16 z^{21}+52
   z^{20}+32 z^{19}+40
   z^{18}+60 z^{17}\\[1mm]
+48
   z^{16}+4 z^{14}+32 z^{13}+4
   z^{12}+36 z^{11}+16
   z^{10}+60 z^9+2 z^8+16
   z^7+4 z^6\\[1mm]
+60 z^5+44 z^4+16
   z^3+54 z^2+60 z
+32+\frac{56}{z}+\frac{
   36}{z^2}+\frac{51}{z^3}+\frac{33}{z^4}+\frac{52}{z^5}+\frac {1} {z^7}\\[2mm]
+\bigg(32
   z^{34}+32 z^{26}+32
   z^{25}+32 z^{24}+16
   z^{22}+32 z^{21}+32
   z^{20}+32 z^{17}+32
   z^{16}\\[1mm]
   +48 z^{14}+16
   z^{13}+16 z^{12}+16
   z^{11}+32 z^{10}+32 z^8+48
   z^7+8 z^5+8 z^4+48 z^3+24
   z+32\\+\frac{20}{z}+\frac{12}
   {z^2}+\frac{8}{z^3}+\frac{3
   6}{z^4}+\frac{4}{z^5}+\frac
   {24}{z^6}\bigg)
   \Phi(z)\\[2mm]
+\bigg(32
   z^{34}+32 z^{29}+32
   z^{28}+32 z^{26}+32
   z^{24}+32 z^{21}+48
   z^{19}+32 z^{18}+48
   z^{17}+32 z^{14}\\[1mm]
   +48 z^{13}+32 z^{12}+56
   z^{10}+8 z^9+16 z^8+48
   z^7+24 z^6+56 z^5+44 z^4+16
   z^3\\[1mm]
   +48 z^2+40
   z+44+\frac{60}{z}+\frac{50}
   {z^2}+\frac{48}{z^3}+\frac{
   8}{z^4}+\frac{50}{z^5}+\frac{52}{z^6}+\frac{52}{z^7}
  \bigg) \Phi^2(z)\\[2mm]
+\bigg(32
   z^{28}+32 z^{24}+32
   z^{21}+32 z^{20}+32
   z^{19}+48 z^{16}+32
   z^{14}+32 z^{13}+32
   z^{12}\\[1mm]
   +32 z^{11}+16
   z^{10}+48 z^9+8 z^8+48
   z^6+56 z^4+8 z^3+16 z^2+48
   z+56+\frac{32}{z}+\frac{20}
   {z^2}\\+\frac{52}{z^3}+\frac{
   4}{z^4}+\frac{36}{z^5}+\frac{12}{z^6}+\frac{36}{z^7}\bigg) 
\Phi^3(z)\\[2mm]
+\bigg(32 z^{44}+32
   z^{41}+32 z^{33}+32
   z^{32}+32 z^{31}+32
   z^{30}+32 z^{28}+32
   z^{27}+16 z^{26}+32
   z^{24}\\[1mm]
   +32 z^{23}+48
   z^{22}+16 z^{21}+40
   z^{20}+32 z^{19}+32
   z^{18}+24 z^{17}+16
   z^{16}+48 z^{15}+32
   z^{14}\\[1mm]
   +16 z^{13}+8
   z^{12}+32 z^{11}+56
   z^{10}+56 z^9+44 z^8+40
   z^7+48 z^6+16 z^5+20 z^4+56
   z^3+30 z^2\\[1mm]
   +32z+28+\frac{40}{z}+\frac{34}
   {z^2}+\frac{52}{z^3}+\frac{
   17}{z^4}+\frac{26}{z^5}+\frac{40}{z^6}+\frac{29}{z^7}\bigg) 
\Phi^4(z)\\[2mm]
+\bigg(32
   z^{32}+32 z^{30}+32
   z^{26}+32 z^{24}+32
   z^{23}+32 z^{22}+32
   z^{21}+48 z^{20}+48
   z^{18}+32 z^{16}+48
   z^{14}\\[1mm]
   +32 z^{13}+48
   z^{12}+48 z^{11}+32 z^8+16
   z^7+56 z^6+48 z^5+48 z^4+40
   z^3+16 z^2\\[1mm]
   +32z+56+\frac{24}{z}+\frac{24}
   {z^2}+\frac{20}{z^3}+\frac{
   24}{z^4}+\frac{40}{z^5}+\frac{20}{z^6}\bigg)
   \Phi^5(z)\\[2mm]
+\bigg(32
   z^{32}+32 z^{31}+32
   z^{30}+32 z^{27}+32
   z^{24}+32 z^{23}+48
   z^{19}+16 z^{18}+48
   z^{17}\\[1mm]
   +16 z^{15}+48
   z^{14}+32 z^{12}+32
   z^{11}+56 z^8+40 z^7+56
   z^6+16 z^5\\[1mm]
   +8 z^4+56 z^3+4
   z^2+56
   z+32+\frac{8}{z}+\frac{52}{
   z^2}+\frac{60}{z^3}+\frac{3
   0}{z^4}+\frac{20}{z^5}+\frac{20}{z^6}+\frac{14}{z^7}\bigg) 
\Phi^6(z)\\[2mm]
+\bigg(32
   z^{30}+32 z^{26}+32
   z^{21}+32 z^{20}+48
   z^{18}+32 z^{16}+48
   z^{14}+32 z^{13}+48
   z^{10}+16 z^9+8 z^6\\[1mm]
+32
   z^5+16 z^4+16 z^3+8 z^2+48
   z+40+\frac{48}{z}+\frac{8}{
   z^2}+\frac{40}{z^3}+\frac{6
   0}{z^4}+\frac{8}{z^5}+\frac
   {24}{z^6}+\frac{60}{z^7}\bigg) \Phi^7(z)\\[2mm]
\text {\em modulo }64.
\end{multline}
\end{theorem}
}

\section{Counting permutation representations of $\Gamma_m(3)$ for $m$ prime}
\label{sec:lift}

Let $m$ be a prime and, for a finitely generated group $\Gamma$, let
$h_\Gamma(n):= \vert \mathrm{Hom}(\Gamma, S_n)\vert$. We want to
determine the function $h_{\Gamma_m(3)}(n)$ counting the permutation
representations of degree $n$
of the lift $\Gamma_m(3)$ of the inhomogeneous modular 
group $PSL_2(\mathbb Z)\cong \mathfrak H(3)$. 
Suppose first that $m\geq5$, and classify the
representations $\Gamma_m(3)\rightarrow S_n$ by the image $\rho\in S_n$
of the central element $x^2=y^3$. The permutation $\rho$ must be of
the form 
\[
\rho = \prod_{i=1}^r \sigma_i,\quad 0\leq r\leq \lfloor n/m\rfloor,
\]
with pairwise disjoint $m$-cycles $\sigma_i$.

For fixed $r$, the symmetric group $S_n$ contains exactly
\begin{equation}
\label{Eq:RhoCount}
\frac{1}{r!} \binom{n}{m,\ldots,m,n-mr} (m-1)!^r = \frac{n!}{r!\,
(n-mr)!\, m^r}
\end{equation}
such elements $\rho$.

Next, given such $\rho$, the image of the generator $x$ will contain a
certain number $s$ of $2m$-cycles in its disjoint cycle decomposition,
$0\leq s\leq \lfloor r/2\rfloor$, each of which breaks into two
$m$-cycles when squared. Thus, in order to construct a square root of
$\rho$ (i.e., a possible image of $x$), we need to
\begin{enumerate}
\item[(i)] fix $s$ in the
range $0\leq s\leq \lfloor r/2\rfloor$, 
\vspace{2mm}
\item[(ii)] select $s$ $2$-element
subsets from the $r$ $m$-cycles of $\rho$, which can be done in 
\[
\frac{1}{s!} \binom{r}{2s} \binom{2s}{2,\ldots,2} = \frac{r!}{2^s
s!\, (r-2s)!}
\] 
different ways,
\vspace{2mm} 
\item[(iii)] lift each of these $s$ pairs of $m$-cycles to a
$2m$-cycle (whose square is the product of the two given $m$-cycles),
which can be done in $m$ ways,
\vspace{2mm} 
\item[(iv)] compute $\sigma^{\om}$ for each of
the $r-2s$ remaining $m$-cycles $\sigma$, where $\om$ is the
multiplicative inverse of $2$ modulo~$m$, and
\vspace{2mm} 
\item[(v)] select a permutation $\pi$
subject only to the condition that $\pi^2=1$ on the $n-mr$ letters not
involved in any of the $r$ $m$-cycles of $\rho$. 
\end{enumerate} 
Hence, there are
precisely 
\begin{equation}
\label{Eq:SquareRoots}
r!\, h_{C_2}(n-mr) \sum_{s=0}^{\lfloor r/2\rfloor} \frac{m^s}{2^s
s!\, (r-2s)!}
\end{equation}
distinct square roots for each $\rho$ involving $r$ $m$-cycles.

Similarly, the number of cubic roots of such $\rho$ is given by
\begin{equation}
\label{Eq:CubicRoots}
r!\, h_{C_3}(n-mr) \sum_{t=0}^{\lfloor r/3\rfloor} \frac{m^{2t}}{3^t
t!\, (r-3t)!}
\end{equation}
(classify the cubic roots by the number $t$ of $3m$-cycles, and use
the fact that each product of three disjoint $m$-cycles is the cube of
precisely $2m^2$ $3m$-cycles). Multiplying
(\ref{Eq:RhoCount})--(\ref{Eq:CubicRoots}) and summing over
$r=0,1,\ldots, \lfloor n/m\rfloor$, we find that 
\begin{equation}
\label{Eq:HomCountGen}
h_{\Gamma_m(3)}(n) = n! \sum_{r=0}^{\lfloor n/m\rfloor}
\sum_{s=0}^{\lfloor r/2\rfloor} \sum_{t=0}^{\lfloor r/3\rfloor}
\frac{r!\, h_{C_2}(n-mr) h_{C_3}(n-mr)}{2^s 3^t m^{r-s-2t} s!\, t!\,
(n-mr)!\,
(r-2s)!\, (r-3t)!},\quad m\geq5. 
\end{equation}
The cases where $m=2,3$ have to be treated separately. By arguments
similar to the ones above, one finds that 
\begin{equation}
\label{Eq:Gamma2HomCount}
h_{\Gamma_2(3)}(n) = n! \sum_{r=0}^{\lfloor n/4\rfloor}
\sum_{s=0}^{\lfloor 2r/3\rfloor} \frac{(2r)!\, h_{C_2}(n-4r)
h_{C_3}(n-4r)}{2^{2(r-s)} 3^s r!\, s!\, (n-4r)!\, (2r-3s)!}, 
\end{equation}
and that
\begin{equation}
\label{Eq:Gamma3HomCount}
h_{\Gamma_3(3)}(n) = n! \sum_{r=0}^{\lfloor n/9\rfloor}
\sum_{s=0}^{\lfloor 3r/2\rfloor} \frac{(3r)!\, h_{C_2}(n-9r)
h_{C_3}(n-9r)}{2^s 3^{2r-s} r!\, s!\, (n-9r)!\, (3r-2s)!}. 
\end{equation} 
Furthermore, it is well-known that, for a prime $p$,
$$
h_{C_p}(n)=\sum_{k=0}^{\fl{n/p}}\frac{n!} {p^k\, k!\,(n-pk)!},
$$
which allows us to make Formulae 
\eqref{Eq:HomCountGen}--\eqref{Eq:Gamma3HomCount}
completely explicit.

\section{Subgroup numbers for the homogeneous modular group}
\label{sec:SL2Z}

In this section we consider the problem of determining the behaviour
of the number of index-$n$-subgroups in the homogeneous modular group
$SL_2(\mathbb Z)$ modulo powers of~$2$. By a folklore result that goes
back at least to Dey \cite{Dey}, these subgroup numbers are in a direct
relation to numbers of permutation representations of $SL_2(\mathbb
Z)$. Our starting point is a recurrence with polynomial coefficients
for the latter numbers,
which is then translated into a Riccati-type differential equation for
the generating function $\sum _{n\ge0} ^{}s_{n+1}(SL_2(\mathbb
Z))\,z^n$ of the subgroup numbers. (Equation~\eqref{eq:Riccati16}
displays this equation when reduced modulo~$16$.) Our method from
Section~\ref{sec:method} is then applied to this differential
equation. Direct application already fails for the
modulus~$8$. Interestingly, if we instead apply our method with the
minimal polynomial for the modulus~$16$ given in 
Proposition~\ref{prop:minpol}, then our algorithm produces a
result for modulus~$8$
(see Theorems~\ref{thm:Sl2Z8} and \ref{thm:1}), but then 
fails at the level of modulus~$16$. By the enhancement of the method outlined
in Appendix~\ref{appD}, we are nevertheless able to treat the
subgroup numbers $s_{n}(SL_2(\Z))$ modulo~$16$ as well (see
Theorem~\ref{thm:Sl2Z16}).
%??(see Remark~\ref{rem:2}). 
%This means that a new phenomenon, not
%covered by our Ansatz, arises in the behaviour of the subgroups
%numbers at the level of modulus~$16$. It would
%be of great interest to find an explicit description of the hidden scheme
%behind the mod-$16$ behaviour of the number of subgroups of index $n$
%in $SL_2(\mathbb Z)$, and, more generally, of the behaviour modulo
%{\it any} power of $2$.
In view of the already substantial computational effort involved in
the case of modulus~$16$, we did not try to push our analysis further
to higher powers of $2$. In particular, as opposed to the case of the
subgroup numbers of $PSL_2(\Z)$, it remains unclear whether 
%??our
%enhanced method always succeeds; that is, whether
it is possible to express the generating function
$\sum _{n\ge0} ^{}s_{n+1}(SL_2(\mathbb Z))\,z^n$, when the
coefficients are reduced modulo~$2^\ga$, as a polynomial in $\Phi(z)$
with coefficients that are Laurent polynomials in $z$ over the
integers for all $\ga\ge1$. We feel, however, that this should be the
case; see Conjecture~\ref{conj:SL2Z}.

\medskip
Let us start with the aforementioned result
(cf.\ \cite[Theorem~6.10]{Dey}, see
\cite[Prop.~1]{DM} for a conceptual proof, plus generalisations)
relating the numbers of subgroups of a finitely generated group to
the numbers of its permutation representations.

\begin{proposition}
Let $\Ga$ be a finitely generated group. Then we have
\begin{equation} \label{eq:Dey}
\sum_{n=0}^\infty \vert\Hom(\Ga,S_n)\vert\frac {z^n} {n!}=
\exp\left(\sum_{n=1}^\infty s_{n}(\Ga)\frac {z^n} {n}\right) .
\end{equation}
\end{proposition}

We take $\Ga=\Ga_2(3)=SL_2(\mathbb Z)$ and combine \eqref{eq:Dey} with
\eqref{Eq:Gamma2HomCount}, the latter giving an explicit formula for 
the homomorphism numbers
$h_n:=\vert\Hom(SL_2(\mathbb Z),S_n)\vert$. Using the Guessing package
\cite{Guess}, we found
a recurrence of order $30$ for the sequence $(h_n/n!)_{n\ge0}$,
with coefficients that are polynomials 
in $n$ over $\mathbb Z$.
The validity of the recurrence was verified by
computing a certificate using Koutschan's {\sl Mathematica} package 
{\tt HolonomicFunctions} \cite{KoutAA}.\footnote{\label{foot:1}%
The certificate has 4~megabytes, and, to obtain it, required about
30~hours of computation time.
The coefficients of this
recurrence are polynomials in $n$ of degree~$34$. Interestingly, there
is a recurrence of order $32$ with leading coefficient~$1$.
Although we did not try to prove it, it is likely that the recurrence
of order $30$ is the recurrence of minimal order.}
However, this recurrence is not suitable for our purpose, for which
we require a recurrence with coefficients that are polynomials in~$n$
over $\mathbb Z$, and with leading coefficient~$n$.
A recurrence of this form, if it exists, must be a left multiple of
the recurrence operator corresponding to the minimal order recurrence.
The construction of such left multiples is known as {\it desingularisation},
and algorithms are known for this purpose~\cite{AvH}. This technique
can be used to eliminate factors from the leading coefficient of
the recurrence (whenever possible), but it cannot be used to ensure
that the leading coefficient be a monic polynomial.
The recurrence of order $32$ mentioned in Footnote~\ref{foot:1},
with leading coefficient $1$, could indeed be used for our purpose,
by simply multiplying it by $n$. However, since this recurrence has
high-degree polynomials as coefficients, we preferred to work with
a different recurrence with lower degree polynomials as coefficients. 
The price to pay is that the order of such a recurrence will be higher.
So, by an indeterminate Ansatz, we computed a
candidate for a recurrence of the desired form of order $50$, with polynomial
coefficients of degree at most $5$.\footnote{We used the function
{\tt LinSolveQ} of the Guessing package \cite{Guess}, which uses
modular arithmetic, in order to
solve the arising system of linear equations. {\sl
Mathematica}'s built-in linear system solver is not capable of
solving it on current hardware due to the huge numerators and denominators of rational numbers 
which arise during the computation.} To be precise, it is the uniquely
determined recurrence of the form
$$
\sum_{k=0}^{50}
\bigg(\sum_{i=0}^{5}a(k,i)n^i\bigg)
\frac {\vert\Hom(SL_2(\mathbb Z),S_{n-k})\vert} {(n-k)!}=0,\quad \quad n\ge50,
$$
where
\begin{align*}
a(0,0)&=a(0,2)=a(0,3)=a(0,4)=a(0,5)=0,\\
a(0,1)&=1,\\
a(50,5)&=47323476536606893277939021129424044201294092725261226600745838\backslash\\
&\quad \quad \quad \quad  993897087202045010603943040012232525, \\
a(50,4)&=
-853333370519051585059335896571817612918194491041969759097679\backslash\\
&\quad \quad \quad \quad  3078743106989966250706985019403282594096, \\
a(49,5)&= 
  2507660784286104701612089471873568042396155618028516886767837\backslash\\
&\quad \quad \quad \quad  559764248217845308468763736164634176, \\
a(49,4)&=
a(48,5)=
a(48,4)=
a(47,5)=
a(47,4)\\
&=
a(46,5)=
a(46,4)=
a(45,5)=
a(45,4)=0.
\end{align*}
Subsequently, we checked
that this recurrence is a left-multiple of the certified
recurrence of order $30$, thereby establishing validity of this
candidate recurrence of order $50$.
This last recurrence was then converted into a linear differential equation 
with polynomial coefficients for the series
$$H(z):=\sum_{n=0}^\infty h_n\frac {z^n} {n!}
=\sum_{n=0}^\infty \vert\Hom(SL_2(\mathbb Z),S_n)\vert\frac {z^n} {n!}.$$
Finally, this last mentioned differential equation can be
translated into a Riccati-type differential equation for
the generating function 
\begin{equation} \label{eq:Sdef}
S(z)=\sum _{n\ge0} ^{}s_{n+1}(SL_2(\mathbb Z))\,z^n
\end{equation}
for the subgroup numbers of $SL_2(\mathbb Z)$.
This is done by differentiating the relation \eqref{eq:Dey}, with
$\Ga=SL_2(\mathbb Z)$, several times and by dividing by $H(z)$. This leads
to relations of the form
\begin{equation} \label{eq:HS}
\frac {H^{(k)}(z)}
{H(z)}=P_k(S(z),S'(z),\dots),\quad \quad 
k=1,2,\dots,
\end{equation}
where $P_k(S(z),S'(z),\dots)$ is a polynomial in $S(z)$ and its
derivatives, which can be determined explicitly using the Fa\`a di Bruno
formula for derivatives of composite functions (cf.\ \cite[Sec.~3.4]{ComtAA};
but see also \cite{CraiAA,JohWAE}). 
Substituting these relations in the linear differential equation for
$H(z)$, one obtains the announced Riccati-type differential equation
for $S(z)$. It turns out that this differential equation has integral
coefficients, so that it is amenable to our method from
Section~\ref{sec:method}. The differential equation
cannot be displayed here since this would 
require about ten pages. Its reduction modulo~$16$ is written out in
\eqref{eq:Riccati16}. Our method from
Section~\ref{sec:method} with $\al=0$ applied to \eqref{eq:Riccati16}
does in fact {\it not\/} produce a result modulo~$8=2^{3\cdot 2^0}$ 
(it stops at the
level of modulus~$4$). If, however, we use the method from 
Section~\ref{sec:method} with the minimal polynomial for the
modulus~$16$ in place of the one for the modulus~$8$, then the method
goes through up to modulus~$8$ (but fails for modulus~$16$).
This yields the following theorem. It refines the parity result
\cite[Eq.~(6.3) with $\vert H\vert=1$, $q=3$, $m=2$]{KrMuAC}.
%\pagebreak

{
\allowdisplaybreaks
\begin{theorem} \label{thm:Sl2Z8}
Let $\Phi(z)=\sum _{n\ge0} ^{}z^{2^n}$.
Then we have
\begin{multline} \label{eq:Sl2Z8}
\sum _{n\ge0} ^{}s_{n+1}(SL_2(\mathbb Z))\,z^n\\[2mm]
=
4 z^{20}+4 z^{17}+4 z^{14}+4
   z^{12}+4
   z^{10}+4 z^9+6
   z^8+4 z^5+6 z^4+4
   z^2+4 z+6\\[1mm]
+\frac{7}{z^2}+\frac{3}{z^3}+\frac{6}{z^6
   }+\frac{6}{z^7}
+\frac{4}{z^8}+\frac{1}{z^9}+\frac{3}{z^{11}}+\frac{6}{z^{12}}\\[1mm]
+\left(
4 z^3
+4   z^2
+\frac{4}{z}
+\frac{6}{z^3}
+\frac{6}{z^4   }
+\frac{6}{z^6}
+\frac{2}{z^7}
+\frac{4}{z   ^8}
+\frac{4}{z^9}
+\frac{4}{z^{10}}
+\frac{6}{z^{12   }}
+\frac{2}{z^{13}}
\right)
   \Phi(z)\\[2mm]
+\left(
4
   z^8
+4 z^4+4
   z^3+6
   z^2+4
+\frac{4}{z}+\frac{6}{z^2}+\frac{2}{z^3}
+\frac{5}{z^4}
+\frac{2}{z^5}\right.\\[1mm]
\left.+\frac{6}{z^6
   }+\frac{1}{z^7}+\frac{4}{z^8}+\frac{6}{
   z^9}+\frac{4}{z^{10}}
+\frac{6}{z^{11}}+\frac{6}{z^{12}}+\frac{5}{z^{13}}
\right)
   \Phi^2(z)\\[2mm]
+\left(
4 z^2+\frac{4}{z^2}
+\frac{4}
   {z^3}+\frac{2}{z^4}+\frac{4}{z^5}+\frac{
   4}{z^6}+\frac{2}{z^7}+\frac{4}{z^9}+\frac{4
   }{z^{11}}+\frac{4}{z^{12}}+\frac{2}{z^{13}}
\right)
   \Phi^3(z)\\[2mm]
\text {\em modulo }8.
\end{multline}
\end{theorem}
}

\begin{proof}
{
\allowdisplaybreaks
The Riccati-type differential equation for 
$S(z)$ (as defined in \eqref{eq:Sdef})
modulo $16$ is\footnote{We display the differential equation modulo~$16$
  in order to prepare for Theorem~\ref{thm:Sl2Z16}.}
\begin{multline} \label{eq:Riccati16}
p_0(z)+ 
p_1(z)  S(z) + p_2(z) {{S(z)}^2} + p_3(z) {{S(z)}^3} +
p_4(z)  {{S(z)}^4} + p_5(z)  {{S(z)}^5} +
p_6(z)  S'(z) \\[2mm]
+ p_7(z) {{S'(z)}^2} + p_8(z) S(z) S'(z) 
+ p_{9}(z)  {{S(z)}^2} S'(z)  + p_{10}(z) {{S(z)}^3} S'(z)  \\[2mm]
+ p_{11}(z) S(z) {{S'(z)}^2}
+ p_{12}(z) S''(z) 
     + p_{13}(z) S(z) S''(z) + p_{14}(z) {{S(z)}^2} S''(z) \\[2mm]
    + p_{15}(z) S'(z) S''(z) 
  + p_{16}(z)  S'''(z)  + p_{17}(z)  S(z) S'''(z) 
  + p_{18}(z) S''''(z) =0\\[2mm]
\text {modulo }16,
\end{multline}
with coefficients $p_j(z)$, $j=0,1,\dots,18$ as displayed in
Appendix~\ref{appB}.

The differential equation \eqref{eq:Riccati16} has a unique solution
since comparison of coefficients of $z^{N}$ fixes the initial values,
and yields a recurrence for
the sequence $\big(s_n(SL_2(\mathbb Z))\big)_{n\ge1}$ 
which computes $s_{n+1}(SL_2(\mathbb Z))$ from terms
involving only $s_i(SL_2(\mathbb Z))$ with $i\le n$. 

Now we apply the method from Section~\ref{sec:method} with the
polynomial 
\begin{equation} \label{eq:minpol16}
(\Phi^2(z)+\Phi(z)+z)
(\Phi^4(z)+6\Phi^3(z)+(2z+3)\Phi^2(z)+(2z+6)\Phi(z)+2z+5z^2)
\end{equation}
in place of the polynomial on the left-hand side of \eqref{eq:PhiRel}
to the differential equation \eqref{eq:Riccati16}
(that is, in view of Proposition~\ref{prop:minpol} 
we are aiming at determining the subgroup numbers of
$s_n(SL_2(\mathbb Z))$ modulo~$16$).
This yields the above result by means of a straightforward computer calculation.
}
\end{proof}

If we want to know explicitly for which $n$ the subgroup number 
$s_{n}(SL_2(\mathbb Z))$ is congruent to a particular value modulo~$8$, 
then we should first apply the algorithm from Section~\ref{sec:extr}
(see \eqref{eq:Phipot} and the proof of Lemma~\ref{lem:Hbi}) in order
to express powers of $\Phi(z)$ on the right-hand side of \eqref{eq:Sl2Z8} 
in terms of the series $H_{a_1,\dots,a_r}(z)$. (The corresponding
expansions are in fact listed in \eqref{eq:Phi2} and \eqref{eq:Phi3}.)
The result is
{\allowdisplaybreaks
\begin{multline} \label{eq:Sl2Z8H}
\sum _{n\ge0} ^{}s_{n+1}(SL_2(\mathbb Z))\,z^n
=
\left(
\frac{4}{z^4}+
\frac{4}{z^7}+
\frac{4}{z^{13}}
\right)
    H_{3}(z)
+
 \left(
\frac{4}{z^4}+
\frac{4}{z^7}+
\frac{4}{z^{13}}
\right)
    H_{1,1,1}(z)\\[2mm]
+
\left(
4 z^2+
\frac{4}{z^2}+
\frac{4}
    {z^3}+
\frac{6}{z^4}+
\frac{4}{z^5}+
\frac{4}{z^6}+
\frac{6}{z^7}+
\frac{4}{z^9}+
\frac{4}{z^{11}}+
\frac{4}{z^{12}}+
\frac{6}{z^{13}}
\right)
    H_{1,1}(z)\\[2mm]
+
\left(
4
    z^4+
4 z^3+
6
    z^2+
4+
4
    z^8+
\frac{4}{z}+
\frac{6}{z^2}+
\frac{6}{z^3}+
\frac{5}{z^4}\right.\kern3cm\\[1mm]
\kern1cm
\left.+
\frac{2}{z^5}+
\frac{2}{z^6}+
\frac{1}{z^7}+
\frac{4}{z^8}+
\frac{6}{z^9}+
\frac{4} {z^{10}}+
\frac{6}{z^{11}}+
\frac{2}{z^{12}}+
\frac{5}{z^{13}}
\right) H_{1}(z)\\[2mm]
+
4 z^{20}+
4 z^{17}+
4
    z^{14}+
4 z^{12}+
4
    z^{10}+
6
    z^8+
2
    z^4+
6 z^3+
4 z^2+
2\\[1mm]
+ \frac{6}{z}+
\frac{1}{z^2}+
\frac{2}{z^4}+
\frac{6}{z^5}+
\frac{7}{z^6}+
\frac{2}{z^7}+
\frac{2}{z^8}+
\frac{5}{z^9}+
\frac{6}{z^{10}}+
\frac{1}{z^{11}}+
\frac{3}{z^{12}}\quad 
\text {modulo }8.
\end{multline}}%
From this expression, it is a routine (albeit tedious) task to
extract an explicit description
of the behaviour of the subgroup numbers of $SL_2(\mathbb Z)$
modulo~$8$.
Since the corresponding result can be stated within moderate amount of space, 
we present it in the next theorem.

\begin{theorem} \label{thm:1}
The subgroup numbers $s_n(SL_2(\mathbb Z))$ 
obey the following congruences modulo $8:$
\begin{enumerate}
\item[(i)] $s_n(SL_2(\mathbb Z))\equiv 1$~{\em(mod~$8$)} if, and only if, $n=1,2,4,10,$ or if 
$n$ is of the form $2^\si-3$ for some $\si\ge4;$
\vspace{2mm}
\item[(ii)]$s_n(SL_2(\mathbb Z))\equiv 2$~{\em(mod~$8$)} if, and only if, $n=7,12,17,$ or if 
$n$ is of one of the forms 
$$3\cdot 2^\si-3,\ 3\cdot 2^\si-6,\
3\cdot 2^\si-12,\quad \text {for some }\si\ge4;
$$
\item[(iii)]$s_n(SL_2(\mathbb Z))\equiv 4$~{\em(mod~$8$)} if, and only if,
$n=3,22,23,27,46,47,51,$ or if
$n$ is of one of the forms 
\begin{align} 
\label{eq:4si}
\kern1.1cm
&2^\si+6,\ 2^\si+7,\
2^\si+11,\
2^\si+12,\
2^\si+18,\
2^\si+21,\
\quad \text {for some }\si\ge5,\\[2mm]
\notag
&2^\si+2^\ta-2,\
2^\si+2^\ta+1,\
2^\si+2^\ta+3,\\
& \kern1.5cm
\text {for some $\si,\ta$ with }\si\ge6\text { and
}4\le\ta\le\si-1,
\label{eq:4sita}
\\[2mm]
\notag
&2^\si+2^\ta+2^\nu-12,\
2^\si+2^\ta+2^\nu-6,\
2^\si+2^\ta+2^\nu-3,\\
& \kern1.5cm
\text {for some $\si,\ta,\nu$ with }
\si\ge6,\ 5\le\nu\le\si-1,\text { and }3\le\ta\le\nu-1;
\label{eq:4sitanu}
\end{align}
\item[(iv)]$s_n(SL_2(\mathbb Z))\equiv 5$~{\em(mod~$8$)} if, and only if, $n=5,$ or if 
$n$ is of one of the forms 
$$2^\si-6,\ 2^\si-12,\quad \text {for some }\si\ge5;
$$
\item[(v)]$s_n(SL_2(\mathbb Z))\equiv 6$~{\em(mod~$8$)} if, and only if,
$n=6,11,14,18,19,21,33,34,35,37,$ or if 
$n$ is of one of the forms 
\begin{align} 
\label{eq:6siA}
\kern1.1cm
&2^\si-2,\ 2^\si-4,\
\quad \text {for some }\si\ge5,\\[2mm]
\notag
\kern1.1cm
&2^\si+1,\ 
2^\si+2,\ 
2^\si+3,\ 
2^\si+4,\ 
2^\si+5,\ 
2^\si+10,\ 
2^\si+13,\\
\label{eq:6siB}
& \kern3cm
\text {for some }\si\ge6,\\[2mm]
\notag
&2^\si+2^\ta-3,\
2^\si+2^\ta-6,\
2^\si+2^\ta-12,\\
& \kern3cm
\text {for some $\si,\ta$ with }\si\ge7\text { and
}5\le\ta\le\si-2;
\label{eq:6sita}
\end{align}
\item[(vi)]in the cases not covered by items {\em(i)}--{\em(v),}
$s_n(SL_2(\mathbb Z))$ is divisible by $8;$
in particular, 
$s_n(SL_2(\mathbb Z))\not\equiv 3,7$~{\em(mod~$8$)} for all $n$.
\end{enumerate}
\end{theorem}

%??\begin{remarknu} \label{rem:2}
%\allowdisplaybreaks
As we already said earlier, the method from Section~\ref{sec:method} with the 
polynomial in \eqref{eq:minpol16}
in place of the polynomial on the left-hand side of \eqref{eq:PhiRel}
applied to the differential equation \eqref{eq:Riccati16}
does not actually produce a result modulo~$16$ (although this is what
it would be designed to). It only produces the result modulo~$8$
given in Theorem~\ref{thm:Sl2Z8}
%??the mod-8-level, we obtain 
%\begin{multline} \label{eq:Sl2Z16}
%\sum _{n\ge0} ^{}s_{n+1}(SL_2(\mathbb Z))\,z^n
%=
%4
%   z^{20}+4 z^{17}+4 z^{14}+4
%   z^{12}
%+4
%   z^9
%+6
%   z^8
%+4
%   z^6
%+4 z^4
%+4
%   z^3\\[1mm]
%+\frac{2}{z}
%+\frac{4}{
%   z^2}
%+\frac{3}{z^3}
%+\frac{6}{z^4}
%+\frac{1}{
%   z^5}
%+\frac{4}{z^6}
%+\frac{4}{
%   z^7}
%+\frac{4}{z^8}
%+\frac{3}{z^9}
%+\frac{6}{z^{10}}\\[2mm]
%+\left(
%4
%   z^4
%+4
%+\frac{6}{
%   z^2}
%+\frac{4}{z^3}
%+\frac{6}{z^5}
%+\frac{4}{z
%   ^7}
%+\frac{4}{z^9}
%+\frac{6}{z^{11
%   }}
%\right)
%   \Phi(z)\\[2mm]
%+\left(
%4
%   z^3
%+\frac{4}{
%   z}
%+\frac{2}{z^3}
%+\frac{4}{z^4}
%+\frac{2}{z^6}
%+\frac{4}{z^8}
%+\frac{4
%   }{z^{10}}
%+\frac{2}{z^{12}}
%\right)
%   \Phi^2(z)\\[2mm]
%+\left(
%4
%   z^2
%+\frac{4}{z^2}
%+\frac{4}{z^3}
%+\frac{2}{z^4}
%+\frac{4}{z^5}
%+\frac{4}{z^6}
%+\frac
%   {2}{z^7}
%+\frac{4}{z^9}
%+\frac{4}{
%   z^{11}}
%+\frac{4}{z^{12}}
%+\frac{2}{z^{13}}   
%\right)
%   \Phi^3(z)\\[2mm]
%+\left(
%4
%   z^8
%+4 z^4
%+4
%   z^3
%+6
%   z^2
%+4
%+\frac{4}{
%   z}
%+\frac{2}{z^2}
%+\frac{6}{z^3}
%+\frac{1}{z^4}\right.
%\kern4cm
%\\[1mm]
%\kern4cm
%\left.
%+\frac{6}
%   {z^5}
%+\frac{2}{z^6}
%+\frac{5}{
%   z^7}
%+\frac{4}{z^8}
%+
%   \frac{6}{z^9}
%+\frac{4}{z^{10}}
%+\frac{2
%   }{z^{11}}
%+\frac{2}{z^{12}}
%+\frac{1}{z^{13
%   }}
%\right) \Phi^4(z)\\[2mm]
%+\left(
%4
%   z^2
%+\frac{4}{z^2}
%+\frac{6}{z^4}
%+\frac{
%   4}{z^5}
%+\frac{6}{z^7}
%+\frac{4
%   }{z^9}
%+\frac{4}{z^{11}}
%+\frac{6}{z^{13
%   }}
%\right)
%   \Phi^5(z)\quad 
%\text {modulo }8.
%\end{multline}
since, at the mod-16-level, the arising system of equations 
has no polynomial solutions. Nevertheless, by applying the enhanced
method from Appendix~\ref{appD} to this last system of equations,
a solution modulo~$16$ can still be found, the result being displayed
in our next theorem.
%\end{remarknu}

\begin{theorem} \label{thm:Sl2Z16}
Let $\Phi(z)=\sum _{n\ge0} ^{}z^{2^n}$.
Then we have
{\allowdisplaybreaks
\begin{multline} \label{eq:Sl2Z16}
\sum _{n\ge0} ^{}s_{n+1}(SL_2(\mathbb Z))\,z^n\\[2mm]
=
8 z^{74}+8 z^{71}+8 z^{68}+8
   z^{67}+8 z^{62}+8 z^{61}+8
   z^{57}+8 z^{56}+8 z^{54}+8
   z^{50}+8 z^{48}+8 z^{47}\\[1mm]
+8 z^{45}
+8 z^{44}
+8 z^{43}+8
   z^{42}+8 z^{41}+8 z^{40}+8
   z^{38}+8 z^{35}+8 z^{31}+8
   z^{26}+8 z^{24}+8 z^{21}\\[1mm]
+12 z^{20}+12 z^{17}
+8 z^{16}+8
   z^{15}
+4 z^{14}+4 z^{12}+12
   z^9+14 z^8+8 z^7+12 z^6+16
   z^5\\[1mm]
+12 z^4+12 z^3+8 z^2
+8+\frac{10}{z}+\frac{12}{z
   ^2}
+\frac{3}{z^3}+\frac{14}{z^4}
   +\frac{9}{z^5}+\frac{4}{z^6}
+\frac{12}{z^7}+\frac{4}{z^8}+\frac{3}{z^9}+\frac{6}{z^{10}}\\[1mm]
+\bigg(8 z^{73}+8
   z^{72}+8 z^{71}+8 z^{69}+8
   z^{68}+8 z^{67}+8 z^{66}+8
   z^{64}+8 z^{63}+8 z^{62}+8
   z^{61}\\[1mm]
+8 z^{60}+8 z^{59}+8
   z^{55}+8 z^{54}+8 z^{49}+8
   z^{47}+8 z^{41}+8 z^{38}+8
   z^{36}+8 z^{35}+8 z^{32}+8
   z^{30}\\[1mm]
+8 z^{29}+8 z^{28}+8
   z^{27}
+8 z^{24}+8 z^{23}+8
   z^{21}+8 z^{17}+8 z^{14}+8
   z^{13}+8 z^{12}+8 z^9+8 z^8\\[1mm]
+8 z^7+8 z^5
+12 z^4+8
   z^3+4+\frac{16}{z}+\frac{6}{z^2
   }
+\frac{12}{z^3}+\frac{8}{z^4}+
   \frac{6}{z^5}+\frac{4}{z^7}+\frac{4}{z^9}
+\frac{8}{z^{10}}+\frac{14}{z^{11}}\bigg)
   \Phi(z)\\[1mm]
+\bigg(8 z^{72}+8
   z^{69}+8 z^{68}+8 z^{66}+8
   z^{65}+8 z^{62}+8 z^{61}+8
   z^{58}+8 z^{57}+8 z^{56}+8
   z^{55}\\[1mm]
+8 z^{53}+8 z^{51}+8
   z^{50}+8 z^{49}+8 z^{38}+8
   z^{37}+8 z^{36}+8 z^{35}+8
   z^{28}+8 z^{25}+8 z^{24}+8
   z^{22}\\[1mm]
+8 z^{18}+8 z^{17}+8
   z^{14}+8 z^{11}+8 z^{10}+8
   z^8+8 z^4+12 z^3+8
   z^2+8+\frac{4}{z}+\frac{8}{z^2}
   +\frac{10}{z^3}\\[1mm]
+\frac{4}{z^4}+\frac{10}{z^6}+\frac{12}{z^8}
+\frac{12}{z^{10}}+\frac{8}{z^{11}}
   +\frac{2}{z^{12}}\bigg)
   \Phi^2(z)\\[1mm]
+\bigg(8 z^{72}+8
   z^{69}+8 z^{67}+8 z^{66}+8
   z^{61}+8 z^{60}+8 z^{56}+8
   z^{55}+8 z^{50}+8 z^{47}+8
   z^{46}\\[1mm]
+8 z^{45}+8 z^{42}+8
   z^{40}+8 z^{37}+8 z^{33}+8
   z^{32}+8 z^{31}+8 z^{30}+8
   z^{28}+8 z^{25}+8 z^{20}+8
   z^{19}\\[1mm]
+8 z^{17}+8 z^{16}+8
   z^{13}
+8 z^{12}+8 z^{10}+8
   z^9+8 z^8+8 z^3+12 z^2+8
   z+\frac{8}{z}+\frac{4}{z^2}+\frac{4}{z^3}+\frac{10}{z^4}\\[1mm]
+\frac
   {12}{z^5}+\frac{12}{z^6}+\frac{1
   8}{z^7}+\frac{8}{z^8}+\frac{12}{
   z^9}
+\frac{8}{z^{10}}+\frac{4}{
   z^{11}}+\frac{12}{z^{12}}+\frac
   {10}{z^{13}}\bigg)
   \Phi^3(z)\\[1mm]
+\bigg(8 z^{72}+8
   z^{71}+8 z^{70}+8 z^{67}+8
   z^{65}+8 z^{64}+8 z^{63}+8
   z^{60}+8 z^{59}+8 z^{58}+8
   z^{57}\\[1mm]
+8 z^{56}+8 z^{54}+8
   z^{53}+8 z^{51}+8 z^{49}+8
   z^{48}+8 z^{45}+8 z^{40}+8
   z^{37}+8 z^{36}+8 z^{34}+8
   z^{33}\\[1mm]
+8 z^{32}
+8 z^{30}+8
   z^{29}
+8 z^{28}+8 z^{25}+8
   z^{24}+8 z^{21}+8 z^{18}+8
   z^{17}+8 z^{13}+8 z^{11}+8
   z^9\\[1mm]
+12 z^8+8 z^7+8 z^6+8 z^5
+12
   z^4+12 z^3+6 z^2+8
   z+12
+\frac{4}{z}+\frac{10}{z^2}
   +\frac{14}{z^3}+\frac{1}{z^4}\\[1mm]
+\frac{14}{z^5}+\frac{10}{z^6}
+\frac{5}{z^7}+\frac{12}{z^8}+\frac{
   14}{z^9}+\frac{12}{z^{10}}+\frac
   {2}{z^{11}}+\frac{2}{z^{12}}+\frac{1}{z^{13}}\bigg)
   \Phi^4(z)\\[1mm]
+\bigg(8
   z^{67}+8 z^{65}+8 z^{63}+8
   z^{62}+8 z^{58}+8 z^{57}+8
   z^{53}+8 z^{52}+8 z^{50}+8
   z^{48}\\[1mm]
+8 z^{46}
+8 z^{44}+8
   z^{43}
+8 z^{40}+8 z^{39}+8
   z^{37}+8 z^{34}+8 z^{33}+8
   z^{28}+8 z^{26}\\[1mm]
+8 z^{23}+8
   z^{20}+8 z^{17}+8 z^{15}+8
   z^{13}
+8 z^{11}
+8 z^9+8 z^8
+8 z^7+8 z^5\\[1mm]
+8 z^4+4 z^2+8
   z+\frac{4}{z^2}+\frac{8}{z^3}+\frac{14}{z^4}+\frac{12}{z^5}
+\frac{14}{z^7}+\frac{4}{z^9}+\frac
   {4}{z^{11}}+\frac{6}{z^{13}}\bigg) \Phi^5(z)\\[2mm]
\text {\em modulo }16.
\end{multline}}%
\end{theorem}

We did not attempt to push this analysis further to moduli $32,$ $64,$ etc., 
since the computational effort seemed immodest. %unsurmountable.
With the (not very substantial) evidence of Theorems~\ref{thm:Sl2Z8}
and \ref{thm:Sl2Z16} (but see Remark~\ref{rem:Ga3}), 
we still expect the enhanced method
to be successful for any given $2$-power.

\begin{conjecture} \label{conj:SL2Z}
Let $\Phi(z)=\sum _{n\ge0} ^{}z^{2^n}$, and
let $\ga$ be a positive integer.
Then the generating function 
$\sum _{n\ge0} ^{}s_{n+1}(SL_2(\mathbb Z))\,z^n$, 
reduced modulo $2^\ga,$ 
can be expressed as a polynomial in $\Phi(z)$ 
with coefficients that are Laurent polynomials in
$z$ over the integers.
\end{conjecture}

\section{Subgroup numbers for the lift $\Ga_3(3)$}
\label{sec:Ga3}

Continuing in the spirit of the previous section, we now
consider the number of index-$n$-subgroups in the lift
$\Ga_3(3)$ (of the Hecke group $\mathfrak H(3)\cong PSL_2(\mathbb Z)$) 
modulo powers of~$2$. We shall see that, again, our method from
Section~\ref{sec:method} already fails for 
modulus~$8$. %??(see Theorems~\ref{thm:Ga38} and \ref{thm:Ga3}), while it
%fails for the modulus~$16$. 
While this can again be overcome by, instead, designing the computation
so that the target is modulus~$16$, the method then fails at the
level of modulus~$16$.
Moreover, for modulus~$16$, even the enhancement of the
method described in Appendix~\ref{appD} fails (see Remark~\ref{rem:Ga3}). 
This means that a new phenomenon, not
covered by our Ansatz, arises in the behaviour of the subgroup
numbers at the level of modulus~$16$. It would
be of great interest to find an explicit description of the hidden scheme
behind the mod-$16$ behaviour of the number of subgroups of index $n$
in $\Ga_3(3)$, and, more generally, of the behaviour modulo
{\it any} power of $2$.
%??By the enhancement of the method outlined
%in Appendix~\ref{appD}, we are nevertheless able to treat the
%subgroup numbers $s_{n}(\Ga_3(3))$ modulo~$16$ as well (see
%Theorem~\ref{thm:Ga316}).
%Again, 
%in view of the already substantial computational effort involved in
%the case of modulus~$16$, we did not try to push our analysis further
%to higher powers of $2$. Nevertheless, we conjecture that it should
%be possible to describe the generating function
%$\sum _{n\ge0} ^{}s_{n+1}(\Ga_3(3))\,z^n$, when the
%coefficients are reduced modulo~$2^\ga$, as a polynomial in $\Phi(z)$
%with coefficients that are Laurent polynomials in $z$ over the
%integers for all $\ga\ge1$, see Conjecture~\ref{conj:Ga3}.

We take $\Ga=\Ga_3(3)$ in \eqref{eq:Dey} and combine the resulting
formula with
\eqref{Eq:Gamma3HomCount}, the latter giving an explicit formula for 
the homomorphism numbers
$h_n:=\vert\Hom(\Ga_3(3),S_n)\vert$. Using the Guessing package
\cite{Guess}, we found
a recurrence of order $42$ for the sequence $(h_n/n!)_{n\ge0}$,
with coefficients that are polynomials 
in $n$ over $\mathbb Z$.
The validity of the recurrence was verified by
computing a certificate using Koutschan's {\sl Mathematica} package 
{\tt HolonomicFunctions} \cite{KoutAA}.\footnote{\label{foot:2}%
The computation took about one week, producing a certificate of 
28~megabytes. The coefficients of this
recurrence are polynomials in $n$ of degree up to~$105$.
In this case, we were not able to find a recurrence
with leading coefficient~$1$. (It may still exist.) The best that we
found in this direction was a recurrence with leading coefficient a
$1828$-digit number.
Again, although we did not try to prove it, it is likely that the recurrence
of order $42$ is the recurrence of minimal order.}
However, again, this recurrence is not suitable for our purpose, for which
we require a recurrence with coefficients that are polynomials in~$n$
over $\mathbb Z$, and with leading coefficient~$n$.
By an indeterminate Ansatz, we computed a
candidate for a recurrence of the desired form of order $60$, with polynomial
coefficients of degree at most $10$.\footnote{Again, we used the function
{\tt LinSolveQ} of the Guessing package \cite{Guess} in order to
solve the arising system of linear equations.}
To be precise, it is the uniquely determined recurrence of the form
$$
\sum_{k=0}^{60}
\bigg(\sum_{i=0}^{10}b(k,i)n^i\bigg)
\frac {\vert\Hom(\Ga_3(3),S_{n-k})\vert} {(n-k)!}=0,\quad \quad n\ge60,
$$
where
{\allowdisplaybreaks
\begin{align*}
b(0,0)&=b(0,2)=b(0,3)=b(0,4)=b(0,5)\\
&=b(0,6)=b(0,7)=b(0,8)=b(0,9)=b(0,10)=0,\\
b(0,1)&=1,\\
b(60,8)&=9649124343496238177846526221678676069879148435557456840677
\backslash\\
&\quad \quad \quad \quad 68567400990643919180258204664996863270960793634431477
\backslash\\
&\quad \quad \quad \quad 96875828563496094243333614632539311543926582958877938
\backslash\\
&\quad \quad \quad \quad 09887854513738722474642524334737161421912431106592005
\backslash\\
&\quad \quad \quad \quad 22984304410147101964876864298627928130880022459406799
\backslash\\
&\quad \quad \quad \quad 539461032349694733915947489297372243661012,\\
b(60,10)&=
b(60,9)=
b(60,4)\\
&=
b(59,10)=
b(59,9)=
b(59,8)=
b(59,7)=
b(59,6)=
b(59, 5)\\
&=
b(59, 4)= 
b(59, 3)= 
b(59, 2)= 
b(59, 1)= 
b(59, 0)\\
&= 
b(58,10)=
b(58,9)=
b(58,8)=
b(58,7)=
b(58,6)\\
&=
b(58, 5)= 
b(58, 4)\\
&= 
b(57,10)=
b(57,9)=
b(57,8)=
b(57,7)=
b(57,6)=
b(57, 5)= 
b(57, 4)\\
&= 
b(56,10)=
b(56,9)=
b(56,8)=
b(56,7)=
b(56,6)=
b(56,5)=
b(56,4)\\
&=
b(55,10)=
b(55,9)=
b(55,8)=
b(55,7)=
b(55,6)=
b(55, 5)= 
b(55, 4)\\
&= 
b(54,10)=
b(54,9)=
b(54,8)=
b(54,7)=
b(54,6)=
b(54, 5)= 
b(54, 4)\\
&= 
b(53,10)=
b(53,9)=
b(53,8)=
b(53,7)=
b(53,6)=
b(53,5)=
b(53,4)\\
&=
b(52,10)=
b(52,9)=
b(52,8)\\
&=
b(50,7)=
b(50,6)\\
&=
b(49,10)=
b(49,9)=
b(49,8)=
b(49,7)=
b(49,6)=
b(49,5)=
b(49,4)\\
&=
b(1,7)=
0.
\end{align*}}%
Subsequently, we checked
that this recurrence is a left-multiple of the certified
recurrence of order $42$, thereby establishing validity of this
candidate recurrence of order $60$.
This last recurrence was then converted into a linear differential equation 
with polynomial coefficients for the series
$$H(z):=\sum_{n=0}^\infty h_n\frac {z^n} {n!}
=\sum_{n=0}^\infty \vert\Hom(\Ga_3(3),S_n)\vert\frac {z^n} {n!}.$$
Finally, this last mentioned differential equation can be
translated into a Riccati-type differential equation for
the generating function 
\begin{equation} \label{eq:SGa3def}
S(z)=\sum _{n\ge0} ^{}s_{n+1}(\Ga_3(3))\,z^n
\end{equation}
for the subgroup numbers of $\Ga_3(3)$ in the same way
as we obtained \eqref{eq:Riccati16} in the previous section.
It turns out that this differential equation has integral
coefficients, so that it is amenable to our method from
Section~\ref{sec:method}. The differential equation
cannot be displayed here since this would 
require about 100 pages.\footnote{The integers appearing as coefficients have up
to 320 digits.}
 Its reduction modulo~$16$ is written out in
\eqref{eq:Ga3Riccati16}. By applying our method from
Section~\ref{sec:method} with the minimal polynomial for the modulus~$16$
(!) in place of the polynomial on the left-hand side of \eqref{eq:PhiRel} 
to \eqref{eq:Ga3Riccati16}, we
obtain the following theorem. It refines the parity result
\cite[Eq.~(6.3) with $\vert H\vert=1$, $q=3$, $m=3$]{KrMuAC}.

{
\allowdisplaybreaks
\begin{theorem} \label{thm:Ga38}
Let $\Phi(z)=\sum _{n\ge0} ^{}z^{2^n}$.
Then we have
\begin{multline} \label{eq:Ga38}
\sum _{n\ge0} ^{}s_{n+1}(\Ga_3(3))\,z^n\\[2mm]
=
4 z^{62}+
4 z^{53}+
4 z^{44}+
4 z^{35}+
6 z^{26}+
4 z^{20}+
4 z^{14}+
4 z^{12}+
4 z^{11}+
4 z^{10}\\[1mm]
+ 4 z^9+
4 z^5+
6 z^4+
4 z^3+
4 z^2+
4 z+
6+
\frac{7}{z^2}+
\frac{7}{z^3}+
\frac{3}{z^5}+
\frac{6}{z^6}\\[2mm]
+
 \left(
4 z^3+
4 z^2+
\frac{4}{z}+
\frac{6}{z^3}+
\frac{6}{z^4}+
\frac{6}{z^6}+
\frac{2}{z^7}
\right)  \Phi(z)\\[2mm]
+
\left(
4 z^8+
4 z^4+
4 z^3+
6 z^2+
4+
\frac{4}{z}+
\frac{6}{z^2}+
\frac{2}{z^3}+
\frac{5}{z^4}+
\frac{6}{z^5}+
\frac{6}{z^6}+
\frac{5}{z^7}
\right) \Phi^2(z)\\[2mm]
+
\left(
4 z^2+
\frac{4}{z^2}+
\frac{4}{z^3}+
\frac{2}{z^4}+
\frac{4}{z^5}+
\frac{4}{z^6}+
\frac{2}{z^7}
\right) \Phi^3(z)\quad \quad 
\text {\em modulo }8.
\end{multline}
\end{theorem}
}

\begin{proof}
{
\allowdisplaybreaks
The Riccati-type differential equation for 
$S(z)$ (as defined in \eqref{eq:SGa3def})
modulo $16$ is\footnote{We display the differential equation modulo~$16$
  in order to prepare for Remark~\ref{rem:Ga3}.}
{\allowdisplaybreaks
\begin{multline} \label{eq:Ga3Riccati16}
q_0(z)
+q_1(z)S(z) 
+q_2(z)S(z) {S'(z)} 
+q_3(z)S(z) {S'(z)}^2 
+q_4(z)S(z) {S'(z)}^3 \\[2mm]
+q_5(z)S(z) {S'(z)}^4 
+q_6(z)S(z) {S''(z)} 
+q_7(z)S(z) {S''(z)}^2 
+q_8(z)S(z) {S'''(z)} \\[2mm]
+q_9(z)S(z) {S'''(z)}^2 
+q_{10}(z)S(z) {S'''''(z)} 
+q_{11}(z)S(z) {S'(z)} {S''(z)} \\[2mm]
+q_{12}(z)S(z) {S'(z)} {S'''(z)} 
+q_{13}(z)S(z) {S'(z)}^2 {S'''(z)} 
+q_{14}(z)S(z)^2 
+q_{15}(z)S(z)^2 {S'(z)} \\[2mm]
+q_{16}(z)S(z)^2 {S'(z)}^2 
+q_{17}(z)S(z)^2 {S'(z)}^3 
+q_{18}(z)S(z)^2 {S'(z)}^4 
+q_{19}(z)S(z)^2 {S''(z)} \\[2mm]
+q_{20}(z)S(z)^2 {S'''(z)} 
+q_{21}(z)S(z)^2 {S'''(z)}^2 
+q_{22}(z)S(z)^2 {S''''(z)} 
+q_{23}(z)S(z)^2 {S'(z)} {S''(z)}\\[2mm] 
+q_{24}(z)S(z)^2 {S'(z)}^2 {S''(z)} 
+q_{25}(z)S(z)^2 {S'(z)} {S'''(z)} 
+q_{26}(z)S(z)^2 {S'(z)}^2 {S'''(z)}\\[2mm] 
+q_{27}(z)S(z)^3 
+q_{28}(z)S(z)^3 {S'(z)} 
+q_{29}(z)S(z)^3 {S'(z)}^2 
+q_{30}(z)S(z)^3 {S'(z)}^3 \\[2mm]
+q_{31}(z)S(z)^3 {S''(z)} 
+q_{32}(z)S(z)^3 {S'''(z)} 
+q_{33}(z)S(z)^3 {S'(z)} {S'''(z)} 
+q_{34}(z)S(z)^4 \\[2mm]
+q_{35}(z)S(z)^4 {S'(z)} 
+q_{36}(z)S(z)^4 {S'(z)}^2 
+q_{37}(z)S(z)^4 {S'(z)}^3 
+q_{38}(z)S(z)^4 {S''(z)}\\[2mm] 
+q_{39}(z)S(z)^4 {S'''(z)} 
+q_{40}(z)S(z)^4 {S'(z)} {S''(z)} 
+q_{41}(z)S(z)^4 {S'(z)} {S'''(z)} 
+q_{42}(z)S(z)^5 \\[2mm]
+q_{43}(z)S(z)^5 {S'(z)} 
+q_{44}(z)S(z)^5 {S'(z)}^2 
+q_{45}(z)S(z)^5 {S''(z)} 
+q_{46}(z)S(z)^5 {S'''(z)}\\[2mm] 
+q_{47}(z)S(z)^6 
+q_{48}(z)S(z)^6 {S'(z)} 
+q_{49}(z)S(z)^6 {S'(z)}^2 
+q_{50}(z)S(z)^6 {S''(z)} \\[2mm]
+q_{51}(z)S(z)^6 {S'''(z)} 
+q_{52}(z)S(z)^7 
+q_{53}(z)S(z)^7 {S'(z)} 
+q_{54}(z)S(z)^8 
+q_{55}(z)S(z)^8 {S'(z)} \\[2mm]
+q_{56}(z)S(z)^9 
+q_{57}(z)S(z)^{10} 
+q_{58}(z){S'(z)} 
+q_{59}(z){S'(z)} {S''(z)} 
+q_{60}(z){S'(z)} {S'''(z)}\\[2mm]
+q_{61}(z){S'(z)} {S'''(z)}^2 
+q_{62}(z){S'(z)} {S''''(z)} 
+q_{63}(z){S'(z)}^2 
+q_{64}(z){S'(z)}^2 {S''(z)} \\[2mm]
+q_{65}(z){S'(z)}^2 {S'''(z)} 
+q_{66}(z){S'(z)}^3 
+q_{67}(z){S'(z)}^3 {S''(z)} 
+q_{68}(z){S'(z)}^3 {S'''(z)} \\[2mm]
+q_{69}(z){S'(z)}^4 
+q_{70}(z){S'(z)}^5 
+q_{71}(z){S''(z)} 
+q_{72}(z){S''(z)} {S'''(z)} 
+q_{73}(z){S''(z)}^2 \\[2mm]
+q_{74}(z){S'''(z)} 
+q_{75}(z){S'''(z)}^2 
+q_{76}(z){S''''(z)} 
+q_{77}(z){S'''''(z)} 
=0\\[2mm]
\text {modulo }16,
\end{multline}}%
with coefficients $q_j(z)$, $j=0,1,\dots,77$ as displayed in
Appendix~\ref{appC}.

The differential equation \eqref{eq:Ga3Riccati16} has a unique solution
since comparison of coefficients of $z^{N}$ fixes the initial values,
and yields a recurrence for
the sequence $\big(s_n(\Ga_3(3))\big)_{n\ge1}$ 
which computes $s_{n+1}(\Ga_3(3))$ from terms
involving only $s_i(\Ga_3(3))$ with $i\le n$. 

Now we apply the method from Section~\ref{sec:method} with the
polynomial in \eqref{eq:minpol16} in place 
of the polynomial on the left-hand side of \eqref{eq:PhiRel} to
the differential equation \eqref{eq:Ga3Riccati16}. 
This yields the above result by means of a straightforward computer 
calculation.\footnote{The calculation being straightforward, it nevertheless
required a machine with substantial amount of memory
(we had available 32~gigabytes of memory, of which almost 50\% were used).}
}
\end{proof}

Also here, if we want to know criteria in terms of $n$ when
a subgroup number $s_n(\Ga_3(3))$ is congruent to a particular value 
modulo~$8$, then we must first
apply the algorithm from Section~\ref{sec:extr} to
the right-hand side of \eqref{eq:Ga38}. This leads to the identity
\begin{multline} \label{eq:Ga3S}
\sum _{n\ge0} ^{}s_{n+1}(\Ga_3(3))\,z^n
=
\left(
\frac{4}{z^4}+
\frac{4}{z^7}
\right) H_{3}(z)
+
 \left(
\frac{4}{z^4}+
\frac{4}{z^7}
\right)
    H_{1,1,1}(z)\\[2mm]
    +
\left(
4 z^2+
\frac{4}{z^2}+
\frac{4}{z^3}+
\frac{6}{z^4}+
\frac{4}{z^5}+
\frac{4}{z^6}+
\frac{6}{z^7}
\right)
    H_{1,1}(z)\\[2mm]
+
\left(
4
    z^8+
4 z^4+
4
    z^3+
6 z^2+
4+
\frac{4}{z}+
\frac{6}{z^2}+
\frac{6}{z^3}+
\frac{5}{z^4}+
\frac{6}{z^5}+
\frac{2}{z^6}+
\frac{5}{z^7}
\right) H_{1}(z)\\[2mm]
+
4
    z^{62}+
4 z^{53}+
4 z^{44}+
4 z^{35}+
6 z^{26}+
4 z^{20}+
4 z^{14}+
4 z^{12}+
4
    z^{11}+
4 z^{10}\\[1mm]
+
2 z^4+
2
    z^3+
4 z^2+
2+
\frac{6}{z}+
\frac{1}{z^2}+
\frac{4}{z^3}+
\frac{6}{z^4}+
\frac{1}{z^5}+
\frac{3}{z^6}\quad 
\text {modulo }8,
\end{multline}
from which we can extract the following explicit description of the
behaviour of the subgroup numbers of $\Ga_3(3)$ modulo~$8$.

\begin{theorem} \label{thm:Ga3}
The subgroup numbers $s_n(\Ga_3(3))$ 
obey the following congruences modulo $8:$
\begin{enumerate}
\item[(i)] $s_n(\Ga_3(3))\equiv 1$~{\em(mod~$8$)} if, and only if, 
$n=1,2,10,$ or if 
$n$ is of the form $2^\si-3$ for some $\si\ge4;$
\vspace{2mm}
\item[(ii)] $s_n(\Ga_3(3))\equiv 2$~{\em(mod~$8$)} if, and only if, 
$n=7, 9, 17, 18, 27, 42,$ or if 
$n$ is of one of the forms 
$$3\cdot 2^\si-3,\ 3\cdot 2^\si-6,\quad \text {for some }\si\ge5;
$$
\item[(iii)] $s_n(\Ga_3(3))\equiv 4$~{\em(mod~$8$)} if, and only if,
$n=3, 12, 22, 23, 36, 38, 39, 43, 46, 49, 50,\break 51, 53, 54, 63,$ or if 
$n$ is of one of the forms 
\begin{align} 
\notag
\kern1.1cm
&2^\si+6,\ 2^\si+7,\
2^\si+11,\
2^\si+14,\
2^\si+17,\
2^\si+18,\
2^\si+19,\
2^\si+21,\\
\label{eq:4siGa3}
& \kern1.5cm
\text {for some }\si\ge6,\\[1mm]
\notag
&2^\si+2^\ta-2,\
2^\si+2^\ta+1,\
2^\si+2^\ta+2,\
2^\si+2^\ta+3,\
2^\si+2^\ta+5,\
2^\si+2^\ta+10,\\
&2^\si+2^\ta+13,\
\quad \text {for some $\si,\ta$ with }\si\ge6\text { and
}5\le\ta\le\si-1,
\label{eq:4sitaGa3}
\\[1mm]
\notag
&2^\si+2^\ta+2^\nu-6,\
2^\si+2^\ta+2^\nu-3,\\
& \kern1.5cm
\text {for some $\si,\ta,\nu$ with }
\si\ge7,\ 6\le\nu\le\si-1,\text { and }5\le\ta\le\nu-1;
\label{eq:4sitanuGa3}
\end{align}
\item[(iv)] $s_n(\Ga_3(3))\equiv 5$~{\em(mod~$8$)} if, and only if, $n=5,$ or if 
$n$ is of the form 
$2^\si-6$ for some $\si\ge5;$
\vspace{2mm}
\item[(v)] $s_n(\Ga_3(3))\equiv 6$~{\em(mod~$8$)} if, and only if,
$n=6,11,$ or if 
$n$ is of one of the forms 
\begin{align} 
\label{eq:6siAGa3}
\kern1.1cm
&2^\si-2,\ 2^\si+3,\ 2^\si+4,\
\quad \text {for some }\si\ge4,\\[1mm]
\label{eq:6siAAGa3}
\kern1.1cm
&2^\si+1,\ 2^\si+2,\ 2^\si+13,\
\quad \text {for some }\si\ge5,\\[1mm]
\label{eq:6siBGa3}
\kern1.1cm
&2^\si+10,\
\quad \text {for some }\si\ge6,\\[1mm]
\notag
&2^\si+2^\ta-6,\
2^\si+2^\ta-3,\\
& \kern3cm
\text {for some $\si,\ta$ with }\si\ge7\text { and
}5\le\ta\le\si-2;
\label{eq:6sitaGa3}
\end{align}
\item[(vi)] in the cases not covered by items {\em(i)}--{\em(v),}
$s_n(\Ga_3(3))$ is divisible by $8;$
in particular, $s_n(\Ga_3(3))\not\equiv 3,7$~{\em(mod~$8$)} for all $n$.
\end{enumerate}
\end{theorem}

\begin{remarknu} \label{rem:Ga3}
\allowdisplaybreaks
In the application of the method from Section~\ref{sec:method} in the
proof of Theorem~\ref{thm:Ga3}, when we arrive at 
%??again the method of
%Section~\ref{sec:method} (with \eqref{eq:minpol16} in place of the
%polynomial on the left-hand side of \eqref{eq:PhiRel}) goes through
%up to the mod-8-level, 
the mod-8-level, we obtain 
\begin{multline} \label{eq:Ga316}
\sum _{n\ge0} ^{}s_{n+1}(\Ga_3(3))\,z^n
=
4 z^{62}+
4 z^{53}+
4 z^{44}+
4 z^{35}+
6 z^{26}+
4 z^{20}\\
+
4 z^{14}+
4 z^{12}+
4 z^{11}+
4 z^9+
4 z^6+
4 z^4+
\frac{2}{z}+
\frac{4}{z^2}+
\frac{3}{z^3}+
\frac{6}{z^4}\\[2mm]
+
\left(
4 z^4+
4+
\frac{6}{z^2}+
\frac{4}{z^3}+
\frac{6}{z^5}
\right) \Phi(z)
+
\left(
4 z^3+
\frac{4}{z}+
\frac{2}{z^3}+
\frac{4}{z^4}+
\frac{2}{z^6}
\right)
    \Phi^2(z)\\[2mm]
+
\left(
4 z^2+
\frac{4}{z^2}+
\frac{4}{z^3}+
\frac{2}{z^4}+
\frac{4}{z^5}+
\frac{4}{z^6}+
\frac{2}{z^7}
\right) \Phi^3(z)\\[2mm]
+
\left(
4 z^8+
4 z^4+
4 z^3+
6 z^2+
12+
\frac{4}{z}+
\frac{2}{z^2}+
\frac{6}{z^3}+
\frac{1}{z^4}+
\frac{2}{z^5}+
\frac{2}{z^6}+
\frac{1}{z^7}
\right) \Phi^4(z)\\[2mm]
+
\left(
4 z^2+
\frac{4}{z^2}+
\frac{6}{z^4}+
\frac{4}{z^5}+
\frac{6}{z^7}
\right) \Phi^5(z)\quad 
\text{modulo }8.
\end{multline}
%??but fails to find polynomial solutions for the system of equations 
%at the mod-16-level.\footnote{The corresponding computation 
However, the system of equations for the next level, the mod-16-level,
has no polynomial solutions.\footnote{The corresponding computation 
took almost 5 hours, using 94\% of the 32~gigabytes of memory of the
machine on which the computation was performed.} 
Even the enhancement of our method described in Appendix~\ref{appD}
fails. (There are actually several problems arising.
It turns out that, due to the reduction modulo~$2$, the variables
$b_i(z)$ expressed in
\eqref{eq:bi(z)} do not solve the system \eqref{eq:abcd} unless one
puts further restrictions on
$a^{(o)}(z),a^{(e)}(z),\dots,d^{(o)}(z),d^{(e)}(z)$. But even if we
ignore that and continue to follow the procedure described in
Appendix~\ref{appD}, then a contradiction arises at a later point: 
one of the factors of the polynomial $P(z)$ turns
out to be $(1+z)^{10}$, and the congruence \eqref{eq:Pjmj} with
$P_j^{m_j}(z)=(1+z)^{10}$ has no solution.) This {\it proves}
that it is impossible to find a polynomial in $\Phi(z)$ 
with coefficients that are Laurent polynomials in
$z$ over the integers which agrees with the generating function
for the subgroup numbers of $\Ga_3(3)$ modulo~$16$.
%??Nevertheless, by applying the enhanced
%method from Appendix~\ref{appD} to this last system of equations,
%a solution modulo~$16$ can still be found. We display the result 
%in the next theorem.
\end{remarknu}

%??\begin{theorem} \label{thm:Ga316}
%Let $\Phi(z)=\sum _{n\ge0} ^{}z^{2^n}$.
%Then we have
%{\allowdisplaybreaks
%\begin{multline} \label{eq:Ga316}
%\sum _{n\ge0} ^{}s_{n+1}(\Ga_3(3))\,z^n\\[2mm]
%=
%\\[2mm]
%\text {\em modulo }16.
%\end{multline}}%
%\end{theorem}%%
%
%We did not attempt to push this analysis further to moduli $32,$ $64,$ etc., 
%since the computational effort seemed immodest. %unsurmountable.
%With the (not very substantial) evidence of Theorems~\ref{thm:Ga38}
%and \ref{thm:Ga316}, we still expect the enhanced method
%to be successful for any given $2$-power.%
%
%\begin{conjecture} \label{conj:Ga3}
%Let $\Phi(z)=\sum _{n\ge0} ^{}z^{2^n}$, and
%let $\ga$ be a positive integer.
%Then the generating function 
%$\sum _{n\ge0} ^{}s_{n+1}(\Ga_3(3))\,z^n$, 
%reduced modulo $2^\ga,$ 
%can be expressed as a polynomial in $\Phi(z)$ 
%with coefficients that are Laurent polynomials in
%$z$ over the integers.
%\end{conjecture}

\section{A variation I: free subgroup numbers for lifts of Hecke groups}
\label{sec:Hecke}

In this section, we consider the functional equation
\begin{equation} \label{eq:2meq}
zf^{2^h}(z)-f(z)+1=0,
\end{equation}
which generalises the functional equation \eqref{eq:CatEF} for the
generating function of Catalan numbers. It is easy to see that this
equation has a unique formal power series solution. The
coefficients of this uniquely determined series can be calculated
explicitly by means of the Lagrange inversion formula, the result
being 
\begin{equation} \label{eq:FCat} 
\coef{z^n}f(z)= \frac {1} {n}\binom {2^hn}{n-1},
\end{equation}
but this will not be relevant here.\footnote{See
Footnote~\ref{F:Lucas}.}
Again, the numbers in \eqref{eq:FCat} are special instances of
numbers that are now commonly known as {\it Fu\ss--Catalan numbers}
(cf.\ the paragraph containing \eqref{eq:EqH7}).

Our aim is to determine the coefficients of $f(z)$ modulo powers of $2$.
Our solution of this problem is that, again, 
the series $f(z)$ can be expressed as a polynomial in a ``basic" series. 
Here, this basic series is
\begin{equation} \label{eq:Phim}
\Phi_h(z)=\sum _{n\ge0} ^{}z^{2^{nh}/(2^h-1)}.
\end{equation}
It will turn out (see Corollary~\ref{thm:freeHecke}) that an adaptation
of the proof of the theorem below will allow us to treat as well
the behaviour, modulo powers of $2$, of free subgroup numbers of 
lifts of Hecke groups $\mathfrak H(q)$, with $q$ a Fermat prime.

The theorem below, in a certain sense, extends Theorem~\ref{thm:Cat}.
It does not, however, reduce to it for $h=1$, due to the choice that,
in the proof below,
the reductions in our algorithm are based on
the polynomial relation \eqref{eq:Phim2al} for the basic series
$\Phi_h(z)$, which, for $h=1$, is ``weaker" than the relation
\eqref{eq:PhiRel} which is used in the proof of Theorem~\ref{thm:Cat}. 

\begin{theorem} \label{thm:2mf}
For a positive integer $h,$ 
let $\Phi_h(z)=\sum _{n\ge0} ^{}z^{2^{nh}/(2^h-1)},$ and
let $\al$ be a further positive integer.
Then the unique solution $f(z)$ to \eqref{eq:2meq}{\em,}
reduced modulo $2^{2^{\al h}},$ 
can be expressed as a polynomial in $\Phi_h(z)$ of degree at most
$2^{(\al+1)h}-1$ with coefficients that are Laurent polynomials in 
$z^{1/(2^h-1)}$ over the integers.
\end{theorem}

\begin{proof}
For ease of notation, we replace $z$ by $z^{2^h-1}$ in
\eqref{eq:2meq}, thereby obtaining the equation
\begin{equation} \label{eq:2meqA}
z^{2^h-1}\tilde f^{2^h}(z)-\tilde f(z)+1=0,
\end{equation}
with $\tilde f(z)=f(z^{2^h-1})$. We now have
to prove that, modulo $2^{2^{\al h}}$, the series $\tilde
f(z)$ can be expressed as a polynomial in
\begin{equation} \label{eq:Phimdef} 
\tilde \Phi_h(z)=\sum _{n=0} ^{\infty}z^{2^{nh}}
\end{equation}
of degree at most $2^{(\al+1)h}-1$ with coefficients that are Laurent 
polynomials in $z$.

It is readily verified that
\begin{equation} \label{eq:Phim2}
\tilde\Phi_h^{2^h}(z) +\tilde\Phi_h(z)+z=0\quad \text{modulo 2},
\end{equation}
whence
\begin{equation} \label{eq:Phim2al}
\left(\tilde\Phi_h^{2^h}(z) +\tilde\Phi_h(z)+z\right)^{2^{\al h}}=0\quad 
\text{modulo }2^{2^{\al h}}.
\end{equation}
We modify our Ansatz \eqref{eq:Ansatz} to
\begin{equation} \label{eq:Ansatzm}
\tilde f(z)=\sum _{i=0} ^{2^{(\al+1)h}-1}a_i(z)\tilde\Phi_h^i(z)\quad \text 
{modulo
}2^{2^{\al h}},
\end{equation}
where the $a_i(z)$'s
are (at this point) undetermined Laurent polynomials in $z$.

Next, we gradually find approximations $a_{i,\be}(z)$ to $a_i(z)$ such that
\eqref{eq:2meqA} holds modulo $2^\be$, for $\be=1,2,\dots,
2^{\al h}$. To start the procedure, we consider the differential equation
\eqref{eq:2meqA} modulo $2$, with
\begin{equation} \label{eq:AnsatzA1}
\tilde f(z)=\sum _{i=0} ^{2^{(\al+1)h}-1}a_{i,1}(z)\tilde\Phi_h^i(z)\quad \text {modulo
}2.
\end{equation}
We substitute the Ansatz \eqref{eq:AnsatzA1} in \eqref{eq:2meqA}, reduce
high powers of $\tilde\Phi_h(z)$ by using Relation~\eqref{eq:Phim2al}, 
reduce the resulting expression modulo~2, thereby
taking advantage of the elementary fact that $\tilde\Phi_h'(z)=1$
modulo $2$, and we finally see that the left-hand side of
\eqref{eq:2meqA} becomes a polynomial in
$\tilde\Phi_h(z)$ of degree at most $2^{(\al+1)h}-1$ 
with coefficients that are Laurent polynomials in $z$. 
Now we compare coefficients of powers $\tilde\Phi_h^k(z)$,
$k=0,1,\dots,2^{(\al+1)h}-1$. This yields a
system of $2^{(\al+1)h}$ equations (modulo~$2$)
for the unknown Laurent polynomials $a_{i,1}(z)$,
$i=0,1,\dots,2^{(\al+1)h}-1$. Since we have already done similar
computations several times before, we content ourselves with stating
the result: all Laurent polynomials $a_{i,1}(z)$ must be zero, except for
$a_{0,1}(z)$ and $a_{2^{\al h},1}(z)$, which are given by
\begin{align} \notag
a_{0,1}(z)&=
\sum _{k=0} ^{{\al}-1}z^{2^{kh}-1},\\[2mm]
\label{eq:Ansatza1}
a_{2^{\al h},1}(z)&=z^{-1}.
\end{align}
 
\medskip
After we have completed the ``base  step," we now proceed with the
iterative steps described in Section~\ref{sec:method}. Our Ansatz
here (replacing the corresponding one in
\eqref{eq:Ansatz2}--\eqref{eq:Ansatz2b}) is
\begin{equation} \label{eq:AnsatzA2}
\tilde f(z)
=\sum _{i=0} ^{2^{(\al+1)h}-1}a_{i,\be+1}(z)\tilde\Phi_h^i(z)\quad 
\text{modulo }2^{\be+1},
\end{equation}
with
\begin{equation} \label{eq:AnsatzA2a}
a_{i,\be+1}(z):=a_{i,\be}(z)+2^{\be}b_{i,\be+1}(z),\quad 
i=0,1,\dots,2^{(\al+1)h}-1,
\end{equation}
where the
coefficients $a_{i,\be}(z)$ are supposed to provide a solution
$$\tilde f_{\be}(z)=\sum _{i=0}
^{2^{(\al+1)h}-1}a_{i,\be}(z)\tilde\Phi_h^i(z)$$ 
to
\eqref{eq:2meqA} modulo~$2^\be$. This Ansatz, substituted in
\eqref{eq:2meqA}, produces the congruence
\begin{equation} \label{eq:iter}
z^{2^h-1}\tilde f_{\be}^{2^h}(z)-\tilde f_{\be}(z)
+2^\be\sum _{i=0} ^{2^{(\al+1)h}-1}b_{i,\be+1}(z)\tilde\Phi_h^i(z)
+1=0
\quad 
\text {modulo }2^{\be+1}.
\end{equation}
By our assumption on $\tilde f_{\be}(z)$, we may divide by $2^\be$.
Comparison of powers of $\tilde\Phi_h(z)$ then yields a system of congruences
of the form
\begin{equation} \label{eq:bi+1} 
b_{i,\be+1}(z)+\text {Pol}_i(z)=0\quad 
\text {modulo }2,\quad \quad 
i=0,1,\dots,2^{(\al+1)h}-1,
\end{equation}
where $\text {Pol}_i(z)$, $i=0,1,\dots,2^{(\al+1)h}-1$, are certain
Laurent polynomials with integer coefficients. This system being trivially
uniquely solvable, we have proved that, for an arbitrary positive
integer $\al$, the modified algorithm that we have presented here
will produce a solution $\tilde f_{{
2^{\al h}}}(z)$ to \eqref{eq:2meqA} modulo $2^{2^{\al h}}$ which is a
polynomial in $\tilde\Phi_h(z)$ with coefficients that are Laurent polynomials in
$z$.
\end{proof}

It has been shown in \cite{MuHecke} that the parity pattern of free
subgroup numbers in Hecke groups $\mathfrak H(q)$, $q$ a Fermat prime,
coincides with the parity pattern of (special) Fu\ss--Catalan numbers.
More precisely, 
let $f_\la^{(q)}$ denote the number of free subgroups of index $2q\la$ in the
Hecke group $\mathfrak H(q)$. (For indices not divisible by $2q$, no free
subgroups exist in $\mathfrak H(q)$.) Then (see \cite[Eq.~(37)]{MuHecke})
$$
f_\la^{(q)}=\frac {1} {\la}\binom {(q-1)\la}{\la-1}\quad \text{modulo }2.
$$ 
The reader should keep in mind that $q-1$ is a $2$-power.
Theorem~\ref{thm:2mf} says that the generating function for the
Fu\ss--Catalan numbers \eqref{eq:FCat}, when reduced modulo a 
given power of $2$, can be expressed as a polynomial in $\Phi_h(z)$.
We are now going to show that the same is true for the
generating function for free subgroup numbers in the Hecke group
$\mathfrak H(q)$, although the equation it satisfies is different from
the functional equation \eqref{eq:2meq} for the generating function of
Fu\ss--Catalan numbers.
In the corollary below, we present actually a more general result:
even the generating function for free subgroup numbers of the lift
$\Ga_m(q)$, when reduced modulo a 
given power of $2$, can be expressed as a polynomial in $\Phi_h(z)$
in the case where $q$ is a Fermat prime.
In a certain sense, this extends Theorem~\ref{thm:freie-m}, although
it does not reduce to it for $h=1$. Again, the reason lies in the choice 
that, in the proof below, the reductions in our algorithm are based on
the polynomial relation \eqref{eq:Phim2al} for the basic series
$\Phi_h(z)$, which, for $h=1$, is ``weaker" than the relation
\eqref{eq:PhiRel} which is used in the proof of Theorem~\ref{thm:freie-m}. 
On the other hand, the corollary does largely extend
the parity result \cite[Cor.~A']{MuHecke}.

\begin{corollary} \label{thm:freeHecke}
Let $q=2^{2^{f}}+1$ be a Fermat prime, and
let $\ga$ be some positive integer. 
Then, for every positive integer $m,$ 
the generating function $F_m(q;z)=1+
\sum_{\lambda\geq1}
f_{\lambda}^{(q)}(m) z^\lambda$ of free
subgroup numbers of $\Ga_m(q),$
when reduced modulo $2^{\ga},$ 
can be expressed as a polynomial in $\Phi_{2^f}(z)$ of degree at most
$2^{2^f}\ga-1$ with coefficients that are Laurent polynomials in
$z^{1/(q-2)},$ where the series $\Phi_h(z)$ is defined as in
\eqref{eq:Phim}.
\end{corollary}

\begin{proof}
In view of Proposition~\ref{Prop:flambdaq2Part}, the assertion
is trivially true for even $m$, the polynomial in $\Phi_{2^f}(z)$ being a
polynomial of degree zero in this case. 
We may thus assume from now on that $m$ is odd.

Equation \eqref{Eq:FmqDiffEq} provides a Riccati-type differential
equation for $F_m(q;z)=1+\break zG_m(q;z)$.
Moreover, this equation, considered modulo~2, is the same for every
odd $m$. Namely, we have
\begin{multline*}
\tfrac {1} {z}\big(F_{m}(q;z)-1\big) = A_0(\mathfrak{H}(q))
 + \sum_{\mu=1}^{q-1}
\sum_{\nu=1}^\mu
\underset{\mu_1+\cdots+\mu_\nu=\mu}{\sum_{\mu_1,\ldots,\mu_\nu>0}}
\binom{\mu}{\mu_1,\ldots,\mu_\nu}\big(\nu!\, (2q)^\nu\big)^{-1}
A_\mu(\mathfrak{H}(q)) z^\mu \\[2mm]
\times\prod_{j=1}^\nu
\Big(\tfrac {1} {z}\big(F_{m}(q;z)-1\big)\Big)^{(\mu_j-1)} \quad \quad 
\text {modulo }2.
\end{multline*}
Moreover, it is shown in \cite[Prop.~2]{MuHecke}
that, modulo~2, this differential equation reduces to
\begin{equation} \label{eq:sq2} 
zF_m^{q-1}(q;z)-F_m(q;z)+1=0\quad 
\text{modulo }2.
\end{equation}
The latter statement means that reduction of coefficients modulo~2
and usage of the simple fact that 
\begin{equation} \label{eq:sq''2} 
F_m''(q;z)=0\quad \text{modulo }2
\end{equation}
leads from the original differential equation
\eqref{Eq:FmqDiffEq} for $F_m(q;z)=1+G_m(q;z)$ to the congruence
\eqref{eq:sq2}. With $q-1$ being a power of $2$ by assumption, we observe
that, disregarding the restriction to modulus~2, 
Equation~\eqref{eq:sq2} is the special case of \eqref{eq:2meq} where
$h=2^f$. In particular,
if, for the moment, we assume that $\ga=2^{\al 2^f}$, for some positive
integer $\al$, then we
see that the base step of the Ansatz outlined (and applied) in the
proof of Theorem~\ref{thm:2mf} (with $h=2^f$) can be successfully performed here: it would
yield exactly the same result as there, namely \eqref{eq:AnsatzA1}
with the Laurent polynomials $a_{i,1}(z)$ being given in the paragraph
containing \eqref{eq:Ansatza1}. 

However, also the subsequent iterative steps would just work in the
same way as in the preceding proof! Indeed, transform the Riccati-type
differential equation \eqref{Eq:FmqDiffEq} 
for $F_m(q;z)$ by the substitution $z\mapsto
z^{q-2}$ (in analogy with the substitution leading to
\eqref{eq:2meqA}). This yields a Riccati-type differential equation
for $F_m(q;z^{q-2})$. A fine point to be observed here is that, in this
equation, the coefficients will not necessarily be integral; due to
the substitution rule for differentials, denominators that are powers of
$(q-2)$ may occur. As in the proof of Theorem~\ref{thm:2mf},
we now continue with the Ansatz
\begin{equation} \label{eq:Ansatzs2}
F_m(q;z^{q-2})
=\sum _{i=0} ^{2^{(\al+1)2^f}-1}a_{i,\be+1}(z)\tilde\Phi_{2^f}^i(z)\quad 
\text{modulo }2^{\be+1},
\end{equation}
with
\begin{equation} \label{eq:Ansatzs2a}
a_{i,\be+1}(z):=a_{i,\be}(z)+2^{\be}b_{i,\be+1}(z),\quad 
i=0,1,\dots,2^{(\al+1)2^f}-1
\end{equation}
(which is analogous to \eqref{eq:AnsatzA2}--\eqref{eq:AnsatzA2a}),
where the
coefficients $a_{i,\be}(z)$ are supposed to provide a solution
$$F_{m,\be}(q;z)=\sum _{i=0}
^{2^{(\al+1)2^f}-1}a_{i,\be}(z)\tilde\Phi_{2^f}^i(z)$$ 
to the differential equation for $F_m(q;z^{q-2})$ modulo~$2^\be$.
The fact that reduction
modulo~2 and usage of \eqref{eq:sq''2} leads from the original
differential equation for $F_m(q;z)$ to \eqref{eq:sq2} implies that
substitution of the Ansatz \eqref{eq:Ansatzs2}--\eqref{eq:Ansatzs2a}
in the differential equation for $F_m(q;z^{q-2})$
yields an equation
completely analogous to \eqref{eq:iter}, namely
\begin{multline*} %\label{eq:iter}
z^{q-2}F_{m,\be}^{q-1}(q;z)-F_{m,\be}(q;z)
+2^\be\sum _{i=0} ^{2^{\al+h}-1}b_{i,\be+1}(z)\tilde\Phi_h^i(z)\\[2mm]
+1+T(z,F_{m,\be}(q;z))=0
\quad 
\text {modulo }2^{\be+1}.
\end{multline*}
Here, $T(z,F_{m,\be}(q;z))$ consists only of terms that may depend on
$F_{m,\be}(q;z)$ but {\it do not\/} depend on the
$b_{i,\be+1}(z)$'s.
The rest of the procedure is then as in the preceding proof:
we divide by $2^\be$,
compare powers of $\tilde\Phi_h(z)$, and obtain a system of congruences
of the form \eqref{eq:bi+1}, which is trivially solvable.
The powers of $(q-2)$ that may appear in the denominators of the
coefficients in the polynomials involved here are disposed of by
interpreting them appropriately as elements of $\mathbb Z/2^\ga\mathbb
Z$. 

Finally, if we are able to express $F_m(q;z)$ as a polynomial in
$\Phi_{2^f}(z)$ modulo $2^\ga=2^{2^{\al 2^f}}$ for all $\al$,
then the same assertion must hold for {\it every} $\ga$.
\end{proof}

In order to illustrate the algorithm described in the last proof, let
us consider the case of the Hecke group $\mathfrak H(5)$, that is,
the case of Corollary~\ref{thm:freeHecke} where $f=1$. 
The Riccati-type differential equation for the series 
$G_m(z):=G_m(5;z)=
\sum_{\la=0}^\infty f_{\la+1}^{(5)}(m)z^\la$ that one obtains from
\eqref{Eq:FmqDiffEq}
in this special case reads
\begin{multline*}
G_m(z)
=189m^4
+4600m^3
   z G_m(z)
+1430 m^2z^2 G_m^2(z)
+80 mz^3
   G_m^3(z)\\[2mm]
+z^4 G_m(z)^4
+14300m^3 z^2
   G_m'(z)
+2400 m^2z^3
   G_m(z)G_m'(z)
+60m
   z^4 G_m^2(z) G_m'(z)\\[2mm]
+300m^2 z^4 \big(G_m'(z)\big)^2
+8000 m^3z^3
   G_m''(z)
+400 m^2z^4
   G_m(z)G_m''(z)
+1000 m^3z^4 G_m'''(z).
\end{multline*}

Since $G_m(z)=G_m(5;z)=\frac {1} {z}\big(F_m(5;z)-1\big)$, we obtain
the differential equation
\begin{multline} \label{eq:F5}
1+ (189 m^4  -  300 m^3 + 130 m^2 - 20 m +1) z\\[2mm]
+((300m^3-260m^2+60m-4)z-1) F_m(5;z)
+(130m^2-60m+6)z
   F_m^2(5;z)\\[2mm]
+(20m-4)z F_m^3(5;z)
+z
   F_m^4(5;z)
+(4300m^3-1000m^2+60m)z^2 F_m'(5;z)\\[2mm]
+(1000m^2-120m)z^2 
   F_m(5;z) F_m'(5;z)
+60m
   z^2 F_m(5;z)^2 F_m'(5;z)
+300m^2 z^3 F_m'(5;z)^2\\[2mm]
+(5000m^3-400m^2)z^3 
   F_m''(5;z)
+400m^2 z^3
   F_m(5;z) F_m''(5;z)
+1000 m^3z^4 F_m'''(5;z)=0
\end{multline}
for $F_m(5;z)$.
We have implemented the algorithm described in the proof of
Theorem~\ref{thm:freeHecke} for this differential equation.
For the modulus $16$, it produces the following result.
(It is independent of $m$ because of the high divisibility of the
coefficients in the differential equation \eqref{eq:F5} by powers of
$2$. The parameter $m$ will show up for
$2$-powers higher than $2^4=16$.)

\begin{theorem} \label{cor:freiH5}
Let $\Phi_2(z)=\sum _{n\ge0} ^{}z^{4^{n}/3},$ as before.
Then, for all positive odd integers $m$, the generating function 
$F_m(5;z)=1+\sum_{\la\ge1} f_\la^{(5)}(m)z^\la$ for the free subgroup
numbers of\/ $\Ga_m(5)$ satisfies
\begin{multline} %\label{}
F_m(5;z)=
4z+1
+12 z^{2/3}
   \Phi_2(z)
+10 z^{1/3} \Phi_2^2(z)
+12 \Phi_2^3(z)
+\left(4
   z^{2/3}+{7}{z^{-1/3}}\right) \Phi_2^4(z)\\[2mm]
+4 z^{1/3}
   \Phi_2^5(z)
+4  \Phi_2^6(z)
+{12{z^{-1/3}}
   \Phi_2^7(z)}
+8 z^{1/3}
   \Phi_2^8(z)
+4
   \Phi_2^9(z)\\[2mm]
+{2{z^{-1/3}}
   \Phi_2^{10}(z)}
+12
   \Phi_2^{12}(z)
+{12{z^{-1/3}}
   \Phi_2^{13}(z)}\quad 
\text{\em modulo }16.
\end{multline}
\end{theorem}
Clearly, coefficient extraction from powers of $\Phi_2(z)$ 
(and, more generally, from powers of $\Phi_h(z)$) can be accomplished
by appropriately adapting the results in Section~\ref{sec:extr}.

\section{A variation II: 
subgroup numbers for Hecke groups}
\label{sec:H5}

In Section~\ref{sec:PSL2Z}, we proved that the generating function
for the subgroup numbers of
the inhomogeneous modular group $PSL_2(\mathbb Z)\cong \mathfrak
H(3)$, when reduced modulo a power of $2$, can always be expressed
as a polynomial in the basic series $\Phi(z)$ with coefficients that
are Laurent polynomials in $z$. 
Here, we discuss possible extensions of this result to Hecke groups 
$\mathfrak H(q)$, where $q$ is a Fermat prime. 
Again, we have to modify the original method from
Section~\ref{sec:method} by using the series $\Phi_h(z)$ defined in
\eqref{eq:Phim} (for suitable $h$) instead of $\Phi(z)$.
We conjecture (see Conjecture~\ref{conj:Hq}) that
this variation of our method will be successful for
arbitrary Fermat primes $q$. If $q=5$, we are actually able to
demonstrate this conjecture, thereby largely refining the $q=5$ case of
\cite[Theorem~B]{MuHecke}.

\begin{theorem} \label{thm:H5}
With notation from the previous section, 
let $\Phi_2(z)=\sum _{n\ge0} ^{}z^{4^{n}/3},$ and
let $\al$ be a positive integer.
Then the generating function $S(z)=S_{\mathfrak H(5)}(z)$
{\em(}see the first paragraph of Section~{\em\ref{sec:PSL2Z}} 
for the definition{\em)}, 
reduced modulo $2^{4^{\al}},$ 
can be expressed as a polynomial in $\Phi_2(z)$ of degree at most
$4^{\al+1}-1$ with coefficients that are Laurent polynomials in $z^{1/3}$
over the integers.
\end{theorem}

\begin{proof}
Let
$$
h(n)=\frac {1} {n!}h_{C_2}(n)h_{C_5}(n).
$$
Using the routine to compute recurrences for the Hadamard product of
recursive sequences, implemented in
{\tt gfun} \cite{SaZiAA} and {\tt GeneratingFunctions} \cite{MallAA}
(cf.\ \cite[Theo\-rem~6.4.12]{StanBI} for the theoretical background),
one obtains the recurrence 
{\allowdisplaybreaks
\begin{multline} \label{eq:h5-1}
  n (16 - 72 n + 174 n^2 - 155 n^3 + 65 n^4 - 13 n^5 + n^6) h(n)\\[2mm]
- (184 - 620 n + 854 n^2 - 555 n^3 + 177 n^4 - 25 n^5 + 
     n^6) h(n-1) \\[2mm]
- (856 - 
     1636 n + 1250 n^2 - 479 n^3 + 101 n^4 - 13 n^5 + n^6) h(
     n-2) \\[2mm]
- 4 (n-6) (n-3)^2 (-7 + 3 n) h(n-3) 
+ 
  8 (n-7) (n-4) (3 n-7) h(n-4) \\[2mm]
- (1136 - 856 n + 1292 n^2 - 
     2930 n^3 + 3115 n^4 - 1718 n^5 + 516 n^6 - 80 n^7 + 5 n^8) h(
      n-5) \\[2mm]
- 
  2 (1856 - 5376 n + 6828 n^2 - 4868 n^3 + 2174 n^4 - 651 n^5 + 
     133 n^6 - 17 n^7 + n^8) h(n-6) \\[2mm]
- 
  4 (n-6) (n-5) (n-3)^2 (n-2) (3 n-7) h(n-7) 
+ 
 8 (n-7) (n-6) (n-4) (n-3) (3 n-7) h(n-8)\\[2mm] 
- 
  16 (n-8) (n-7) (n-5) (-7 + 3 n) h(n-9) \\[2mm]
-(n-9) (n-8) (n-6) (n-3) (16 + 12 n - 16 n^2 - 5 n^3 + 
     15 n^4 - 7 n^5 + n^6) h(n-10) 
=0
\end{multline}}%
for the sequence $\big(h(n)\big)_{n\ge0}$.
Since the leading coefficient (i.e., the coefficient of $h(n)$) is not
$n$, this recurrence is not suitable for being translated into a
Riccati-type differential equation with integral coefficients 
via \eqref{eq:Dey}, to which we can apply our method from 
Section~\ref{sec:method}. 
Using Euclidean division of difference operators, 
one can see that we also have 
\begin{multline} \label{eq:h5-2}
n
   h(n)
-h(n-1)
-h(
   n-2)
-\left(5
   n^2-11 n-44\right)
   h(n-5)
-2 (n-4)
   (n-2)
   h(n-6)\\[2mm]
+12
   h(n-7)
-4
   h(n-8)
-\left(n^4-2
   0 n^3+95 n^2+260
   n-2000\right)
   h(n-10)\\[2mm]
+4 (9 n-85)
   h(n-11)
-8 \left(n^2-19 n+89\right)
   h(n-12)\\[2mm]
+4 (n-14)
   (n-13) (n-11) (n-7)
   h(n-15)\\[2mm]
-4 (n-15)
   (n-14) (n-12) (n-9)
   h(n-16)
=0.
\end{multline}
If we now apply the procedure of converting such a recurrence for
homomorphism numbers (divided by $n!$) into a Riccati-type
differential equation for the generating function of the
corresponding subgroup numbers as explained in the paragraph containing
\eqref{eq:Sdef}, then we obtain the
differential equation
\begin{multline} \label{eq:diffH5}
224
   z^{15}-256 z^{14}+40
   z^{11}-56 z^{10}+100 z^9+4
   z^7-12 z^6+16 z^5+26
   z^4+z+1\\[2mm]
+   \left(736 z^{16}-824
   z^{15}+48 z^{12}-36
   z^{11}+276 z^{10}+14 z^6+44
   z^5-1\right)S(z)\\[2mm]
+   \left(448 z^{17}-488
   z^{16}+8 z^{13}+162
   z^{11}+2 z^7+5
   z^6\right)S^2(z)\\[2mm]
+ \left(80 z^{18}-84
   z^{17}+26
   z^{12}\right)S^3(z)
+ \left(4 z^{19}-4
   z^{18}+z^{13}\right)S^4(z)\\[2mm]
+   \left(448 z^{17}-488
   z^{16}+8 z^{13}+162
   z^{11}+2 z^7+5
   z^6\right)S'(z)\\[2mm]
+   \left(12 z^{19}-12 z^{18}+3
   z^{13}\right)(S')^2(z)
+ \left(240
   z^{18}-252 z^{17}+78
   z^{12}\right)S(z)
   S'(z)\\[2mm]
+\left(24
   z^{19}-24 z^{18}+6
   z^{13}\right)S^2(z)
   S'(z) 
+   \left(80 z^{18}-84
   z^{17}+26
   z^{12}\right)S''(z)\\[2mm]
+   \left(16 z^{19}-16 z^{18}+4
   z^{13}\right)S(z)S''(z)
+   \left(4 z^{19}-4
   z^{18}+z^{13}\right)S'''(z)
=0.
\end{multline}
For convenience, we replace $z$ by $z^3$ in 
\eqref{eq:diffH5}. Writing $\tS(z)=S(z^3)$, the above differential
equation translates into
{\allowdisplaybreaks
\begin{multline} \label{eq:diffH5A}
224
   z^{45}-256 z^{42}+40
   z^{33}-56 z^{30}+100
   z^{27}+4 z^{21}-12
   z^{18}+16 z^{15}+26
   z^{12}+z^3+1\\[2mm]
+   \left(736 z^{48}-824
   z^{45}+48 z^{36}-36
   z^{33}+276 z^{30}+14
   z^{18}+44
   z^{15}-1\right)\tS(z)\\[2mm]
+   \left(448 z^{51}-488
   z^{48}+8 z^{39}+162
   z^{33}+2 z^{21}+5
   z^{18}\right)\tS^2(z)\\[2mm]
+ \left(80 z^{54}-84
   z^{51}+26
   z^{36}\right)\tS^3(z)
+ \left(4 z^{57}-4
   z^{54}+z^{39}\right)\tS^4(z)\\[2mm]
-\frac{1}{27}
  z^{16} \left(96
   z^{36}-3688 z^{33}+3928
   z^{30}-72 z^{21}+24
   z^{18}-1312 z^{15}-18
   z^3-45\right)
    \tS'(z)\\[2mm]
+\frac{1}{3}
   z^{35} \left(4
   z^{18}-4 z^{15}+1\right)
   (\tS')^2(z)
+ \left(80
   z^{52}-84 z^{49}+26
   z^{34}\right)\tS(z)
   \tS'(z)\\[2mm]
+ \left(8 z^{55}-8
   z^{52}+2
   z^{37}\right)\tS^2(z)
   \tS'(z)
+\frac{4}{9}
  z^{32} \left(4 z^{21}+14
   z^{18}-19
   z^{15}+z^3+6\right)
    \tS''(z)\\[2mm]
+\frac{1}{27} z^{33}
    \left(4 z^{18}-4
   z^{15}+1\right)\tS'''(z)
=0.
\end{multline}}%
We have to prove that, modulo $2^{4^{\al}}$, the series 
$\tS(z)$ can be expressed as a polynomial in
$\tilde\Phi_2(z)$ (as defined in \eqref{eq:Phimdef})
of degree at most $4^{\al+1}-1$ with coefficients that are Laurent 
polynomials in $z$.

We make the Ansatz 
\begin{equation} \label{eq:AnsatzH5}
\tS(z)=\sum _{i=0} ^{4^{\al+1}-1}a_i(z)\tilde\Phi_2^i(z)\quad \text {modulo
}2^{4^{\al}},
\end{equation}
where the $a_i(z)$'s
are (at this point) undetermined Laurent polynomials in $z$.

Next we gradually find approximations $a_{i,\be}(z)$ to $a_i(z)$ such that
\eqref{eq:diffH5A} holds modulo $2^\be$, for $\be=1,2,\dots,
4^{\al}$. To start the procedure, we consider the differential equation
\eqref{eq:diffH5A} modulo $2$, with
\begin{equation} \label{eq:AnsatzH51}
\tS(z)=\sum _{i=0} ^{4^{\al+1}-1}a_{i,1}(z)\tilde\Phi_2^i(z)\quad \text {modulo
}2.
\end{equation}
We substitute the Ansatz \eqref{eq:AnsatzH51} in \eqref{eq:diffH5A}, reduce
high powers of $\tilde\Phi_2(z)$ by using the relation
\eqref{eq:Phim2al} with $h=2$, 
and reduce the resulting expression modulo~2, thereby
taking advantage of the elementary fact that $\tilde\Phi_2'(z)=1$
modulo $2$.
The powers of $3$ that appear in the denominators of the
coefficients in the polynomials involved here are disposed of by
interpreting them appropriately as elements of $\mathbb Z/2^{4^\al}\mathbb
Z$. 
We finally see that the left-hand side of
\eqref{eq:diffH5A} becomes a polynomial in
$\tilde\Phi_2(z)$ of degree at most $4^{\al+1}-1$ 
with coefficients that are Laurent polynomials in $z$. 
Now we compare coefficients of powers $\tilde\Phi_2^k(z)$ for
$k=0,1,\dots,4^{\al+1}-1$. This yields a
system of $4^{\al+1}$ equations (modulo~$2$)
for the unknown Laurent polynomials $a_{i,1}(z)$,
$i=0,1,\dots,4^{\al+1}-1$. Since we have already done similar
computations several times before, we content ourselves with stating
the result: all Laurent polynomials $a_{i,1}(z)$ must be zero, except for
\begin{align} \notag
a_{0,1}(z)&=
z^{-9}+z^{-13}\sum _{k=1} ^{{\al}-1}z^{4^{k}}
+z^{-8}\sum _{k=1} ^{{\al}-1}z^{2\cdot 4^{k}}
,\\[2mm]
\notag
a_{4^{\al },1}(z)&=z^{-13},\\[2mm]
\label{eq:AnsatzH5a1}
a_{2\cdot 4^{\al },1}(z)&=z^{-8}.
\end{align}
 
\medskip
After we have completed the ``base  step," we now proceed with the
iterative steps described in Section~\ref{sec:method}. Our Ansatz
here (replacing the corresponding one in
\eqref{eq:Ansatz2}--\eqref{eq:Ansatz2b}) is
\begin{equation} \label{eq:AnsatzH5A2}
\tS(z)
=\sum _{i=0} ^{4^{\al+1}-1}a_{i,\be+1}(z)\tilde\Phi_2^i(z)\quad 
\text{modulo }2^{\be+1},
\end{equation}
with
\begin{equation} %\label{eq:AnsatzH5A2a}
a_{i,\be+1}(z):=a_{i,\be}(z)+2^{\be}b_{i,\be+1}(z),\quad 
i=0,1,\dots,4^{\al+1}-1,
\end{equation}
where the
coefficients $a_{i,\be}(z)$ are supposed to provide a solution
$$\tS_{\be}(z)=\sum _{i=0}
^{4^{\al+1}-1}a_{i,\be}(z)\tilde\Phi_2^i(z)$$ 
to
\eqref{eq:diffH5A} modulo~$2^\be$. This Ansatz, substituted in
\eqref{eq:diffH5A}, produces a congruence of the form
\begin{multline} \label{eq:iterH5}
T(\tS_{\be}(z))
+2^\be\sum _{i=0} ^{4^{\al+1}-1}
\big(b_{i,\be+1}(z)
+z^{16}b'_{i,\be+1}(z)
+(i+1)b_{i+1,\be+1}(z)\big)
\tilde\Phi_2^i(z)
+1=0
\\ 
\text {modulo }2^{\be+1},
\end{multline}
where $T(\tS_{\be}(z))$ represents terms that only depend on
$\tS_{\be}(z)$. Inductively, we have already computed $\tS_\be(z)$,
and we know that $T(\tS_{\be}(z))$ must be divisible by $2^\be$.
Comparison of powers of $\tilde\Phi_2(z)$ then yields a system of congruences
that is equivalent to a system of the form
\begin{equation*} %\label{eq:H5bi+1} 
b_{i,\be+1}(z)+z^{16}b'_{i,\be+1}(z)+\text {Pol}_i(z)=0\quad 
\text {modulo }2,\quad \quad 
i=0,1,\dots,4^{\al+1}-1,
\end{equation*}
where $\text {Pol}_i(z)$, $i=0,1,\dots,4^{\al+1}-1$, are certain
Laurent polynomials with integer coefficients. By
Lemma~\ref{lem:2x2diff}, these equations are solvable for any
polynomials $\text {Pol}_i(z)$. Thus,
we have proved that, for an arbitrary positive
integer $\al$, the modified algorithm that we have presented here
will produce a solution $\tS_{{
4^{\al }}}(z)$ to \eqref{eq:diffH5A} modulo $2^{4^{\al}}$ which is a
polynomial in $\tilde\Phi_2(z)$ of degree at most $4^{\al+1}-1$ 
with coefficients that are Laurent polynomials in $z$.
\end{proof}

Again, we have implemented the algorithm described in the above proof.
For $\al=1$, that is, for the modulus $16$, we obtain the following
result.

{\allowdisplaybreaks
\begin{theorem} \label{cor:H5}
Let $\Phi_2(z)=\sum _{n\ge0} ^{}z^{4^{n}/3},$ as before.
Then, for the generating function 
$S_{\mathfrak H(5)}(z)=
\sum_{n=0}^\infty s_{n+1}(\mathfrak H(5))z^n$ for the subgroup
numbers of $\mathfrak H(5),$ we have
\begin{multline} %\label{}
S_{\mathfrak H(5)}(z)=
8 z^{12}+4 z^{9}+8
   z^{7}+8 z^{5}+2
   z^{4}
+8
   z^2+4 z+14+\frac{7}{z^3}\\[2mm]
+\left(
8
   z^{20/3}
+8
   z^{5/3}
+\frac{8}{
   z^{1/3}}
+\frac{8}{z^{4/3}}
+\frac{12}{z^{10/3}}
\right) \Phi_2(z)\\[2mm]
+\left(
8
   z^{19/3}
+8
   z^{4/3}
+\frac{12}{z^{2/3}}
+\frac{8}{z^{5/3}}
+\frac{10}{z^{11/3}}
\right)
   \Phi_2^2(z)
+\left(
8
   z^{6}
+8
   z
+\frac{8}{z}
+\frac{12}{z^{4}}
\right)
   \Phi_2^3(z)\\[2mm]
+\left(
8 z^{32/3}+8
   z^{20/3}
+8
   z^{11/3}
+8
   z^{8/3}
+8
   z^{5/3}
+\frac{6}{z^{4/3}}
+\frac{8}{z^{7/3}}
+\frac{12}{z^{10/3}}
+\frac{7}{z^{13/3}}
\right)
   \Phi_2^4(z)\\[2mm]
\left(
+\frac{8}{z^{2/3}}
+\frac{8}{z
   ^{5/3}}
+\frac{8}{z^{8/3}}
+\frac{4}{z^{11/3
   }}
\right)
   \Phi_2^5(z)
+\left(
8
   z^{6}
+8
   z
+\frac{8}{z}
+\frac{12}{z^2}
+\frac{4}{z^{4}}
\right)
   \Phi_2^6(z)\\[2mm]
+\left(
\frac{
   8}{z^{4/3}}
+\frac{8}{z^{7/3}}
+\frac{12}{z
   ^{13/3}}
\right)
   \Phi_2^7(z)\\[2mm]
+\left(
8 z^{22/3}
+8
   z^{19/3}
+8
   z^{16/3}
+8
   z^{10/3}
+12
   z^{7/3}
+8
   z^{4/3}\right.\kern4cm\\[1mm]
\kern4cm\left.
+\frac{4}{z^{2/3}}
+\frac{8}{z^{5/3}}
+\frac{5}{z^{8/3}}
+\frac{2}{z^{11/3}}
\right)
   \Phi_2^8(z)\\[2mm]
+\left(
8
   z^{6}
+8
   z
+\frac{8}{z}
+\frac{4}{z^{4}}
+\frac{8}{z^2}
\right)
   \Phi_2^9(z)
+\left(
\frac{12}{z^{4/3}}
+\frac
   {2}{z^{13/3}}
\right)
   \Phi_2^{10}(z)
+\frac{8
  }{z^{8/3}} \Phi_2^{11}(z)\\[2mm]
+\left(
8 z^{6}
+8
   z^{4}
   +8
   z
+8
+\frac{8}{z}
+\frac{12}{z^2}
+\frac{12}{z^3}
+\frac{12}{z^{4}}
\right)
   \Phi_2^{12}(z)\\[2mm]
+\left(
\frac{8}{
   z^{4/3}}
+\frac{8}{z^{7/3}}
+\frac{12}{z^{13/3}}
\right)
   \Phi_2^{13}(z)
+\left(
\frac{4}{z^{8/3}}
+\frac{8}{z^
   {11/3}}
\right)
   \Phi_2^{14}(z)
\quad 
\text{\em modulo }16.
\end{multline}
\end{theorem}
}

\medskip
We conjecture that Theorems~\ref{thm:Unterg} and \ref{thm:H5} 
extend to any Hecke
group $\mathfrak H(q)$, where $q$ is a Fermat prime
(note that $PSL_2(\mathbb Z)\cong C_2*C_3=\mathfrak H(3)$). 

\begin{conjecture} \label{conj:Hq}
For a positive integer $h,$
let $\Phi_h(z)=\sum _{n\ge0} ^{}z^{2^{nh}/(2^h-1)}$.
Let $\al$ be a further positive integer, and let
$q=2^{2^{f}}+1$ be a Fermat prime.
Then the generating function $S_{\mathfrak H(q)}(z)$
{\em(}see the first paragraph of Section~{\em\ref{sec:PSL2Z}} 
for the definition{\em),} 
reduced modulo $2^{2^{\al 2^f}},$ 
can be expressed as a polynomial in $\Phi_{2^f}(z)$ of degree at most
$2^{(\al+1)2^f}-1$ with coefficients that are Laurent polynomials in
$z^{1/(q-2)}$ over the integers.
\end{conjecture}

Note that Theorems~\ref{thm:Unterg} and \ref{thm:H5} are the special
cases corresponding to $f=0$ and $f=1$, respectively.
In particular, we conjecture
that the obvious extension of the algorithm
described in the proofs of the two theorems would be successful
modulo any $2$-power. In more detail, given a Fermat prime $q$, 
the first step consists in deriving a recurrence relation for the
Hadamard product of the sequences $\big(h_{C_2}(n)\big)_{n\ge0}$ 
and $\big(h_{C_q}(n)/n!\big)_{n\ge0}$. By the procedure explained
in the paragraph containing \eqref{eq:Sdef} (where we now use
\eqref{eq:Dey} with $\Ga=\mathfrak H(q)$), this leads to a
Riccati-type differential equation for the generating function
$\sum _{n\ge0} ^{}s_{n+1}(\mathfrak H(q))\,z^n$ for the subgroup
numbers of $\mathfrak H(q)$. The open questions are whether it
will be possible to complete the base step, and whether
it will always be possible to carry out the subsequent iterative steps
in (the variation of) our method or, if necessary, its enhancement
outlined in Appendix~\ref{appD}.
Given the description of the parity pattern of the subgroup
numbers proved in \cite[Theorem~B]{MuHecke}, it is highly probable
that the first question has a positive answer. What the answer
to the second question is, remains entirely open.
(The reader should recall that, in Section~\ref{sec:Ga3}, we met 
a case where our method worked initially, but then stopped to work for
modulus $16$).

\appendix

\section{Expansions of powers of the $2$-power series $\Phi(z)$}
\label{appA}

Recall the notation
$$
H_{a_1,a_2,\dots,a_r}(z)=\sum_{n_1>n_2>\dots>n_r\ge0}
z^{a_12^{n_1}+a_22^{n_2}+\dots+a_r2^{n_r}}
$$
from Section~\ref{sec:extr}. In this appendix we list the expansions
of $\Phi^K(z)$ for $K=5,6,7,8$ in terms of the series
$H_{a_1,a_2,\dots,a_r}(z)$ where all $a_i$'s are odd, obtained
by using the algorithm described in Section~\ref{sec:extr}
(see \eqref{eq:Phipot} and the proof of Lemma~\ref{lem:Hbi}).
Namely, we have
{\allowdisplaybreaks
\begin{multline*}
\Phi^5(z)=
16 H_{5}(z)
-40   H_{3,1,1}(z)
-40   H_{1,3,1}(z)
-40   H_{1,1,3}(z)
+120   H_{1,1,1,1,1}(z)\\[2mm]
-80   H_{3,1}(z)
-80 H_{1,3}(z)
+240   H_{1,1,1,1}(z)
+(20   z 
-90  ) H_{3}(z)\\[2mm]
-(60 z   -270  ) H_{1,1,1}(z)
-(120 z -190)   H_{1,1}(z)\\[2mm]
+(25 z^2   
-125 z 
+75) H_{1}(z)
+50   z^2-75 z,
\end{multline*}
\begin{multline*}
\Phi^6(z)=
96   H_{5,1}(z)
+96   H_{1,5}(z)
+80 H_{3,3}(z)
-240   H_{3,1,1,1}(z)
-240   H_{1,3,1,1}(z)\\[2mm]
-240   H_{1,1,3,1}(z)
-240   H_{1,1,1,3}(z)
+720   H_{1,1,1,1,1,1}(z)
+240   H_{5}(z)\\[2mm]
-600   H_{3,1,1}(z)
-600   H_{1,3,1}(z)
-600   H_{1,1,3}(z)
+1800   H_{1,1,1,1,1}(z)\\[2mm]
+(120 z   
-840  ) H_{3,1}(z)
+(120 z   
-840  ) H_{1,3}(z)
-(360 z -2520  ) H_{1,1,1,1}(z)\\[2mm]
+(300 z 
-764) H_{3}(z)
-(900 z   -2340  ) H_{1,1,1}(z)
+(150 z^2 
-1200 z   
+1470 )  H_{1,1}(z)\\[2mm]
+(375 z^2   
-1020 z   
+525  ) H_{1}(z)
-61 z^3+495 z^2-525   z,
\end{multline*}
\begin{multline*}
\Phi^7(z)=
-272   H_{7}(z)
+672   H_{5,1,1}(z)
+672   H_{1,5,1}(z)
+672   H_{1,1,5}(z)\\[2mm]
+560   H_{3,3,1}(z)
+560   H_{3,1,3}(z)
+560   H_{1,3,3}(z)
+2016   H_{5,1}(z)
+2016   H_{1,5}(z)
+1680   H_{3,3}(z)\\[2mm]
-1680   H_{3,1,1,1,1}(z)
-1680   H_{1,3,1,1,1}(z)
-1680   H_{1,1,3,1,1}(z)
-1680   H_{1,1,1,3,1}(z)
-1680   H_{1,1,1,1,3}(z)\\[2mm]
+5040   H_{1,1,1,1,1,1,1}(z)
-5040   H_{3,1,1,1}(z)
-5040   H_{1,3,1,1}(z)
-5040   H_{1,1,3,1}(z)
-5040   H_{1,1,1,3}(z)\\[2mm]
+15120   H_{1,1,1,1,1,1}(z)
-(336 z   -3360) H_{5}(z)
+(840 z   
-8400  ) H_{3,1,1}(z)
+(840 z   
-8400  ) H_{1,3,1}(z)\\[2mm]
+(840 z   
-8400  ) H_{1,1,3}(z)
-(2520 z   -25200  ) H_{1,1,1,1,1}(z)
+(2520 z   
-9408  ) H_{3,1}(z)\\[2mm]
+(2520 z   
-9408  ) H_{1,3}(z)
-(7560 z   -28560  ) H_{1,1,1,1}(z)
-(350 z^2   -4060 z   +7434  ) H_{3}(z)\\[2mm]
+(1050 z^2   
-12180 z   
+23310  ) H_{1,1,1}(z)
+(3150 z^2 
-13020 z   
+13230  ) H_{1,1}(z)\\[2mm]
-(427   z^3 -5040 z^2   +9555 z   -4347) H_{1}(z)
-1281 z^3
+5208   z^2-4347 z,
\end{multline*}
\begin{multline*}
\Phi^8(z)=
-2176   H_{7,1}(z)
-2176   H_{1,7}(z)
-1792   H_{5,3}(z)
-1792   H_{3,5}(z)\\[2mm]
+5376   H_{5,1,1,1}(z)
+5376   H_{1,5,1,1}(z)
+5376   H_{1,1,5,1}(z)
+5376   H_{1,1,1,5}(z)
+4480   H_{3,3,1,1}(z)\\[2mm]
+4480   H_{3,1,3,1}(z)
+4480   H_{3,1,1,3}(z)
+4480   H_{1,3,3,1}(z)
+4480   H_{1,3,1,3}(z)
+4480   H_{1,1,3,3}(z)\\[2mm]
-13440   H_{3,1,1,1,1,1}(z)
-13440   H_{1,3,1,1,1,1}(z)
-13440   H_{1,1,3,1,1,1}(z)
-13440   H_{1,1,1,3,1,1}(z)\\[2mm]
-13440   H_{1,1,1,1,3,1}(z)
-13440   H_{1,1,1,1,1,3}(z)
+40320   H_{1,1,1,1,1,1,1,1}(z)
-7616 H_{7}(z)\\[2mm]
+18816   H_{5,1,1}(z)
+18816   H_{1,5,1}(z)
+18816   H_{1,1,5}(z)\\[2mm]
+15680   H_{3,3,1}(z)
+15680   H_{3,1,3}(z)
+15680   H_{1,3,3}(z)
-47040   H_{3,1,1,1,1}(z)\\[2mm]
-47040   H_{1,3,1,1,1}(z)
-47040   H_{1,1,3,1,1}(z)
-47040   H_{1,1,1,3,1}(z)
-47040   H_{1,1,1,1,3}(z)\\[2mm]
+141120   H_{1,1,1,1,1,1,1}(z)
-(2688 z   -36288  ) H_{5,1}(z)
-(2688 z   -36288  ) H_{1,5}(z)\\[2mm]
-(2240 z   -30240  ) H_{3,3}(z)
+(6720 z   
-90720  ) H_{3,1,1,1}(z)
+(6720 z   
-90720  ) H_{1,3,1,1}(z)\\[2mm]
+(6720 z   
-90720  ) H_{1,1,3,1}(z)
+(6720 z   
-90720  ) H_{1,1,1,3}(z)
-(20160 z   -272160  ) H_{1,1,1,1,1,1}(z)\\[2mm]
-(9408 z   -47264  ) H_{5}(z)
+(23520 z   
-119504  ) H_{3,1,1}(z)
+(23520 z   
-119504  ) H_{1,3,1}(z)\\[2mm]
+(23520 z   
-119504  ) H_{1,1,3}(z)
-(70560 z   -361200  ) H_{1,1,1,1,1}(z)\\[2mm]
-(2800 z^2   -44240 z   +115304  ) H_{3,1}(z)
-(2800 z^2   -44240 z   +115304  ) H_{1,3}(z)\\[2mm]
+(8400 z^2   
-132720 z   
+355320  ) H_{1,1,1,1}(z)
-(9800 z^2   -55832 z   +80892  ) H_{3}(z)\\[2mm]
+(29400 z^2   
-168840 z   
+260820  ) H_{1,1,1}(z)
-(3416 z^3 -55020 z^2   +154980 z   -135982  ) H_{1,1}(z)\\[2mm]
-(11956 z^3 -68474 z^2   +101206 z   -41245  ) H_{1}(z)
+1385   z^4 -22358 z^3 +59961   z^2 -41245 z.
\end{multline*}
}

\section{The coefficients in the differential equation \eqref{eq:Riccati16}}
\label{appB}

Here we provide explicit expressions for 
the coefficients in the Riccati-type differential equation 
\eqref{eq:Riccati16}, when reduced modulo~16:
{\Small\allowdisplaybreaks
\begin{align*}
p_0(z)&=8 z^{47}+8 z^{46}+12 z^{45}+4
   z^{43}+12 z^{41}+12
   z^{40}+4 z^{39}+12 z^{38}+8
   z^{37}+4 z^{36}+4 z^{34}+2
   z^{33}\\
&\kern.5cm
+11 z^{31}+6
   z^{30}+14 z^{29}+14
   z^{28}+13 z^{27}+6 z^{26}+9
   z^{25}+11 z^{24}+4 z^{23}+7
   z^{22}+9 z^{21}+z^{20}\\
&\kern.5cm
+7
   z^{19}+15 z^{18}+14
   z^{17}+12 z^{16}+11
   z^{15}+z^{14}+10 z^{13}+5
   z^{12}+2 z^{11}+8 z^{10}+9
   z^9+15 z^8\\
&\kern.5cm
+5 z^7+13 z^6+4
   z^4+14 z^3+z^2+11
   z+15,\\
p_1(z)&=8
   z^{48}+12 z^{46}+12
   z^{45}+4 z^{43}+12 z^{42}+4
   z^{41}+4 z^{40}+4 z^{39}+8
   z^{38}+4 z^{36}+13 z^{34}+6
   z^{33}\\
&\kern.5cm
+10 z^{31}+z^{30}+4
   z^{29}+11 z^{28}+8
   z^{27}+13 z^{25}+z^{24}+8
   z^{23}+z^{22}+z^{21}+4
   z^{20}+14 z^{19}\\
&\kern.5cm
+9 z^{18}+6
   z^{17}+14 z^{16}+8
   z^{15}+13 z^{14}+6 z^{13}+2
   z^{12}+9 z^{10}+11 z^9+6
   z^7+8 z^6+6 z^5\\
&\kern.5cm
+5 z^4+3
   z^2+4 z+1,\\
p_2(z)&=12 z^{51}+2 z^{49}+12
   z^{48}+4 z^{47}+4 z^{46}+12
   z^{44}+10 z^{43}+2 z^{42}+4
   z^{41}+8 z^{40}+14 z^{39}\\
&\kern.5cm
+5
   z^{37}+2 z^{36}+6 z^{35}+10
   z^{34}+8 z^{33}+10 z^{32}+5
   z^{31}+7 z^{30}+12 z^{29}+7
   z^{28}+4 z^{27}\\
&\kern.5cm
+12 z^{26}+9
   z^{25}+12 z^{24}+2
   z^{23}+14 z^{22}+10
   z^{21}+12 z^{20}+10
   z^{19}+3 z^{18}+6 z^{16}+9
   z^{15}\\
&\kern.5cm
+12 z^{14}+15
   z^{13}+14 z^{12}+10
   z^{11}+2 z^{10}+12 z^9+14
   z^8+13 z^7+7 z^6+8
   z^5+z^3+2
   z^2,\\
p_3(z)&=6 z^{52}+8 z^{50}+12
   z^{48}+8 z^{47}+14 z^{46}+4
   z^{45}+8 z^{44}+6 z^{43}+8
   z^{42}+4 z^{41}+5 z^{40}+6
   z^{39}\\
&\kern.5cm
+12 z^{38}+2 z^{37}+2
   z^{36}+7 z^{34}+6 z^{33}+7
   z^{31}+14 z^{30}+8 z^{29}+5
   z^{28}+14 z^{27}+8
   z^{26}\\
&\kern.5cm
+14 z^{25}+14
   z^{24}+2 z^{22}+8 z^{21}+12
   z^{20}+14 z^{19}+6
   z^{18}+12 z^{17}+9 z^{16}+4
   z^{15}+2 z^{13}\\
&\kern.5cm
+2 z^{12}+15
   z^{10}+14
   z^9,\\
p_4(z)&=4 z^{53}+15 z^{51}+4
   z^{47}+2 z^{46}+z^{45}+9
   z^{44}+15 z^{42}+4
   z^{41}+10 z^{40}+4 z^{39}+6
   z^{38}+4 z^{37}\\
&\kern.5cm
+2 z^{36}+6
   z^{35}+10 z^{34}+11
   z^{33}+8 z^{32}+8 z^{31}+10
   z^{29}+12 z^{26}+12
   z^{25}+2 z^{23}+4
   z^{22},\\
p_5(z)&=13 z^{54}+15 z^{48}+6
   z^{47}+12 z^{46}+3
   z^{45}+10 z^{44}+4
   z^{43}+12 z^{42}+10
   z^{41}+2 z^{39}+2 z^{38}\\
&\kern.5cm
+15
   z^{36}+14
   z^{35},\\
p_6(z)&=12 z^{51}+2 z^{49}+12
   z^{48}+4 z^{47}+4 z^{46}+12
   z^{44}+10 z^{43}+2 z^{42}+4
   z^{41}+8 z^{40}+14 z^{39}\\
&\kern.5cm
+5
   z^{37}+2 z^{36}+6 z^{35}+10
   z^{34}+8 z^{33}+10 z^{32}+5
   z^{31}+7 z^{30}+12 z^{29}+7
   z^{28}+4 z^{27}\\
&\kern.5cm
+12 z^{26}+9
   z^{25}+12 z^{24}+2
   z^{23}+14 z^{22}+10
   z^{21}+12 z^{20}+10
   z^{19}+3 z^{18}+6 z^{16}\\
&\kern.5cm
+9
   z^{15}+12 z^{14}+15
   z^{13}+14 z^{12}+10
   z^{11}+2 z^{10}+12 z^9+14
   z^8+13 z^7\\
&\kern.5cm
+7 z^6+8
   z^5+z^3+2
   z^2,\\
p_7(z)&=12 z^{53}+13
   z^{51}+12 z^{47}+6 z^{46}+3
   z^{45}+11 z^{44}+13
   z^{42}+12 z^{41}+14
   z^{40}+12 z^{39}+2
   z^{38}\\
&\kern.5cm
+12 z^{37}+6 z^{36}+2
   z^{35}+14 z^{34}+z^{33}+8
   z^{32}+8 z^{31}+14 z^{29}+4
   z^{26}+4 z^{25}+6 z^{23}+12
   z^{22},\\
p_8(z)&=2 z^{52}+8
   z^{50}+4 z^{48}+8 z^{47}+10
   z^{46}+12 z^{45}+8 z^{44}+2
   z^{43}+8 z^{42}+12
   z^{41}+15 z^{40}+2 z^{39}\\
&\kern.5cm
+4
   z^{38}+6 z^{37}+6 z^{36}+5
   z^{34}+2 z^{33}+5 z^{31}+10
   z^{30}+8 z^{29}+15
   z^{28}+10 z^{27}+8
   z^{26}\\
&\kern.5cm
+10 z^{25}+10
   z^{24}+6 z^{22}+8 z^{21}+4
   z^{20}+10 z^{19}+2 z^{18}+4
   z^{17}+11 z^{16}+12
   z^{15}\\
&\kern.5cm
+6 z^{13}+6 z^{12}+13
   z^{10}+10
   z^9,\\
p_{9}(z)&=8 z^{53}+10
   z^{51}+8 z^{47}+12 z^{46}+6
   z^{45}+6 z^{44}+10 z^{42}+8
   z^{41}+12 z^{40}+8 z^{39}+4
   z^{38}+8 z^{37}\\
&\kern.5cm
+12 z^{36}+4
   z^{35}+12 z^{34}+2
   z^{33}+12 z^{29}+8 z^{26}+8
   z^{25}+12 z^{23}+8
   z^{22},\\
p_{10}(z)&=2 z^{54}+6
   z^{48}+12 z^{47}+8
   z^{46}+14 z^{45}+4 z^{44}+8
   z^{43}+8 z^{42}+4 z^{41}+4
   z^{39}+4 z^{38}\\
&\kern.5cm
+6 z^{36}+12
   z^{35},\\
p_{11}(z)&=3
   z^{54}+z^{48}+10 z^{47}+4
   z^{46}+13 z^{45}+6
   z^{44}+12 z^{43}+4 z^{42}+6
   z^{41}+14 z^{39}\\
&\kern.5cm
+14
   z^{38}+z^{36}+2
   z^{35},\\
p_{12}(z)&=6 z^{52}+8 z^{50}+12
   z^{48}+8 z^{47}+14 z^{46}+4
   z^{45}+8 z^{44}+6 z^{43}+8
   z^{42}+4 z^{41}+5 z^{40}+6
   z^{39}+12 z^{38}\\
&\kern.5cm
+2 z^{37}+2
   z^{36}+7 z^{34}+6 z^{33}+7
   z^{31}+14 z^{30}+8 z^{29}+5
   z^{28}+14 z^{27}+8
   z^{26}+14 z^{25}+14
   z^{24}\\
&\kern.5cm
+2 z^{22}+8 z^{21}+12
   z^{20}+14 z^{19}+6
   z^{18}+12 z^{17}+9 z^{16}+4
   z^{15}+2 z^{13}+2 z^{12}+15
   z^{10}+14
   z^9,\\
p_{13}(z)&=12 z^{51}+8
   z^{46}+4 z^{45}+4 z^{44}+12
   z^{42}+8 z^{40}+8 z^{38}+8
   z^{36}+8 z^{35}+8 z^{34}+12
   z^{33}\\
&\kern.5cm
+8 z^{29}+8
   z^{23},\\
p_{14}(z)&=2 z^{54}+6
   z^{48}+12 z^{47}+8
   z^{46}+14 z^{45}+4 z^{44}+8
   z^{43}+8 z^{42}+4 z^{41}+4
   z^{39}\\
&\kern.5cm
+4 z^{38}+6 z^{36}+12
   z^{35},\\
p_{15}(z)&=2 z^{54}+6
   z^{48}+12 z^{47}+8
   z^{46}+14 z^{45}+4 z^{44}+8
   z^{43}+8 z^{42}+4 z^{41}+4
   z^{39}\\
&\kern.5cm
+4 z^{38}+6 z^{36}+12
   z^{35},\\
p_{16}(z)&=4 z^{53}+15 z^{51}+4
   z^{47}+2 z^{46}+z^{45}+9
   z^{44}+15 z^{42}+4
   z^{41}+10 z^{40}+4 z^{39}+6
   z^{38}+4 z^{37}+2 z^{36}\\
&\kern.5cm
+6
   z^{35}+10 z^{34}+11
   z^{33}+8 z^{32}+8 z^{31}+10
   z^{29}+12 z^{26}+12
   z^{25}+2 z^{23}+4
   z^{22},\\
p_{17}(z)&=z^{54}+11
   z^{48}+14 z^{47}+12
   z^{46}+15 z^{45}+2 z^{44}+4
   z^{43}+12 z^{42}+2
   z^{41}+10 z^{39}+10
   z^{38}\\
&\kern.5cm
+11 z^{36}+6
   z^{35},\\
p_{18}(z)&=13 z^{54}+15 z^{48}+6
   z^{47}+12 z^{46}+3
   z^{45}+10 z^{44}+4
   z^{43}+12 z^{42}+10
   z^{41}+2 z^{39}+2 z^{38}\\
&\kern.5cm
+15
   z^{36}+14 z^{35}.
\end{align*}}%

\section{The coefficients in the differential equation \eqref{eq:Ga3Riccati16}}
\label{appC}

Here we provide explicit expressions for 
the coefficients in the Riccati-type differential equation 
\eqref{eq:Ga3Riccati16}, when reduced modulo~16:
{\Small\allowdisplaybreaks
\begin{align*}
q_0(z)&=
8 z^{50}+8 z^{49}+4 z^{47}+12 z^{46}+8 z^{44}+10 z^{43}+2 z^{42}+2 z^{41}+12 
z^{40}+12 z^{38}+7 z^{37}+11 z^{36}\\
&\kern.5cm
+10 z^{35}+11 z^{34}+13 z^{33}+13 
z^{32}+14 z^{31}+15 z^{30}+10 z^{29}+14 z^{28}+8 z^{27}+10 z^{26}+6 
z^{25}\\
&\kern.5cm
+z^{23}+4 z^{22}+8 z^{20}+14 z^{19}+12 
z^{18}+8 z^{17}+9 z^{16}+4 z^{15}+6 z^{14}+8 z^{13}+z^{12}+14 z^{11}\\
&\kern.5cm
+9 
z^{10}+2 z^9+14 z^8+10 z^7+4 z^6+11 z^5+4 z^4+4 z^3+z^2+8 z+15,\\
q_1(z)&=
8 z^{51}+8 z^{50}+8 
z^{49}+12 z^{48}+4 z^{47}+12 z^{46}+4 z^{45}+14 z^{43}+10 z^{42}+5 z^{40}+4 
z^{39}+4 z^{38}\\
&\kern.5cm
+9 z^{37}+4 z^{36}+14 z^{35}+5 z^{34}+14 z^{33}+13 z^{31}+5 
z^{30}+4 z^{29}+14 z^{28}+6 z^{27}+5 z^{26}\\
&\kern.5cm
+10 z^{25}+12 z^{23}+9 z^{22}+12 
z^{21}+13 z^{20}+6 z^{19}+z^{18}+7 z^{17}+15 z^{15}+11 z^{14}+7 z^{13}\\
&\kern.5cm
+5 
z^{12}+z^{11}+6 z^{10}+10 z^9+2 z^8+4 z^7+8 z^6+2 z^5+11 z^4+4 z^3+8 z^2+7 
z+1,\\
q_{2}(z)&=
8 z^{57}+8 
z^{56}+8 z^{55}+8 z^{54}+8 z^{53}+12 z^{51}+10 z^{50}+12 z^{49}+8 z^{48}+10 
z^{47}+15 z^{46}+4 z^{45}\\
&\kern.5cm
+10 z^{43}+z^{42}+11 z^{41}+z^{40}+2 
z^{39}+z^{38}+3 z^{37}+6 z^{36}+7 z^{35}+8 z^{34}+z^{33}+6 z^{31}+9 
z^{29}\\
&\kern.5cm
+14 z^{28}+11 z^{27}+5 z^{26}+2 z^{24}+6 z^{23}+4 z^{22}+6 z^{21}+13 
z^{20}+z^{19}+10 z^{18}+z^{17}+10 z^{15}\\
&\kern.5cm
+12 z^{14}+z^{13}+9 z^{12}+9 
z^{11}+7 z^{10}+4 z^9+12 z^8+3 z^7+2 z^6+9 z^5+3 z^4+9 z^3,\\
q_{3}(z)&=
14 z^{56}+2 
z^{55}+14 z^{54}+11 z^{52}+6 z^{51}+10 z^{50}+11 z^{49}+2 z^{48}+15 z^{47}+6 
z^{46}+10 z^{45}+9 z^{44}\\
&\kern.5cm
+4 z^{43}+5 z^{42}+11 z^{41}+5 z^{40}+10 z^{39}+4 
z^{38}+12 z^{37}+2 z^{36}+5 z^{35}+z^{34}+6 z^{33}+5 z^{32}\\
&\kern.5cm
+13 z^{31}+12 
z^{30}+3 z^{29}+10 z^{28}+13 z^{27}+14 z^{26}+4 z^{25}+12 z^{24}+8 z^{23}+9 
z^{22}+6 z^{21}\\
&\kern.5cm
+14 z^{20}+7 z^{19}+5 z^{18}+4 z^{17}+4 z^{16}+8 z^{14}+15 
z^{13}+13 z^{11}+9 z^9+7 z^8+z^6+4 z^5,\\
q_{4}(z)&=
8 z^{66}+15 z^{58}+14 z^{57}+12 
z^{56}+9 z^{54}+7 z^{53}+9 z^{52}+10 z^{51}+12 z^{50}+z^{49}+10 z^{48}+14 
z^{46}\\
&\kern.5cm
+6 z^{45}+6 z^{44}+6 z^{43}+14 z^{42}+11 z^{41}+13 z^{40}+10 
z^{39}+z^{38}+13 z^{37}+7 z^{36}+2 z^{34}\\
&\kern.5cm
+12 z^{33}+4 z^{32}+13 z^{31}+13 
z^{30}+4 z^{29}+11 z^{28}+7 z^{27}+8 z^{26}+2 z^{25}+6 z^{24}+13 z^{23}\\
&\kern.5cm
+11 
z^{21}+2 z^{20}+9 z^{19}+6 z^{18}+2 z^{17}+15 z^{16}+5 z^{15}+3 z^{14}+14 
z^{13}+2 z^{12}+10 z^{11}\\
&\kern.5cm
+9 z^{10}+6 z^9+12 z^8,\\
q_{5}(z)&=
5 z^{59}+13 
z^{58}+12 z^{56}+z^{55}+8 z^{54}+5 z^{53}+12 z^{52}+15 z^{51}+8 z^{50}+14 
z^{49}+2 z^{48}+4 z^{47}\\
&\kern.5cm
+4 z^{46}+10 z^{45}+12 z^{44}+8 z^{43}+4 z^{41}+2 
z^{40}+8 z^{38}+12 z^{37}+2 z^{36}+12 z^{35}+8 z^{34},\\
q_{6}(z)&=
12 z^{49}+4 z^{48}+12 z^{46}+12 z^{45}+12 z^{44}+4 z^{43}+12 
z^{42}+12 z^{41}+12 z^{39}+4 z^{38}+4 z^{37}\\
&\kern.5cm
+4 z^{36}+4 z^{32}+4 z^{31}+12 
z^{27}+12 z^{24}+12 z^{22}+12 z^{21}+12 z^{20}+12 z^{19}+12 z^{18}+4 
z^{15}\\
&\kern.5cm
+4 z^{13}+12 z^{12}+12 z^{11}+12 z^{10}+4 z^9+4 z^8+4 z^7+12 z^6+4 
z^5,\\
q_{7}(z)&=
10 
z^{58}+6 z^{54}+10 z^{53}+6 z^{52}+6 z^{49}+2 z^{41}+14 z^{40}+6 z^{38}+14 
z^{37}+10 z^{36}+14 z^{31}\\
&\kern.5cm
+14 z^{30}+2 z^{28}+10 z^{27}+14 z^{23}+2 z^{21}+6 
z^{19}+10 z^{16}+14 z^{15}+2 z^{14}+6 z^{10},\\
q_{8}(z)&=
10 z^{56}+6 
z^{55}+10 z^{54}+9 z^{52}+2 z^{51}+14 z^{50}+9 z^{49}+6 z^{48}+5 z^{47}+2 
z^{46}+14 z^{45}+3 z^{44}\\
&\kern.5cm
+12 z^{43}+7 z^{42}+9 z^{41}+7 z^{40}+14 z^{39}+12 
z^{38}+4 z^{37}+6 z^{36}+7 z^{35}+11 z^{34}+2 z^{33}\\
&\kern.5cm
+7 z^{32}+15 z^{31}+4 
z^{30}+z^{29}+14 z^{28}+15 z^{27}+10 z^{26}+12 z^{25}+4 z^{24}+3 z^{22}+2 
z^{21}\\
&\kern.5cm
+10 z^{20}+13 z^{19}+7 z^{18}+12 z^{17}+12 z^{16}+5 z^{13}+15 z^{11}+3 
z^9+13 z^8+11 z^6+12 z^5,\\
q_{9}(z)&=
7 
z^{59}+15 z^{58}+11 z^{55}+7 z^{53}+5 z^{51}+10 z^{49}+6 z^{48}+14 z^{45}+6 
z^{40}+6 z^{36},\\
q_{10}(z)&=
z^{58}+7 z^{54}+9 z^{53}+7 z^{52}+15 z^{49}+5 z^{41}+3 
z^{40}+15 z^{38}+3 z^{37}+9 z^{36}+3 z^{31}+3 z^{30}\\
&\kern.5cm
+5 z^{28}+9 z^{27}+3 
z^{23}+5 z^{21}+7 z^{19}+z^{16}+11 z^{15}+13 z^{14}+7 
z^{10},\\
q_{11}(z)&=
12 z^{55}+4 z^{51}+12 z^{50}+4 z^{47}+12 z^{46}+4 z^{45}+12 z^{43}+4 
z^{42}+4 z^{41}+4 z^{38}+12 z^{34}\\
&\kern.5cm
+4 z^{33}+4 z^{32}+12 z^{29}+4 z^{26}+12 
z^{25}+12 z^{22}+12 z^{16}+12 z^{15}+4 z^{14}+12 z^{13}\\
&\kern.5cm
+12 z^9+4 z^8+4 
z^6,\\
q_{12}(z)&=
15 z^{58}+14 z^{57}+12 z^{56}+9 z^{54}+7 z^{53}+9 z^{52}+10 z^{51}+12 
z^{50}+z^{49}+10 z^{48}+14 z^{46}+6 z^{45}\\
&\kern.5cm
+6 z^{44}+6 z^{43}+14 z^{42}+11 
z^{41}+13 z^{40}+10 z^{39}+z^{38}+13 z^{37}+7 z^{36}+2 z^{34}+12 z^{33}+4 
z^{32}\\
&\kern.5cm
+13 z^{31}+13 z^{30}+4 z^{29}+11 z^{28}+7 z^{27}+2 z^{25}+6 z^{24}+13 
z^{23}+11 z^{21}+2 z^{20}+9 z^{19}+6 z^{18}\\
&\kern.5cm
+2 z^{17}+15 z^{16}+5 z^{15}+3 
z^{14}+14 z^{13}+2 z^{12}+10 z^{11}+9 z^{10}+6 z^9+12 z^8,\\
q_{13}(z)&=
10 z^{59}+10 
z^{58}+2 z^{55}+10 z^{53}+14 z^{51}+12 z^{49}+4 z^{48}+4 z^{45}+4 z^{40}+4 
z^{36},\\
q_{14}(z)&=
8 z^{57}+8 
z^{56}+8 z^{55}+12 z^{53}+8 z^{51}+8 z^{50}+8 z^{49}+8 z^{48}+10 z^{47}+14 
z^{46}+10 z^{45}+8 z^{44}\\
&\kern.5cm
+9 z^{43}+12 z^{42}+2 z^{41}+6 z^{40}+11 z^{39}+3 
z^{38}+12 z^{37}+8 z^{36}+7 z^{35}+2 z^{34}+z^{33}+z^{32}\\
&\kern.5cm
+7 z^{31}+12 
z^{30}+12 z^{29}+z^{28}+z^{26}+15 z^{25}+9 z^{24}+14 z^{23}+z^{22}+3 
z^{21}+3 z^{20}+4 z^{19}\\
&\kern.5cm
+7 z^{17}+7 z^{16}+6 z^{15}+3 z^{14}+z^{13}+3 
z^{12}+2 z^{11}+14 z^9+2 z^8+6 z^7+2 z^5+14 z^4\\
&\kern.5cm
+6 z^3+14 
z^2,\\
q_{15}(z)&=
12 z^{55}+12 z^{54}+8 z^{50}+2 
z^{49}+14 z^{48}+8 z^{47}+10 z^{46}+10 z^{45}+10 z^{44}+14 z^{43}+10 
z^{42}+10 z^{41}\\
&\kern.5cm
+2 z^{39}+14 z^{38}+6 z^{37}+14 z^{36}+12 z^{34}+14 
z^{32}+14 z^{31}+4 z^{30}+8 z^{28}+10 z^{27}+12 z^{25}\\
&\kern.5cm
+10 z^{24}+12 z^{23}+2 
z^{22}+2 z^{21}+2 z^{20}+2 z^{19}+10 z^{18}+8 z^{17}+4 z^{16}+6 z^{15}+8 
z^{14}+14 z^{13}\\
&\kern.5cm
+2 z^{12}+2 z^{11}+2 z^{10}+6 z^9+14 z^8+14 z^7+2 z^6+14 
z^5+12 z^4,\\
q_{16}(z)&=
4 z^{65}+10 z^{57}+6 z^{56}+9 
z^{55}+8 z^{53}+6 z^{52}+11 z^{51}+13 z^{50}+2 z^{49}+12 z^{48}+3 z^{47}+13 
z^{46}\\
&\kern.5cm
+3 z^{45}+10 z^{44}+13 z^{43}+3 z^{42}+3 z^{41}+2 z^{40}+4 z^{39}+7 
z^{38}+4 z^{37}+6 z^{36}+6 z^{35}+5 z^{34}\\
&\kern.5cm
+15 z^{33}+3 z^{32}+10 z^{31}+14 
z^{30}+z^{29}+2 z^{28}+14 z^{27}+11 z^{26}+9 z^{25}+2 z^{24}+9 z^{22}+8 
z^{21}\\
&\kern.5cm
+2 z^{20}+8 z^{19}+10 z^{18}+12 z^{17}+9 z^{16}+5 z^{15}+11 z^{14}+5 
z^{13}+8 z^{12}+10 z^{10}+z^9+7 z^8\\
&\kern.5cm
+6 z^7+7 z^6,\\
q_{17}(z)&=
8 z^{58}+12 z^{55}+8 z^{54}+8 
z^{53}+8 z^{52}+4 z^{51}+12 z^{50}+8 z^{49}+8 z^{48}+4 z^{47}+12 z^{45}+8 
z^{43}\\
&\kern.5cm
+8 z^{42}+12 z^{41}+8 z^{40}+12 z^{39}+8 z^{38}+4 z^{36}+4 z^{35}+8 
z^{34}+4 z^{33}+12 z^{32}+8 z^{31}+4 z^{29}\\
&\kern.5cm
+8 z^{28}+8 z^{26}+8 z^{22}+12 
z^{21}+8 z^{18}+12 z^{17}+8 z^{16},\\
q_{18}(z)&=
6 z^{60}+11 z^{59}+4 z^{57}+10 
z^{56}+11 z^{55}+2 z^{54}+4 z^{53},\\
q_{19}(z)&=
4 z^{56}+12 z^{55}+4 z^{54}+2 z^{52}+4 z^{51}+12 z^{50}+2 
z^{49}+12 z^{48}+10 z^{47}+4 z^{46}+12 z^{45}+6 z^{44}\\
&\kern.5cm
+14 z^{42}+2 z^{41}+14 
z^{40}+12 z^{39}+12 z^{36}+14 z^{35}+6 z^{34}+4 z^{33}+14 z^{32}+14 z^{31}+2 
z^{29}\\
&\kern.5cm
+12 z^{28}+14 z^{27}+4 z^{26}+6 z^{22}+4 z^{21}+4 z^{20}+10 z^{19}+14 
z^{18}+10 z^{13}+14 z^{11}+6 z^9\\
&\kern.5cm
+10 z^8+6 z^6,\\
q_{20}(z)&=
12 z^{65}+14 z^{57}+2 z^{56}+3 z^{55}+2 z^{52}+9 z^{51}+15 z^{50}+6 
z^{49}+4 z^{48}+z^{47}+15 z^{46}+z^{45}\\
&\kern.5cm
+14 z^{44}+15 z^{43}+z^{42}+z^{41}+6 
z^{40}+12 z^{39}+13 z^{38}+12 z^{37}+2 z^{36}+2 z^{35}+7 z^{34}+5 
z^{33}\\
&\kern.5cm
+z^{32}+14 z^{31}+10 z^{30}+11 z^{29}+6 z^{28}+10 z^{27}+9 z^{26}+3 
z^{25}+6 z^{24}+3 z^{22}+6 z^{20}\\
&\kern.5cm
+14 z^{18}+4 z^{17}+3 z^{16}+7 z^{15}+9 
z^{14}+7 z^{13}+14 z^{10}+11 z^9+13 z^8+2 z^7+13 z^6,\\
q_{21}(z)&=
2 z^{60}+9 z^{59}+14 z^{56}+9 z^{55}+6 z^{54},\\
q_{22}(z)&=
3 z^{58}+5 z^{54}+11 z^{53}+5 
z^{52}+13 z^{49}+15 z^{41}+9 z^{40}+13 z^{38}+9 z^{37}+11 z^{36}+9 z^{31}\\
&\kern.5cm
+9 
z^{30}+15 z^{28}+11 z^{27}+9 z^{23}+15 z^{21}+5 z^{19}+3 z^{16}+z^{15}+7 
z^{14}+5 z^{10},\\
q_{23}(z)&=
14 z^{58}+12 z^{57}+2 z^{54}+14 z^{53}+2 z^{52}+4 z^{51}+2 
z^{49}+4 z^{48}+12 z^{46}+12 z^{45}+12 z^{44}+12 z^{43}\\
&\kern.5cm
+12 z^{42}+6 
z^{41}+10 z^{40}+4 z^{39}+2 z^{38}+10 z^{37}+14 z^{36}+4 z^{34}+10 z^{31}+10 
z^{30}+6 z^{28}\\
&\kern.5cm
+14 z^{27}+4 z^{25}+12 z^{24}+10 z^{23}+6 z^{21}+4 z^{20}+2 
z^{19}+12 z^{18}+4 z^{17}+14 z^{16}+10 z^{15}\\
&\kern.5cm
+6 z^{14}+12 z^{13}+4 z^{12}+4 
z^{11}+2 z^{10}+12 z^9,\\
q_{24}(z)&=
4 z^{59}+4 z^{58}+4 z^{55}+4 z^{53}+12 
z^{51},\\
q_{25}(z)&=
12 z^{55}+4 z^{51}+12 
z^{50}+4 z^{47}+12 z^{45}+12 z^{41}+12 z^{39}+4 z^{36}+4 z^{35}+4 z^{33}+12 
z^{32}+4 z^{29}\\
&\kern.5cm
+12 z^{21}+12 z^{17},\\
q_{26}(z)&=
12 z^{60}+6 z^{59}+4 z^{56}+6 z^{55}+4 z^{54},\\
q_{27}(z)&=
8 z^{57}+8 z^{56}+8 z^{55}+8 z^{54}+8 
z^{53}+4 z^{51}+14 z^{50}+4 z^{49}+8 z^{48}+14 z^{47}+5 z^{46}+12 z^{45}\\
&\kern.5cm
+14 
z^{43}+11 z^{42}+9 z^{41}+11 z^{40}+6 z^{39}+11 z^{38}+z^{37}+2 z^{36}+13 
z^{35}+8 z^{34}+11 z^{33}+2 z^{31}\\
&\kern.5cm
+3 z^{29}+10 z^{28}+9 z^{27}+7 z^{26}+6 
z^{24}+2 z^{23}+12 z^{22}+2 z^{21}+15 z^{20}+11 z^{19}+14 z^{18}\\
&\kern.5cm
+11 
z^{17}+14 z^{15}+4 z^{14}+11 z^{13}+3 z^{12}+3 z^{11}+13 z^{10}+12 z^9+4 
z^8+z^7+6 z^6\\
&\kern.5cm
+3 z^5+z^4+3 z^3,\\
q_{28}(z)&=
4 
z^{56}+12 z^{55}+4 z^{54}+2 z^{52}+4 z^{51}+12 z^{50}+2 z^{49}+12 z^{48}+10 
z^{47}+4 z^{46}+12 z^{45}+6 z^{44}\\
&\kern.5cm
+8 z^{43}+14 z^{42}+2 z^{41}+14 z^{40}+12 
z^{39}+8 z^{38}+8 z^{37}+12 z^{36}+14 z^{35}+6 z^{34}+4 z^{33}\\
&\kern.5cm
+14 z^{32}+14 
z^{31}+8 z^{30}+2 z^{29}+12 z^{28}+14 z^{27}+4 z^{26}+8 z^{25}+8 z^{24}+6 
z^{22}+4 z^{21}\\
&\kern.5cm
+4 z^{20}+10 z^{19}+14 z^{18}+8 z^{17}+8 z^{16}+10 z^{13}+14 
z^{11}+6 z^9+10 z^8+6 z^6+8 z^5,\\
q_{29}(z)&=
8 z^{66}+15 z^{58}+14 z^{57}+12 z^{56}+9 z^{54}+7 z^{53}+9 
z^{52}+10 z^{51}+12 z^{50}+z^{49}+10 z^{48}+14 z^{46}\\
&\kern.5cm
+6 z^{45}+6 z^{44}+6 
z^{43}+14 z^{42}+11 z^{41}+13 z^{40}+10 z^{39}+z^{38}+13 z^{37}+7 z^{36}+2 
z^{34}\\
&\kern.5cm
+12 z^{33}+4 z^{32}+13 z^{31}+13 z^{30}+4 z^{29}+11 z^{28}+7 z^{27}+8 
z^{26}+2 z^{25}+6 z^{24}+13 z^{23}\\
&\kern.5cm
+11 z^{21}+2 z^{20}+9 z^{19}+6 z^{18}+2 
z^{17}+15 z^{16}+5 z^{15}+3 z^{14}+14 z^{13}+2 z^{12}+10 z^{11}\\
&\kern.5cm
+9 z^{10}+6 
z^9+12 z^8,\\
q_{30}(z)&=
12 z^{59}+12 z^{58}+12 z^{55}+12 z^{53}+4 z^{51}+8 
z^{49}+8 z^{48}+8 z^{45}+8 z^{40}+8 z^{36},\\
q_{31}(z)&=
4 z^{55}+12 
z^{51}+4 z^{50}+12 z^{47}+4 z^{46}+12 z^{45}+4 z^{43}+12 z^{42}+12 z^{41}+12 
z^{38}+4 z^{34}+12 z^{33}\\
&\kern.5cm
+12 z^{32}+4 z^{29}+12 z^{26}+4 z^{25}+4 z^{22}+4 
z^{16}+4 z^{15}+12 z^{14}+4 z^{13}+4 z^9+12 z^8+12 z^6,\\
q_{32}(z)&=
5 z^{58}+10 z^{57}+4 
z^{56}+3 z^{54}+13 z^{53}+3 z^{52}+14 z^{51}+4 z^{50}+11 z^{49}+14 z^{48}+10 
z^{46}+2 z^{45}\\
&\kern.5cm
+2 z^{44}+2 z^{43}+10 z^{42}+9 z^{41}+15 z^{40}+14 z^{39}+11 
z^{38}+15 z^{37}+13 z^{36}+6 z^{34}+4 z^{33}\\
&\kern.5cm
+12 z^{32}+15 z^{31}+15 
z^{30}+12 z^{29}+9 z^{28}+13 z^{27}+6 z^{25}+2 z^{24}+15 z^{23}+9 z^{21}+6 
z^{20}+3 z^{19}\\
&\kern.5cm
+2 z^{18}+6 z^{17}+5 z^{16}+7 z^{15}+z^{14}+10 z^{13}+6 
z^{12}+14 z^{11}+3 z^{10}+2 z^9+4 z^8,\\
q_{33}(z)&=
12 z^{59}+12 z^{58}+12 z^{55}+12 z^{53}+4 z^{51},\\
q_{34}(z)&=
2 z^{55}+10 z^{54}+4 z^{50}+3 z^{49}+13 
z^{48}+4 z^{47}+7 z^{46}+7 z^{45}+15 z^{44}+5 z^{43}+15 z^{42}+15 z^{41}\\
&\kern.5cm
+8 
z^{40}+11 z^{39}+13 z^{38}+z^{37}+5 z^{36}+8 z^{35}+2 z^{34}+5 z^{32}+5 
z^{31}+14 z^{30}+8 z^{29}+4 z^{28}\\
&\kern.5cm
+15 z^{27}+10 z^{25}+15 z^{24}+10 
z^{23}+11 z^{22}+11 z^{21}+11 z^{20}+3 z^{19}+15 z^{18}+4 z^{17}+14 
z^{16}\\
&\kern.5cm
+z^{15}+12 z^{14}+13 z^{13}+11 z^{12}+11 z^{11}+11 z^{10}+z^9+5 z^8+13 
z^7+11 z^6+5 z^5+2 z^4,\\
q_{35}(z)&=
12 z^{65}+14 z^{57}+2 z^{56}+3 z^{55}+8 z^{53}+2 z^{52}+9 
z^{51}+15 z^{50}+6 z^{49}+4 z^{48}+z^{47}+15 z^{46}\\
&\kern.5cm
+z^{45}+14 z^{44}+15 
z^{43}+z^{42}+z^{41}+6 z^{40}+12 z^{39}+13 z^{38}+12 z^{37}+2 z^{36}+2 
z^{35}+7 z^{34}\\
&\kern.5cm
+5 z^{33}+z^{32}+14 z^{31}+10 z^{30}+11 z^{29}+6 z^{28}+10 
z^{27}+9 z^{26}+3 z^{25}+6 z^{24}+3 z^{22}+8 z^{21}\\
&\kern.5cm
+6 z^{20}+8 z^{19}+14 
z^{18}+4 z^{17}+3 z^{16}+7 z^{15}+9 z^{14}+7 z^{13}+8 z^{12}+14 z^{10}+11 
z^9+13 z^8\\
&\kern.5cm
+2 z^7+13 z^6,\\
q_{36}(z)&=
8 z^{67}+12 z^{58}+14 z^{55}+4 
z^{54}+4 z^{53}+4 z^{52}+2 z^{51}+6 z^{50}+12 z^{49}+12 z^{48}+10 z^{47}+6 
z^{45}\\
&\kern.5cm
+8 z^{44}+12 z^{43}+12 z^{42}+6 z^{41}+12 z^{40}+6 z^{39}+4 z^{38}+2 
z^{36}+10 z^{35}+12 z^{34}+2 z^{33}\\
&\kern.5cm
+6 z^{32}+12 z^{31}+10 z^{29}+12 
z^{28}+12 z^{26}+8 z^{25}+12 z^{22}+14 z^{21}+8 z^{19}+4 z^{18}+14 z^{17}\\
&\kern.5cm
+4 
z^{16}+8 z^{15},\\
q_{37}(z)&=
4 
z^{60}+2 z^{59}+8 z^{57}+12 z^{56}+2 z^{55}+12 z^{54}+8 
z^{53},\\
q_{38}(z)&=
5 z^{58}+10 z^{57}+4 z^{56}+3 
z^{54}+13 z^{53}+3 z^{52}+14 z^{51}+4 z^{50}+11 z^{49}+14 z^{48}+10 z^{46}+2 
z^{45}\\
&\kern.5cm
+2 z^{44}+2 z^{43}+10 z^{42}+9 z^{41}+15 z^{40}+14 z^{39}+11 z^{38}+15 
z^{37}+13 z^{36}+6 z^{34}+4 z^{33}\\
&\kern.5cm
+12 z^{32}+15 z^{31}+15 z^{30}+12 z^{29}+9 
z^{28}+13 z^{27}+6 z^{25}+2 z^{24}+15 z^{23}+9 z^{21}+6 z^{20}\\
&\kern.5cm
+3 z^{19}+2 
z^{18}+6 z^{17}+5 z^{16}+7 z^{15}+z^{14}+10 z^{13}+6 z^{12}+14 z^{11}+3 
z^{10}+2 z^9+4 z^8,\\
q_{39}(z)&=
4 
z^{58}+10 z^{55}+12 z^{54}+12 z^{53}+12 z^{52}+6 z^{51}+2 z^{50}+4 z^{49}+4 
z^{48}+14 z^{47}+2 z^{45}+4 z^{43}\\
&\kern.5cm
+4 z^{42}+2 z^{41}+4 z^{40}+2 z^{39}+12 
z^{38}+6 z^{36}+14 z^{35}+4 z^{34}+6 z^{33}+2 z^{32}+4 z^{31}+14 z^{29}\\
&\kern.5cm
+4 
z^{28}+4 z^{26}+4 z^{22}+10 z^{21}+12 z^{18}+10 z^{17}+12 
z^{16},\\
q_{40}(z)&=
12 z^{59}+12 
z^{58}+12 z^{55}+12 z^{53}+4 z^{51},\\
q_{41}(z)&=
4 z^{60}+2 
z^{59}+12 z^{56}+2 z^{55}+12 z^{54},\\
q_{42}(z)&=
2 z^{56}+14 
z^{55}+2 z^{54}+5 z^{52}+10 z^{51}+6 z^{50}+5 z^{49}+14 z^{48}+z^{47}+10 
z^{46}+6 z^{45}+7 z^{44}\\
&\kern.5cm
+12 z^{43}+11 z^{42}+5 z^{41}+11 z^{40}+6 z^{39}+12 
z^{38}+4 z^{37}+14 z^{36}+11 z^{35}+15 z^{34}+10 z^{33}\\
&\kern.5cm
+11 z^{32}+3 z^{31}+4 
z^{30}+13 z^{29}+6 z^{28}+3 z^{27}+2 z^{26}+12 z^{25}+4 z^{24}+8 z^{23}+7 
z^{22}+10 z^{21}\\
&\kern.5cm
+2 z^{20}+9 z^{19}+11 z^{18}+12 z^{17}+12 z^{16}+8 
z^{14}+z^{13}+3 z^{11}+7 z^9+9 z^8+15 z^6+12 z^5,\\
q_{43}(z)&=
8 z^{66}+3 
z^{58}+6 z^{57}+12 z^{56}+5 z^{54}+11 z^{53}+5 z^{52}+2 z^{51}+12 z^{50}+13 
z^{49}+2 z^{48}+6 z^{46}\\
&\kern.5cm
+14 z^{45}+14 z^{44}+14 z^{43}+6 z^{42}+15 z^{41}+9 
z^{40}+2 z^{39}+13 z^{38}+9 z^{37}+11 z^{36}+10 z^{34}\\
&\kern.5cm
+12 z^{33}+4 z^{32}+9 
z^{31}+9 z^{30}+4 z^{29}+15 z^{28}+11 z^{27}+8 z^{26}+10 z^{25}+14 z^{24}+9 
z^{23}\\
&\kern.5cm
+15 z^{21}+10 z^{20}+5 z^{19}+14 z^{18}+10 z^{17}+3 z^{16}+z^{15}+7 
z^{14}+6 z^{13}+10 z^{12}+2 z^{11}\\
&\kern.5cm
+5 z^{10}+14 z^9+12 z^8,\\
q_{44}(z)&=
2 z^{59}+2 z^{58}+8 z^{56}+10 z^{55}+2 z^{53}+8 z^{52}+6 
z^{51}+12 z^{49}+4 z^{48}+8 z^{47}+8 z^{46}+4 z^{45}\\
&\kern.5cm
+8 z^{44}+8 z^{41}+4 
z^{40}+8 z^{37}+4 z^{36}+8 z^{35},\\
q_{45}(z)&=
8 z^{21} ,\\
q_{46}(z)&=
6 z^{59}+6 z^{58}+14 z^{55}+6 z^{53}+2 z^{51}+4 z^{49}+12 
z^{48}+12 z^{45}+12 z^{40}+12 z^{36},\\
q_{47}(z)&=
4 
z^{65}+2 z^{57}+14 z^{56}+13 z^{55}+8 z^{53}+14 z^{52}+7 z^{51}+z^{50}+10 
z^{49}+12 z^{48}+15 z^{47}+z^{46}\\
&\kern.5cm
+15 z^{45}+2 z^{44}+z^{43}+15 z^{42}+15 
z^{41}+10 z^{40}+4 z^{39}+3 z^{38}+4 z^{37}+14 z^{36}+14 z^{35}+9 z^{34}\\
&\kern.5cm
+11 
z^{33}+15 z^{32}+2 z^{31}+6 z^{30}+5 z^{29}+10 z^{28}+6 z^{27}+7 z^{26}+13 
z^{25}+10 z^{24}+13 z^{22}+8 z^{21}\\
&\kern.5cm
+10 z^{20}+8 z^{19}+2 z^{18}+12 z^{17}+13 
z^{16}+9 z^{15}+7 z^{14}+9 z^{13}+8 z^{12}+2 z^{10}+5 z^9+3 z^8\\
&\kern.5cm
+14 z^7+3 
z^6,\\
q_{48}(z)&=
8 
z^{58}+4 z^{55}+8 z^{54}+8 z^{53}+8 z^{52}+12 z^{51}+4 z^{50}+8 z^{49}+8 
z^{48}+12 z^{47}+4 z^{45}+8 z^{43}+8 z^{42}\\
&\kern.5cm
+4 z^{41}+8 z^{40}+4 z^{39}+8 
z^{38}+12 z^{36}+12 z^{35}+8 z^{34}+12 z^{33}+4 z^{32}+8 z^{31}+12 z^{29}+8 
z^{28}\\
&\kern.5cm
+8 z^{26}+8 z^{22}+4 z^{21}+8 z^{18}+4 z^{17}+8 
z^{16},\\
q_{49}(z)&=
4 z^{60}+10 z^{59}+8 z^{57}+12 
z^{56}+10 z^{55}+12 z^{54}+8 z^{53},\\
q_{50}(z)&=
4 
z^{59}+4 z^{58}+4 z^{55}+4 z^{53}+12 z^{51},\\
q_{51}(z)&=
12 z^{60}+14 z^{59}+4 z^{56}+14 z^{55}+4 
z^{54},\\
q_{52}(z)&=
8 z^{66}+7 z^{58}+14 z^{57}+12 z^{56}+z^{54}+15 z^{53}+z^{52}+10 
z^{51}+12 z^{50}+9 z^{49}+10 z^{48}+14 z^{46}\\
&\kern.5cm
+6 z^{45}+6 z^{44}+6 z^{43}+14 
z^{42}+3 z^{41}+5 z^{40}+10 z^{39}+9 z^{38}+5 z^{37}+15 z^{36}+2 z^{34}+12 
z^{33}\\
&\kern.5cm
+4 z^{32}+5 z^{31}+5 z^{30}+4 z^{29}+3 z^{28}+15 z^{27}+8 z^{26}+2 
z^{25}+6 z^{24}+5 z^{23}+3 z^{21}+2 z^{20}\\
&\kern.5cm
+z^{19}+6 z^{18}+2 z^{17}+7 
z^{16}+13 z^{15}+11 z^{14}+14 z^{13}+2 z^{12}+10 z^{11}+z^{10}+6 z^9+12 
z^8,\\
q_{53}(z)&=
4 z^{59}+4 z^{58}+4 z^{55}+4 
z^{53}+12 z^{51}+8 z^{49}+8 z^{48}+8 z^{45}+8 z^{40}+8 
z^{36},\\
q_{54}(z)&=
4 z^{67}+14 
z^{58}+8 z^{57}+15 z^{55}+10 z^{54}+10 z^{53}+2 z^{52}+z^{51}+11 z^{50}+14 
z^{49}+6 z^{48}+5 z^{47}\\
&\kern.5cm
+8 z^{46}+3 z^{45}+12 z^{44}+6 z^{43}+6 z^{42}+3 
z^{41}+6 z^{40}+3 z^{39}+2 z^{38}+9 z^{36}+5 z^{35}+6 z^{34}\\
&\kern.5cm
+9 z^{33}+3 
z^{32}+6 z^{31}+8 z^{30}+13 z^{29}+6 z^{28}+6 z^{26}+4 z^{25}+8 z^{24}+14 
z^{22}+15 z^{21}+4 z^{19}\\
&\kern.5cm
+2 z^{18}+15 z^{17}+10 z^{16}+4 
z^{15},\\
q_{55}(z)&=
6 
z^{60}+3 z^{59}+4 z^{57}+10 z^{56}+3 z^{55}+2 z^{54}+4 
z^{53},\\
q_{56}(z)&=
5 z^{59}+13 z^{58}+12 z^{56}+z^{55}+8 
z^{54}+5 z^{53}+12 z^{52}+15 z^{51}+8 z^{50}+14 z^{49}+2 z^{48}+4 z^{47}\\
&\kern.5cm
+4 
z^{46}+10 z^{45}+12 z^{44}+8 z^{43}+4 z^{41}+2 z^{40}+8 z^{38}+12 z^{37}+2 
z^{36}+12 z^{35}+8 z^{34},\\
q_{57}(z)&=
14 z^{60}+15 z^{59}+4 z^{57}+2 z^{56}+15 
z^{55}+10 z^{54}+4 z^{53},\\
q_{58}(z)&=
8 z^{57}+8 z^{56}+8 z^{55}+12 z^{53}+8 z^{51}+8 
z^{50}+8 z^{49}+8 z^{48}+10 z^{47}+14 z^{46}+10 z^{45}+8 z^{44}\\
&\kern.5cm
+9 z^{43}+12 
z^{42}+2 z^{41}+6 z^{40}+11 z^{39}+3 z^{38}+12 z^{37}+8 z^{36}+7 z^{35}+2 
z^{34}+z^{33}+z^{32}\\
&\kern.5cm
+7 z^{31}+12 z^{30}+12 z^{29}+z^{28}+z^{26}+15 z^{25}+9 
z^{24}+14 z^{23}+z^{22}+3 z^{21}+3 z^{20}+4 z^{19}\\
&\kern.5cm
+7 z^{17}+7 z^{16}+6 
z^{15}+3 z^{14}+z^{13}+3 z^{12}+2 z^{11}+14 z^9+2 z^8+6 z^7+2 z^5+14 z^4\\
&\kern.5cm
+6 
z^3+14 z^2,\\
q_{59}(z)&=
4 z^{56}+12 z^{55}+4 z^{54}+2 z^{52}+4 z^{51}+12 z^{50}+2 z^{49}+12 
z^{48}+10 z^{47}+4 z^{46}+12 z^{45}+6 z^{44}\\
&\kern.5cm
+14 z^{42}+2 z^{41}+14 z^{40}+12 
z^{39}+12 z^{36}+14 z^{35}+6 z^{34}+4 z^{33}+14 z^{32}+14 z^{31}+2 z^{29}\\
&\kern.5cm
+12 
z^{28}+14 z^{27}+4 z^{26}+6 z^{22}+4 z^{21}+4 z^{20}+10 z^{19}+14 z^{18}+10 
z^{13}+14 z^{11}+6 z^9\\
&\kern.5cm
+10 z^8+6 z^6,\\
q_{60}(z)&=
12 z^{65}+14 z^{57}+2 z^{56}+3 z^{55}+2 z^{52}+9 z^{51}+15 
z^{50}+6 z^{49}+4 z^{48}+z^{47}+15 z^{46}+z^{45}\\
&\kern.5cm
+14 z^{44}+15 
z^{43}+z^{42}+z^{41}+6 z^{40}+12 z^{39}+13 z^{38}+12 z^{37}+2 z^{36}+2 
z^{35}+7 z^{34}+5 z^{33}\\
&\kern.5cm
+z^{32}+14 z^{31}+10 z^{30}+11 z^{29}+6 z^{28}+10 
z^{27}+9 z^{26}+3 z^{25}+6 z^{24}+3 z^{22}+6 z^{20}+14 z^{18}\\
&\kern.5cm
+4 z^{17}+3 
z^{16}+7 z^{15}+9 z^{14}+7 z^{13}+14 z^{10}+11 z^9+13 z^8+2 z^7+13 
z^6,\\
q_{61}(z)&=
2 z^{60}+9 z^{59}+14 z^{56}+9 z^{55}+6 
z^{54},\\
q_{62}(z)&=
3 z^{58}+5 z^{54}+11 
z^{53}+5 z^{52}+13 z^{49}+15 z^{41}+9 z^{40}+13 z^{38}+9 z^{37}+11 z^{36}+9 
z^{31}+9 z^{30}\\
&\kern.5cm
+15 z^{28}+11 z^{27}+9 z^{23}+15 z^{21}+5 z^{19}+3 
z^{16}+z^{15}+7 z^{14}+5 z^{10},\\
q_{63}(z)&=
6 z^{55}+14 z^{54}+12 z^{50}+9 z^{49}+7 z^{48}+12 z^{47}+5 z^{46}+5 
z^{45}+13 z^{44}+15 z^{43}+13 z^{42}\\
&\kern.5cm
+13 z^{41}+8 z^{40}+z^{39}+7 z^{38}+3 
z^{37}+15 z^{36}+8 z^{35}+6 z^{34}+15 z^{32}+15 z^{31}+10 z^{30}\\
&\kern.5cm
+8 z^{29}+12 
z^{28}+13 z^{27}+14 z^{25}+13 z^{24}+14 z^{23}+z^{22}+z^{21}+z^{20}+9 
z^{19}+13 z^{18}\\
&\kern.5cm
+12 z^{17}+10 z^{16}+3 z^{15}+4 z^{14}+7 
z^{13}+z^{12}+z^{11}+z^{10}+3 z^9+15 z^8+7 z^7+z^6\\
&\kern.5cm
+15 z^5+6 
z^4,\\
q_{64}(z)&=
15 z^{58}+14 
z^{57}+12 z^{56}+9 z^{54}+7 z^{53}+9 z^{52}+10 z^{51}+12 z^{50}+z^{49}+10 
z^{48}+14 z^{46}+6 z^{45}\\
&\kern.5cm
+6 z^{44}+6 z^{43}+14 z^{42}+11 z^{41}+13 z^{40}+10 
z^{39}+z^{38}+13 z^{37}+7 z^{36}+2 z^{34}+12 z^{33}+4 z^{32}\\
&\kern.5cm
+13 z^{31}+13 
z^{30}+4 z^{29}+11 z^{28}+7 z^{27}+2 z^{25}+6 z^{24}+13 z^{23}+11 z^{21}+2 
z^{20}+9 z^{19}+6 z^{18}\\
&\kern.5cm
+2 z^{17}+15 z^{16}+5 z^{15}+3 z^{14}+14 z^{13}+2 
z^{12}+10 z^{11}+9 z^{10}+6 z^9+12 z^8,\\
q_{65}(z)&=
12 z^{58}+14 z^{55}+4 z^{54}+4 
z^{53}+4 z^{52}+2 z^{51}+6 z^{50}+12 z^{49}+12 z^{48}+10 z^{47}+6 z^{45}+12 
z^{43}\\
&\kern.5cm
+12 z^{42}+6 z^{41}+12 z^{40}+6 z^{39}+4 z^{38}+2 z^{36}+10 z^{35}+12 
z^{34}+2 z^{33}+6 z^{32}+12 z^{31}+10 z^{29}\\
&\kern.5cm
+12 z^{28}+12 z^{26}+12 
z^{22}+14 z^{21}+4 z^{18}+14 z^{17}+4 z^{16},\\
q_{66}(z)&=
12 z^{65}+14 z^{57}+2 z^{56}+3 z^{55}+8 
z^{53}+2 z^{52}+9 z^{51}+15 z^{50}+6 z^{49}+4 z^{48}+z^{47}+15 
z^{46}+z^{45}\\
&\kern.5cm
+14 z^{44}+15 z^{43}+z^{42}+z^{41}+6 z^{40}+12 z^{39}+13 
z^{38}+12 z^{37}+2 z^{36}+2 z^{35}+7 z^{34}+5 z^{33}\\
&\kern.5cm
+z^{32}+14 z^{31}+10 
z^{30}+11 z^{29}+6 z^{28}+10 z^{27}+9 z^{26}+3 z^{25}+6 z^{24}+3 z^{22}+8 
z^{21}+6 z^{20}\\
&\kern.5cm
+8 z^{19}+14 z^{18}+4 z^{17}+3 z^{16}+7 z^{15}+9 z^{14}+7 
z^{13}+8 z^{12}+14 z^{10}+11 z^9+13 z^8\\
&\kern.5cm
+2 z^7+13 z^6,\\
q_{67}(z)&=
12 z^{59}+12 z^{58}+12 z^{55}+12 
z^{53}+4 z^{51},\\
q_{68}(z)&=
4 z^{60}+2 z^{59}+12 z^{56}+2 
z^{55}+12 z^{54},\\
q_{69}(z)&=
4 z^{67}+14 z^{58}+8 z^{57}+7 z^{55}+10 
z^{54}+10 z^{53}+2 z^{52}+9 z^{51}+3 z^{50}+14 z^{49}+6 z^{48}+13 z^{47}\\
&\kern.5cm
+8 
z^{46}+11 z^{45}+12 z^{44}+6 z^{43}+6 z^{42}+11 z^{41}+6 z^{40}+11 z^{39}+2 
z^{38}+z^{36}+13 z^{35}+6 z^{34}\\
&\kern.5cm
+z^{33}+11 z^{32}+6 z^{31}+8 z^{30}+5 
z^{29}+6 z^{28}+6 z^{26}+4 z^{25}+8 z^{24}+14 z^{22}+7 z^{21}+4 z^{19}\\
&\kern.5cm
+2 
z^{18}+7 z^{17}+10 z^{16}+4 z^{15},\\
q_{70}(z)&=
14 z^{60}+15 z^{59}+4 
z^{57}+2 z^{56}+15 z^{55}+10 z^{54}+4 z^{53},\\
q_{71}(z)&=
4 z^{51}+14 z^{50}+4 z^{49}+14 z^{47}+5 z^{46}+12 
z^{45}+14 z^{43}+11 z^{42}+9 z^{41}+11 z^{40}+6 z^{39}+11 z^{38}\\
&\kern.5cm
+z^{37}+2 
z^{36}+13 z^{35}+11 z^{33}+2 z^{31}+3 z^{29}+10 z^{28}+9 z^{27}+7 z^{26}+6 
z^{24}+2 z^{23}+12 z^{22}\\
&\kern.5cm
+2 z^{21}+15 z^{20}+11 z^{19}+14 z^{18}+11 
z^{17}+14 z^{15}+4 z^{14}+11 z^{13}+3 z^{12}+3 z^{11}+13 z^{10}\\
&\kern.5cm
+12 z^9+4 
z^8+z^7+6 z^6+3 z^5+z^4+3 z^3,\\
q_{72}(z)&=
5 z^{58}+10 z^{57}+4 z^{56}+3 
z^{54}+13 z^{53}+3 z^{52}+14 z^{51}+11 z^{49}+14 z^{48}+10 z^{46}+2 z^{45}+2 
z^{44}\\
&\kern.5cm
+2 z^{43}+10 z^{42}+9 z^{41}+15 z^{40}+14 z^{39}+11 z^{38}+15 
z^{37}+13 z^{36}+6 z^{34}+15 z^{31}+15 z^{30}\\
&\kern.5cm
+9 z^{28}+13 z^{27}+6 z^{25}+2 
z^{24}+15 z^{23}+9 z^{21}+6 z^{20}+3 z^{19}+2 z^{18}+6 z^{17}+5 z^{16}+7 
z^{15}\\
&\kern.5cm
+z^{14}+10 z^{13}+6 z^{12}+14 z^{11}+3 z^{10}+2 z^9+4 
z^8,\\
q_{73}(z)&=
2 z^{55}+6 z^{51}+10 z^{50}+6 z^{47}+10 
z^{46}+6 z^{45}+10 z^{43}+6 z^{42}+6 z^{41}+14 z^{38}+10 z^{34}+14 z^{33}\\
&\kern.5cm
+6 
z^{32}+2 z^{29}+6 z^{26}+2 z^{25}+2 z^{22}+2 z^{16}+10 z^{15}+6 z^{14}+10 
z^{13}+2 z^9+14 z^8+14 z^6,\\
q_{74}(z)&=
2 z^{55}+10 z^{54}+4 
z^{50}+3 z^{49}+13 z^{48}+4 z^{47}+7 z^{46}+7 z^{45}+15 z^{44}+5 z^{43}+15 
z^{42}+15 z^{41}\\
&\kern.5cm
+11 z^{39}+13 z^{38}+z^{37}+5 z^{36}+2 z^{34}+5 z^{32}+5 
z^{31}+14 z^{30}+4 z^{28}+15 z^{27}+10 z^{25}\\
&\kern.5cm
+15 z^{24}+10 z^{23}+11 
z^{22}+11 z^{21}+11 z^{20}+3 z^{19}+15 z^{18}+4 z^{17}+14 z^{16}+z^{15}+12 
z^{14}\\
&\kern.5cm
+13 z^{13}+11 z^{12}+11 z^{11}+11 z^{10}+z^9+5 z^8+13 z^7+11 z^6+5 
z^5+2 z^4,\\
q_{75}(z)&=
10 
z^{58}+13 z^{55}+14 z^{54}+14 z^{53}+6 z^{52}+3 z^{51}+z^{50}+10 z^{49}+2 
z^{48}+15 z^{47}+9 z^{45}+2 z^{43}\\
&\kern.5cm
+2 z^{42}+9 z^{41}+2 z^{40}+9 z^{39}+6 
z^{38}+11 z^{36}+15 z^{35}+2 z^{34}+11 z^{33}+9 z^{32}+2 z^{31}+7 z^{29}\\
&\kern.5cm
+2 
z^{28}+2 z^{26}+10 z^{22}+13 z^{21}+6 z^{18}+13 z^{17}+14 
z^{16},\\
q_{76}(z)&=
5 z^{52}+5 z^{49}+z^{47}+7 z^{44}+11 z^{42}+5 
z^{41}+11 z^{40}+11 z^{35}+15 z^{34}+11 z^{32}+3 z^{31}+13 z^{29}\\
&\kern.5cm
+3 z^{27}+7 
z^{22}+9 z^{19}+11 z^{18}+z^{13}+3 z^{11}+7 z^9+9 z^8+15 
z^6,\\
q_{77}(z)&=
13 z^{55}+7 z^{51}+z^{50}+15 
z^{47}+z^{46}+15 z^{45}+z^{43}+15 z^{42}+15 z^{41}+3 z^{38}+9 z^{34}+11 
z^{33}\\
&\kern.5cm
+15 z^{32}+5 z^{29}+7 z^{26}+13 z^{25}+13 z^{22}+13 z^{16}+9 z^{15}+7 
z^{14}+9 z^{13}+5 z^9+3 z^8+3 z^6.
\end{align*}}%

\section{The method ``reloaded"}
\label{appD}

As the reader will recall from Section~\ref{sec:method}
(cf.\ Remark~\ref{rem:comp}), 
our method described there is based on the ``hope"
that, if a polynomial in $\Phi(z)$ is zero modulo a $2$-power $2^\be$
(as a formal Laurent series), then already all coefficients
of powers of $\Phi(z)$ in this polynomial vanish modulo $2^\be$.
(This is manifest in each comparison of coefficients of powers of
$\Phi(z)$ in Section~\ref{sec:method}.) 
In general, however, this implication does not
hold (see Lemma~\ref{lem:Null16} below for the case of modulus
$2^4=16$). It may consequently happen
that the method from Section~\ref{sec:method} fails to find a
solution modulo $2^\be$
to a given differential equation in the form of a polynomial in
$\Phi(z)$ with coefficients that are Laurent polynomials in $z$ over
the integers, while such a solution may in fact
exist. As it turns out, this situation occurs when treating the
subgroup numbers of $SL_2(\Z)$ and of $\Ga_3(3)$ modulo~$16$,
see the paragraph above
Theorem~\ref{thm:Sl2Z16} and Remark~\ref{rem:Ga3}.
(In the former case, there is indeed a solution, while in the
latter there is not.)

Our aim here is to explain how the method from
Section~\ref{sec:method} can be enhanced so that one can {\it decide}
whether or not such a solution modulo a given $2$-power $2^\be$
exists; and, if it exists, how to find it. In principle, it should be
possible to describe such an improved method for an arbitrary $2$-power
$2^\be$. Since, in the present paper, we need it only for the modulus
$16$, and since we are not able to rigorously establish the validity of
the enhancement we have in mind in general (it would depend on
Conjecture~\ref{conj:1}, which at present we are not able to prove), 
we content ourselves with describing the enhanced method for the
modulus $16$. From this description, the reader should have no
difficulty to ``extrapolate" to arbitrary $2$-powers, assuming the
truth of Conjecture~\ref{conj:1}. 
%(The embarrassment that we are not
%able to prove this conjecture is the reason that we fail to put the
%improved method on rigorous grounds.)

\medskip
We begin by characterising when a polynomial in $\Phi(z)$
with coefficients that are Laurent polynomials in $z$ over the
integers vanishes modulo $16$ as a Laurent series in $z$.

\begin{lemma} \label{lem:Null16}
As before, let $\Phi(z)=\sum _{n\ge0} ^{}z^{2^n}$.
Furthermore,
let $P(z,\Phi(z))$ be a polynomial in $\Phi(z)$ with coefficients that are
Laurent polynomials in $z$ over the integers. Then, as a Laurent
series in $z$,
$$P(z,\Phi(z))=0\quad \text {\em modulo }16
$$
if, and only if, 
the coefficients of powers of $\Phi$ in $P(z,\Phi(z))$ agree
modulo~$16$ with the corresponding ones in
\begin{multline} \label{eq:Pmod16}
c_1(z)M_{16}(z,\Phi(z))+2\big(c_2(z)\Phi(z)+c_3(z)\big)M_8(z,\Phi(z))\\
+
8\big(c_4(z)\Phi(z)+c_5(z)\big)M_2(z,\Phi(z)).
\end{multline}
Here, $M_2(z,t),M_8(z,t),M_{16}(z,t)$ are the minimal polynomials for the
moduli $2,8,16$, respectively, given in
Proposition~{\em\ref{prop:minpol}}, 
and 
$c_1(z),c_2(z),c_3(z),c_4(z),c_5(z)$ are suitable Laurent polynomials
in $z$ over the integers.
\end{lemma}

\begin{proof}
We assume that $P(z,\Phi(z))=0\ \text {modulo }16$.

Recall that, by definition, $M_{16}(z,\Phi(z))$ is a {\it monic}
polynomial in $\Phi(z)$. We use this fact to
perform division of $P(z,\Phi(z))$ by $M_{16}(z,\Phi(z))$ (as
polynomials in $\Phi(z)$),
thus obtaining
\begin{equation} \label{eq:R1}
P(z,\Phi(z))=c_1(z)M_{16}(z)+P_1(z,\Phi(z)),
\end{equation}
where $P_1(z,\Phi(z))$ is a polynomial in $\Phi(z)$ of degree at most
$5$, with coefficients that are Laurent polynomials in $z$ over the
integers, say
\begin{multline} \label{eq:R1a}
P_1(z,\Phi(z))=
d_5(z)\Phi^5(z)+
d_4(z)\Phi^4(z)+
d_3(z)\Phi^3(z)\\
+
d_2(z)\Phi^2(z)+
d_1(z)\Phi(z)+
d_0(z).
\end{multline}
As Laurent series in $z$, both $P(z,\Phi(z))$ and $M_{16}(z,\Phi(z))$ 
vanish modulo~$16$. Using this observation in \eqref{eq:R1}, we see
that $P_1(z,\Phi(z))$ vanishes modulo~16 as well.
Now recall from \eqref{eq:PhiE} and the proof of
Lemma~\ref{lem:minpol} that 
$$
P_1(z,\Phi(z))=
d_5(z)\,5!\,E_5(z)+Q_1(z)
$$
(with a suitable series $Q_1(z)$), where 
$D(Q_1(z),16;n)$ has strictly smaller
asymptotic growth (in $n$) than $D(E_5(z),16;n)$.
Since, as we already observed, $P_1(z,\Phi(z))$ vanishes modulo~16,
it follows that $5!d_5(z)$ must vanish modulo~16, that is, 
there exists a Laurent polynomial $c_2(z)$ over the integers such
that $d_5(z)=2c_2(z)$. We use this observation in \eqref{eq:R1a} 
to see that
\begin{equation} \label{eq:R2}
P(z,\Phi(z))=c_1(z)M_{16}(z)+2c_2(z)\Phi(z)M_8(z,\Phi(z))+P_2(z,\Phi(z)),
\end{equation}
where $P_2(z,\Phi(z))$ is a polynomial in $\Phi(z)$ of degree at most
$4$, with coefficients that are Laurent polynomials in $z$ over the
integers, say
\begin{equation*} %\label{eq:R2a}
P_2(z,\Phi(z))=
e_4(z)\Phi^4(z)+
e_3(z)\Phi^3(z)+
e_2(z)\Phi^2(z)+
e_1(z)\Phi(z)+
e_0(z).
\end{equation*}
Applying the same kind of argument again, we further deduce that
\begin{equation} \label{eq:R3}
P(z,\Phi(z))=c_1(z)M_{16}(z)+2\big(c_2(z)\Phi(z)+c_3(z)\big)M_8(z,\Phi(z))
+P_3(z,\Phi(z)),
\end{equation}
where $c_3(z)$ is a Laurent polynomial in $z$ over the integers and
$P_3(z,\Phi(z))$ is a polynomial in $\Phi(z)$ of degree at most
$3$, with coefficients that are Laurent polynomials in $z$ over the
integers, say
\begin{equation*} %\label{eq:R3a}
P_3(z,\Phi(z))=
f_3(z)\Phi^3(z)+
f_2(z)\Phi^2(z)+
f_1(z)\Phi(z)+
f_0(z).
\end{equation*}
As Laurent series in $z$, all of $P(z,\Phi(z))$, $M_{16}(z,\Phi(z))$,
and $2M_8(z,\Phi(z))$ vanish modulo~$16$. 
Using this observation in \eqref{eq:R3}, we see
that $P_3(z,\Phi(z))$ vanishes modulo~16 as well.
Equation~\eqref{eq:PhiE} and the proof of
Lemma~\ref{lem:minpol} give that
$$
P_3(z,\Phi(z))=
d_3(z)\,3!\,E_3(z)+Q_3(z)
$$
(with a suitable series $Q_3(z)$), where 
$D(Q_3(z),16;n)$ has strictly smaller
asymptotic growth (in $n$) than $D(E_3(z),16;n)$.
Since, as we already observed, $P_3(z,\Phi(z))$ vanishes modulo~16,
it follows that $3!d_3(z)$ must vanish modulo~16, that is, 
there exists a Laurent polynomial $c_4(z)$ over the integers such
that $d_3(z)=8c_4(z)$. By another application of the same kind of argument,
this leads to
\begin{multline} \label{eq:R4}
P(z,\Phi(z))=c_1(z)M_{16}(z)+2\big(c_2(z)\Phi(z)+c_3(z)\big)M_8(z,\Phi(z))\\
+
8\big(c_4(z)\Phi(z)+c_5(z)\big)M_2(z,\Phi(z))
+P_4(z,\Phi(z)),
\end{multline}
where $c_5(z)$ is a Laurent polynomial in $z$ over the
integers and
$P_4(z,\Phi(z))$ is a polynomial in $\Phi(z)$ of degree at most
$1$, with coefficients that are Laurent polynomials in $z$ over the
integers, say
\begin{equation*} %\label{eq:R4a}
P_4(z,\Phi(z))=
g_1(z)\Phi(z)+
g_0(z).
\end{equation*}
Since $P(z,\Phi(z))$, $M_{16}(z,\Phi(z))$, 
$2M_8(z,\Phi(z)))$, $8M_2(z,\Phi(z))$
all vanish modulo~$16$, also $P_4(z,\Phi(z))$ must have this
property; but this means that $g_1(z)$ and $g_0(z)$ both vanish
modulo $16$. 

If we combine \eqref{eq:R1}, 
\eqref{eq:R2},
\eqref{eq:R3},
\eqref{eq:R4}, then we obtain our claim.
\end{proof}

Now we put ourselves in the situation that we want to describe the
coefficients of the formal power series $F(z)$ modulo~$16$, where
$F(z)$ solves a Riccati-type differential equation of the form
\eqref{eq:diffeq}, and that we try to solve this problem by
expressing $F(z)$ in the form
\begin{equation*} %\label{eq:Ansatz}
F(z)=\sum _{i=0} ^{5}a_i(z)\Phi^i(z)\quad \text {modulo
}16,
\end{equation*}
where the $a_i(z)$'s are Laurent polynomials in $z$ over the integers 
to be determined.
Let us assume that, while following the approach outlined in
Section~\ref{sec:method} (with $M_{16}(z,\Phi(z))$ in place of the
polynomial in \eqref{eq:PhiRel}), we have already reached the level of
modulus~$8$, that is, that we have found Laurent polynomials 
$a_{0,3}(z),
a_{1,3}(z),
a_{2,3}(z),
a_{3,3}(z),
a_{4,3}(z),
a_{5,3}(z)$
such that
$$
\sum _{i=0} ^{5}a_{i,3}(z)\Phi^i(z)
$$
solves the differential equation \eqref{eq:diffeq} modulo~$8$.
According to the Ansatz \eqref{eq:Ansatz2}--\eqref{eq:Ansatz2b} with
$\be=3$, we now substitute 
\begin{equation} \label{eq:bi4}
\sum _{i=0} ^{5}(a_{i,3}(z)+8b_{i,4}(z))\Phi^i(z)
\end{equation}
(where the $b_{i,4}(z)$'s are at this point undetermined Laurent
polynomials in $z$)
instead of $F(z)$ in \eqref{eq:diffeq}. For the sake of better
readability, in the sequel we write $b_0(z)$ for $b_{0,4}(z)$, etc.
After simplification of the left-hand side of \eqref{eq:diffeq}
modulo $16$ as
described below \eqref{eq:Ansatz2b}, and after reduction of the
resulting expression 
modulo~$M_{16}(z,\Phi(z))$ (which is a polynomial in $\Phi(z)$ of
degree~6), we obtain a polynomial of the form
\begin{equation} \label{eq:ExpG}
8\sum _{i=0} ^{5}\Big(p_i(z)+G_i\big(z,\mathbf b(z),\mathbf
b'(z)\big)\Big)\Phi^i(z),
\end{equation}
where the $p_i(z)$'s are certain Laurent polynomials in $z$ over the
integers, and
the $G_i\big(z,\mathbf b(z),\mathbf b'(z)\big)$'s are certain linear
forms in
$$
b_0(z),
b_1(z),
b_2(z),
b_3(z),
b_4(z),
b_5(z)\text { and }
b'_0(z),
b'_1(z),
b'_2(z),
b'_3(z),
b'_4(z),
b'_5(z),
$$
with coefficients that are (known) Laurent polynomials in $z$ over the
integers. 

Our goal is to find Laurent polynomials
$b_0(z),b_1(z),b_2(z),b_3(z),b_4(z),b_5(z)$
such that the expression \eqref{eq:ExpG} is zero modulo~$16$ as
Laurent series in $z$. Lemma~\ref{lem:Null16}, combined with the
explicit forms of $M_2(z,t)$ and $M_8(z,t)$ given in
Proposition~\ref{prop:minpol}, then says that
\begin{align*} %\label{}
8\Big(p_0(z)+G_0\big(z,\mathbf b(z),\mathbf b'(z)\big)\Big)
&=(4z+10z^2)c_3(z)+8zc_5(z)
\quad \text {modulo }16,\\
8\Big(p_1(z)+G_1\big(z,\mathbf b(z),\mathbf b'(z)\big)\Big)
&=(4z+10z^2)c_2(z)+(12+4z)c_3(z)+8zc_4(z)+8c_5(z)\\
&\kern7.7cm
\quad \text {modulo }16,\\
8\Big(p_2(z)+G_2\big(z,\mathbf b(z),\mathbf b'(z)\big)\Big)
&=(12+4z)c_2(z)+(6+4z)c_3(z)+8c_4(z)+8c_5(z)\\
&\kern7.7cm
\quad \text {modulo }16,\\
8\Big(p_3(z)+G_3\big(z,\mathbf b(z),\mathbf b'(z)\big)\Big)
&=(6+4z)c_2(z)+12c_3(z)+8c_4(z)
\quad \text {modulo }16,\\
8\Big(p_4(z)+G_4\big(z,\mathbf b(z),\mathbf b'(z)\big)\Big)
&=12c_2(z)+2c_3(z)
\quad \text {modulo }16,\\
8\Big(p_5(z)+G_5\big(z,\mathbf b(z),\mathbf b'(z)\big)\Big)&=2c_2(z)
\quad \text {modulo }16,
\end{align*}
for suitable Laurent polynomials $c_2(z),c_3(z),c_4(z),c_5(z)$.
From the last congruence one sees that $c_2(z)$ is actually zero
modulo~$4$, and then the next-to-last congruence implies that the
same holds for $c_3(z)$. Writing 
$4a(z)=c_2(z)$,
$4b(z)=c_3(z)$,
$c(z)=c_4(z)$,
$d(z)=c_5(z)$, we see that the above system of congruences simplifies
to
\begin{align} \notag
G_0\big(z,\mathbf b(z),\mathbf b'(z)\big)&=p_0(z)+z^2b(z)+zd(z)
\quad \text {modulo }2,\\
\notag
G_1\big(z,\mathbf b(z),\mathbf b'(z)\big)
&=p_1(z)+z^2a(z)+zc(z)+d(z)
\quad \text {modulo }2,\\
\notag
G_2\big(z,\mathbf b(z),\mathbf b'(z)\big)
&=p_2(z)+b(z)+c(z)+d(z)
\quad \text {modulo }2,\\
\notag
G_3\big(z,\mathbf b(z),\mathbf b'(z)\big)
&=p_3(z)+a(z)+c(z)
\quad \text {modulo }2,\\
\notag
G_4\big(z,\mathbf b(z),\mathbf b'(z)\big)&=p_4(z)+b(z)
\quad \text {modulo }2,\\
G_5\big(z,\mathbf b(z),\mathbf b'(z)\big)&=p_5(z)+a(z)
\quad \text {modulo }2.
\label{eq:abcd}
\end{align}
This puts us in the situation of Lemma~\ref{lem:2x2diff}, except that
on the right-hand sides of the congruences (denoted by $r_i(z)$,
$i=1,2,\dots,N$, in Lemma~\ref{lem:2x2diff}) there appear the unknown
Laurent polynomials $a(z),b(z),c(z),d(z)$. Still, the idea of the
proof of Lemma~\ref{lem:2x2diff} may be applied: the system of
congruences \eqref{eq:abcd} can be solved with respect to the
``variables" $b_0(z),b_1(z),b_2(z),b_3(z),b_4(z),b_5(z)$ by
separating odd and even parts, and thereby converting the original system
\eqref{eq:2x2diff} of congruences into the system \eqref{eq:4x4} of
linear congruences for the odd and even parts of the variables.
We solve this last system over the {\it field\/} of {\it rational\/} functions
over $\Z/2\Z$, where odd and even parts of the ``auxiliary
variables" $a(z),b(z),c(z),d(z)$ ``sit" inside the odd and even parts of
the ``constants" $r_i(z)$. In the end, if odd and even parts of the variables
$b_0(z),b_1(z),b_2(z),b_3(z),b_4(z),b_5(z)$ are put together, then 
we are able to express these variables in the form
\begin{equation} \label{eq:bi(z)} 
b_i(z)=
\frac {q_i(z)+H_i\big(z,a^{(o)}(z),a^{(e)}(z),\dots,d^{(o)}(z),d^{(e)}(z)\big)} 
{P(z)}\quad \text {modulo }2,\quad \
i=0,1,\dots,5,
\end{equation}
where the $q_i(z)$'s are (known) Laurent polynomials in $z$ over the
integers, $P(z)$ is a (known) polynomial in $z$ over the integers,
and the 
$$H_i\big(z,a^{(o)}(z),
a^{(e)}(z),\dots,d^{(o)}(z),d^{(e)}(z)\big)\text {'s}$$
are linear forms in $a^{(o)}(z),a^{(e)}(z),\dots,d^{(o)}(z),d^{(e)}(z)$
with coefficients that are (known) Laurent polynomials in $z$ 
over the integers.
%Equivalently, the task is to solve the system
%\begin{multline} \label{eq:bi(z)}
%P(z)b_i(z)-
%H_i\big(z,a^{(o)}(z),a^{(e)}(z),\dots,d^{(o)}(z),d^{(e)}(z)\big)=
%q_i(z)\\
%\quad \text {modulo }2,\quad \
%i=0,1,\dots,5,
%\end{multline}
%in the unknowns $b_0(z),\dots,b_5(z),
%a^{(o)}(z),a^{(e)}(z),\dots,d^{(o)}(z),d^{(e)}(z)$.
%It is classical how to solve such a system, see for instance
%\cite[pp.~422--424]{AschAA}. (In particular, our system is of the
%form \cite[Eq.~(S$_b$)]{AschAA}.)

The task now is to choose
$a^{(o)}(z),a^{(e)}(z),\dots,d^{(o)}(z),d^{(e)}(z)$ in such a way
that in each of the fractions on the right-hand sides of
\eqref{eq:bi(z)} the denominator $P(z)$ cancels out.

In order to carry out this task, we decompose $P(z)$ into its prime
factors (over $\Z/2\Z$), say
$$
P(z)=\prod _{j=1} ^{\ell}P_j^{m_j}(z)
\quad  \text {modulo }2.
$$
Using a standard inductive procedure,\footnote{One first solves
\eqref{eq:Pjmj} modulo $P_j(z)$ (instead of $P_j^{m_j}(z)$); this means
solving a system of linear equations over a field. If one has solved
\eqref{eq:Pjmj} already modulo $P_j^h(z)$, for each variable
$\text{var}(z)$ one makes the Ansatz $\text{var}(z)=\text{var}_0(z)+
\text{var}_1(z)P_j^h(z)$, where $\text{var}_0(z)$ is the value of
$\text{var}(z)$ in the solution modulo $P_j^h(z)$. If this is
substituted in \eqref{eq:Pjmj}, after cancellation,
solving \eqref{eq:Pjmj} modulo
$P_j^{h+1}(z)$ boils again down to solving a system of linear
equations modulo $P_j(z)$, that is, over a field.}
we find $a^{(o)}(z),a^{(e)}(z),\dots,d^{(o)}(z),d^{(e)}(z)$ 
(if there are) such that
\begin{multline} \label{eq:Pjmj}
q_i(z)+H_i\big(z,a^{(o)}(z),a^{(e)}(z),\dots,d^{(o)}(z),d^{(e)}(z)\big)
=0\ \left(\text {mod }P_j^{m_j}(z)\right),\\
\quad i=0,1,\dots,5,\ j=1,2,\dots,\ell,
\end{multline}
(again, over the field $\Z/2\Z$), and then put the particular results
for each $j$ together by means of the Chinese remainder theorem.
We only discuss the generic case here, the discussion for other cases
being completely analogous. Namely, generically,
having to solve $6$ equations in $8$ variables, 
one will be able to express six of the variables in terms of
two ``free" variables. Let us say,
$b^{(o)}(z),b^{(e)}(z),c^{(o)}(z),c^{(e)}(z),d^{(o)}(z),d^{(e)}(z)$ 
can be expressed in terms of $a^{(o)}(z),a^{(e)}(z)$,
\begin{alignat}2 
\notag
b^{(o)}(z)&=s_1(z)+u_1(z)a^{(o)}(z)+v_1(z)a^{(e)}(z)
\quad &\text {modulo }2,\\
\notag
b^{(e)}(z)&=s_2(z)+u_2(z)a^{(o)}(z)+v_2(z)a^{(e)}(z),
\quad &\text {modulo }2,\\
\notag
c^{(o)}(z)&=s_3(z)+u_3(z)a^{(o)}(z)+v_3(z)a^{(e)}(z),
\quad &\text {modulo }2,\\
\notag
c^{(e)}(z)&=s_4(z)+u_4(z)a^{(o)}(z)+v_4(z)a^{(e)}(z),
\quad &\text {modulo }2,\\
\notag
d^{(o)}(z)&=s_5(z)+u_5(z)a^{(o)}(z)+v_5(z)a^{(e)}(z),
\quad &\text {modulo }2,\\
d^{(e)}(z)&=s_6(z)+u_6(z)a^{(o)}(z)+v_6(z)a^{(e)}(z)
\quad &\text {modulo }2,
\label{eq:aoae}
\end{alignat}
where the $s_i(z)$'s, the $u_i(z)$'s, and the $v_i(z)$'s are certain
(known) Laurent polynomials in $z$ over the integers, and where we
are free to choose $a^{(o)}(z)$ and $a^{(e)}(z)$.
If this is substituted in \eqref{eq:bi(z)}, then on the right-hand
sides the denominator $P(z)$ cancels out, and
$b_0(z),b_1(z),b_2(z),b_3(z),b_4(z),b_5(z)$ will all be equal to
Laurent polynomials in $z$ over the integers.

We are still not finished, though. 
In the ``solution" \eqref{eq:aoae} the Laurent polynomials
$a^{(o)}(z),
b^{(o)}(z),
c^{(o)}(z),
d^{(o)}(z)$ must be chosen as odd, while the Laurent polynomials
$a^{(e)}(z),
b^{(e)}(z),
c^{(e)}(z),
d^{(e)}(z)$ must be chosen as even.
In order to achieve this, we must (again) separate odd and even
parts: doing so in \eqref{eq:aoae} yields the system 
\begin{alignat}2
\notag
0&=s_1^{(e)}(z)+u_1^{(o)}(z)a^{(o)}(z)+v_1^{(e)}(z)a^{(e)}(z)
\quad &\text {modulo }2,\\
\notag
0&=s_2^{(o)}(z)+u_2^{(e)}(z)a^{(o)}(z)+v_2^{(o)}(z)a^{(e)}(z),
\quad &\text {modulo }2,\\
\notag
0&=s_3^{(e)}(z)+u_3^{(o)}(z)a^{(o)}(z)+v_3{(e)}(z)a^{(e)}(z),
\quad &\text {modulo }2,\\
\notag
0&=s_4^{(o)}(z)+u_4^{(e)}(z)a^{(o)}(z)+v_4^{(o)}(z)a^{(e)}(z),
\quad &\text {modulo }2,\\
\notag
0&=s_5^{(e)}(z)+u_5^{(o)}(z)a^{(o)}(z)+v_5^{(e)}(z)a^{(e)}(z),
\quad &\text {modulo }2,\\
0&=s_6^{(o)}(z)+u_6^{(e)}(z)a^{(o)}(z)+v_6^{(o)}(z)a^{(e)}(z)
\quad &\text {modulo }2.
\label{eq:aoae2}
\end{alignat}
This is a system of six linear congruences with two variables,
$a^{(o)}(z)$ and $a^{(e)}(z)$, where the first of these should be an
odd Laurent polynomial and the second an even one.
It is of the type of the system of congruences
\eqref{eq:4x4}. How to solve such a system is explained in
the paragraph below the proof of Lemma~\ref{lem:2x2diff}.
(One first solves over the field of rational functions in $z$ over
$\Z/2\Z$, and then cancels denominators, if possible.)
Moreover, the argument in the proof of
Lemma~\ref{lem:2x2diff} showing that, if \eqref{eq:4x4} has {\it some}
solution in Laurent polynomials, then it also has a solution 
in which all $f_j^{(1)}(z)$'s are even Laurent polynomials and all 
$f_j^{(2)}(z)$'s are odd Laurent polynomials, also applies to the
system \eqref{eq:aoae2} to guarantee that, if one is able to find
{\it some} solution $a^{(o)}(z),a^{(e)}(z)$, then one can also find
one in which $a^{(o)}(z)$ is an odd Laurent polynomial and $a^{(e)}(z)$
is an even Laurent polynomial.

If one is able to carry through this procedure, then one has found
the unknowns $b_{i,4}(z)$, $i=0,1,\dots,5$, so that \eqref{eq:bi4}
produces the
desired description modulo~$16$ of the solution $F(z)$ to the Riccati-type
differential equation \eqref{eq:diffeq}. Conversely, if one of the
systems of linear congruences which one has to solve along the way
(these are \eqref{eq:abcd}, \eqref{eq:Pjmj}, and \eqref{eq:aoae2})
has no solution, then one has {\it proved\/} that it is impossible
to describe the series $F(z)$ modulo~$16$ in terms of a polynomial in
$\Phi(z)$ with coefficients that are Laurent polynomials in $z$ over
the integers.

\end{document}